\documentclass[reqno]{amsart}
\usepackage{amssymb}
\usepackage[utf8]{inputenc}
\usepackage{amsmath,amsthm,amstext,amssymb,bbm}
\usepackage{amscd,  amsfonts,  amsbsy,  mathrsfs, upgreek, mathtools, stmaryrd, bbm}
\usepackage{todonotes} 
\usepackage{pdfsync}
\usepackage{hyperref}
\usepackage{tikz,tikz-cd}

\usepackage{geometry}
\usepackage[all]{xy}

\usepackage{mathtools}
\usepackage{bm}

\newtheorem{theorem}{Theorem}[section]
\newtheorem{proposition}[theorem]{Proposition}
\newtheorem{conjecture}[theorem]{Conjecture}
\newtheorem{definition}[theorem]{Definition}
\newtheorem{lemma}[theorem]{Lemma}
\newtheorem{conclusion}[theorem]{Conclusion}
\newtheorem{corollary}[theorem]{Corollary}

\newtheorem*{claim*}{Claim}
\newtheorem*{subclaim*}{Subclaim}

\newcommand{\Ce}{{\mathcal{C}}}

\newcommand{\PPP}{{\mathbb{P}}}
\newcommand{\QQQ}{{\mathbb{Q}}}
\newcommand{\RRR}{{\mathbb{R}}}

\newcommand{\ESR}{\mathrm{ESR}}

\newcommand{\seq}[2]{\langle{#1}~\vert~{#2}\rangle}
\newcommand{\map}[3]{{#1}:{#2}\longrightarrow{#3}}
\newcommand{\pmap}[4]{{#1}:{#2}\xrightarrow{#4}{#3}}

\newcommand{\dom}[1]{{{\rm{dom}}(#1)}}
\newcommand{\ran}[1]{{{\rm{ran}}(#1)}}
\newcommand{\id}{{\rm{id}}}
\newcommand{\tc}[1]{{\rm{tc}}({#1})}
\newcommand{\crit}[1]{{{\rm{crit}}\left({#1}\right)}}
\newcommand{\calL}{\mathcal{L}}
\newcommand{\calU}{\mathcal{U}}
\newcommand{\rank}[1]{{\rm{rnk}}({#1})}
\newcommand{\cof}[1]{{{\rm{cof}}(#1)}}
\newcommand{\Set}[2]{\{{#1}~ \vert~{#2}\}}
\newcommand{\Add}[2]{\mathrm{Add}({#1},{#2})}
\newcommand{\Ult}[2]{\mathrm{Ult}({#1},{#2})}

\newcommand{\HOD}{{\rm{HOD}}}
\newcommand{\OD}{{\rm{OD}}}

\newcommand{\On}{{\rm{Ord}}}

\newcommand{\betrag}[1]{\vert{#1}\vert}

\renewcommand{\emptyset}{\varnothing}

\newtheorem{theoremletter}{Theorem}

\title[Large cardinals, structural reflection and the HOD Conjecture]{Large cardinals, structural reflection, and\\ the HOD Conjecture}

\author[Aguilera]{Juan P. Aguilera}
\address{Institut f\"ur diskrete Mathematik und Geometrie, TU Wien. Wiedner Hauptstrasse 8-10, 1040 Vienna, Austria.}
\email{aguilera@logic.at}

\author[Bagaria]{Joan Bagaria}
\address{ICREA (Instituci\'o Catalana de Recerca i Estudis Avan\c{c}ats) and
\newline \indent Departament de Matem\`atiques i Inform\`atica, Universitat de Barcelona. 
Gran Via de les Corts Catalanes, 585,
08007 Barcelona, Catalonia.}
\email{joan.bagaria@icrea.cat}

\author[L\"ucke]{Philipp L\"{u}cke}
\address{Fachbereich Mathematik, Universit\"at Hamburg, Bundesstra{\ss}e 55, Hamburg, 20146, Germany}
\email{philipp.luecke@uni-hamburg.de}

\makeatletter
\@namedef{subjclassname@2020}{\textup{2020} Mathematics Subject Classification}
\makeatother

\thanks{
The work of the first author was partially supported by FWF grants ESP-3N and STA-139. 
The work of the second author was supported by the Generalitat de Catalunya (Catalan Government) under
grant 2021 SGR 00348, and by the Spanish Government under grants MTM-PID2020-116773GBI00, PID2023-147428NB-I00, and Europa Excelencia grant EUR2022-134032.
The  third author gratefully acknowledges support from the Deutsche Forschungsgemeinschaft (Project number 522490605).}

\date{\today }

\begin{document}

\begin{abstract}
We introduce \textit{exacting cardinals} and a strengthening of these, \textit{ultraexacting cardinals}. These are natural large cardinals defined equivalently as weak forms of rank-Berkeley cardinals, strong forms of J\'onsson cardinals, or in terms of principles of structural reflection. However, they challenge commonly held intuition on strong axioms of infinity.

We prove that ultraexacting cardinals are consistent with Zermelo-Fraenkel Set Theory with the Axiom of Choice ($\mathsf{ZFC}$) relative to the existence of an I0 embedding. However, the existence of an ultraexacting cardinal below a measurable cardinal implies the consistency of $\mathsf{ZFC}$ with a proper class of I0 embeddings, thus challenging the linear--incremental picture of the large cardinal hierarchy.

We show that the existence of an exacting cardinal implies that $V$ is not equal to $\HOD$ (G\"odel's universe of Hereditarily Ordinal Definable sets), showing that these cardinals surpass the current  hierarchy of large cardinals consistent with $\mathsf{ZFC}$. 
Moreover, we prove that the existence of an exacting cardinal above an extendible cardinal implies the "\emph{$V$ is  far from $\HOD$}" alternative of Woodin's $\HOD$ Dichotomy. 
In particular, it follows that the consistency of $\mathsf{ZFC}$ with an exacting cardinal above an extendible cardinal would refute Woodin's $\HOD$ Conjecture and Ultimate-L Conjecture. 
Finally, we show that the consistency of $\mathsf{ZF}$ with certain large cardinals beyond choice implies the consistency of $\mathsf{ZFC}$ with the existence of an exacting cardinal above an extendible cardinal. 
\end{abstract}
\maketitle
\setcounter{tocdepth}{1}
\tableofcontents

%%%%%%%%%%%%%%%%%%%%
%%%%%%%%%%%%%%%%%%%%
\section{Introduction}
\subsection{Background}
\textit{Scott's inconsistency theorem} \cite{Scott:1961} asserts that the theory $\mathsf{ZFC} + V=L$ (i.e., Zermelo-Fraenkel set theory plus the axiom of Choice, together with the Axiom of Constructibility, which asserts that the set-theoretic universe $V$ equals G\"odel's constructible universe, $L$) is inconsistent with the existence of a non-trivial elementary\footnote{I.e., truth-preserving.} embedding from $V$ to a transitive class\footnote{Recall that a class or a set is said to be \emph{transitive} if it contains all elements of its elements.} $M$, a kind of embedding that can be obtained from the existence of a \emph{measurable cardinal}. It was then shown by Keisler that, conversely,  the \emph{critical point} of such an  embedding, i.e., the least ordinal that is moved by it, is a measurable cardinal (see \cite[5.5 \& 5.6]{Kan:THI}). These \emph{large cardinals} were introduced by Ulam \cite{Ul30} and are defined as cardinals $\kappa$ carrying a total $\kappa$-complete, uniform, two-valued measure.  The Scott-Kiesler proofs  show in fact  that $\mathsf{ZFC} + V=L$ is inconsistent with the existence  of a non-trivial elementary embedding from a rank-initial segment $V_{\kappa+1}$ of the universe $V$ into a transitive set $M$. Here, recall the definition of the von Neumann hierarchy of sets given inductively by $V_0=\emptyset$, $V_{\alpha+1} = \mathcal{P}(V_\alpha)$, and taking unions at limit stages.

There is a rich world of large cardinals defined in terms of the existence of elementary embeddings 
$\map{j}{V_\alpha}{N}$, where $V_\alpha$ is a rank-initial segment of $V$ and $N$ is a transitive set, with various closure properties. While such embeddings are incompatible with the Axiom of Constructibility, $L$ can glimpse a modest outline of these in the form of so-called \textit{weak} large cardinals, in which one substitutes the domain $V_\alpha$  of the elementary embedding by an elementary substructure of the same size as the critical point of the embedding. Effectively one avoids Scott's inconsistency by restricting \textit{partial} measures on the critical point $\kappa$. There are many examples of these, including weakly compact cardinals (a weak form of measurable cardinals introduced by Erd\H{o}s and Tarski \cite{ErTa61}), strongly unfoldable cardinals (introduced by Villaveces \cite{Vi98} as a weak form of strong  cardinals), subtle cardinals (a weak form of Vop\u{e}nka cardinals), as well as weakly extendible cardinals (recently introduced by Fuchino and Sakai \cite{FuSa22}), and  $C^{(n)}$-strongly unfoldable cardinals, introduced by the second and third authors in \cite{BL2} as a weak version of $C^{(n)}$-extendible cardinals \cite{Ba:CC}  which generalizes both Villaveces' notion of strongly unfoldable cardinals and Rathjen's notion of shrewd cardinals \cite{Ra:Shrewd}.

Reaching upon the highest  end of the large-cardinal spectrum, we have \textit{Kunen's inconsistency theorem} \cite{ku:EE}, which asserts that the existence of a non-trivial elementary embedding $\map{j}{V}{V}$ is inconsistent with $\mathsf{ZFC}$. The existence of such an embedding had been proposed by Reinhardt in his 1968 Berkeley PhD thesis as the ultimate large-cardinal axiom. With the Axiom of Choice playing an essential part in the proof, Kunen's inconsistency theorem  shows, in fact, that the existence of a non-trivial elementary embedding $\map{j}{V_{\lambda+2}}{V_{\lambda+2}}$ is inconsistent with $\mathsf{ZFC}$. Thus, we have two dividing lines in the hierarchy of large cardinals: one separating those consistent with the Axiom of Constructibility from those which are not, and one separating those consistent with the Axiom of Choice with those which are not. 

Just like the anti-constructibility cardinals project their shadows onto $L$, one finds vestiges of Reinhardt's dream in the world of $\mathsf{ZFC}$. Examples of these are the well known large-cardinal principles ${\rm I}3$, ${\rm I}2$, ${\rm I}1$, and ${\rm I}0$ (see \cite{Kan:THI}), the last of which was introduced by Woodin  in 1984 in order to prove the consistency of the Axiom of Determinacy.\footnote{Woodin later reduced the large-cardinal hypothesis for this theorem, which appeared in print as \cite{Wo88}.}
Laver \cite{Lav92} uncovered a close connection between cardinals of this sort and finite left-distributive algebras, and this fact later led to deep connections between large cardinals, finite left-distributive algebras, arithmetic (see, e.g.,\cite{DoJe97}), knot theory (see, e.g., \cite{BCM20} and \cite{BrMi20}), and braid theory (see, e.g., \cite{Deh94}).

In this work, we introduce new large cardinals, which we call \emph{exacting} and \emph{ultraexacting}, and show that they are consistent with $\mathsf{ZFC}$ relative to I0.\footnote{See \S\ref{remark1} for commentary and historical remarks, including connections with earlier work of Woodin.} 
These new cardinals are obtained as direct analogues of the \textit{weak} large cardinals compatible with $V = L$. 
Starting from cardinals incompatible with $\mathsf{ZFC}$ (specifically, with \textit{rank-Berkeley} cardinals in the sense of \cite{GoSch24}), we weaken the definition by shrinking the domains of the elementary embeddings, from some $V_\alpha$ to an elementary substructure of $V_\alpha$ containing the supremum $\lambda$ of the critical sequence\footnote{Recall that the \emph{critical sequence} of a non-trivial  elementary embedding $j$ between set-theoretic structures is the sequence $\seq{\lambda_m}{m<\omega}$, where $\lambda_0$ is the critical point of $j$ and $\lambda_{m+1}=j(\lambda_m)$.} as well as $V_\lambda$ as a subset (see Definition \ref{DefinitionWRB}). Thus, we obtain the following analogy:
\[\frac{\text{Exacting}}{\text{Rank-Berkeley}} ~ = ~  \frac{\text{Weakly compact}}{\text{Measurable}} ~ = ~  \frac{\text{Strongly unfoldable}}{\text{Supercompact}}.\]
In the case of \textit{ultraexacting} cardinals, we also demand that the restriction of the embedding to $V_\lambda$ belong to its domain.

We prove that, surprisingly, exacting cardinals imply  $V$ is not equal to $\HOD$ (G\"odel's universe of Hereditarily Ordinal Definable sets), and so they establish a third dividing line in the hierarchy of large cardinals, namely between those that are compatible with $V=\HOD$, which include all traditional large cardinals,  and those that are not. Thus they challenge the commonly held intuition about strong axioms of infinity that, as with all such axioms considered so far, they should be compatible with $V=\HOD$. 
Moreover, we show that these large cardinal axioms challenge the linear--incremental picture of the large cardinal landscape in the sense that they interact in rather extreme ways with other large cardinal axioms, in terms of consistency strength. For instance, the existence of an ultraexacting cardinal below a measurable cardinal implies the consistency of $\mathsf{ZFC}$ with a proper class of I0 cardinals.

\subsection{The HOD Conjecture}\label{SectIntroHODC}
The celebrated HOD Dichotomy theorem of Woodin asserts that, assuming the existence of an extendible cardinal, the set-theoretic universe  is either ``very close'' to HOD, or ``very far'' from it. Namely,

\begin{theorem}[The HOD Dichotomy. Woodin \cite{WOEM1}]\label{HODdichotomy} 
Suppose there exists an extendible cardinal $\kappa$. Then exactly one of the following holds:
 \begin{enumerate}
  \item Every singular cardinal $\lambda >\kappa$ is singular in HOD, and HOD computes  $\lambda^+$ correctly (i.e., $V$ is {\emph{close}} to HOD).

  \item Every regular cardinal greater than $\kappa$ is $\omega$-strongly measurable  in $\HOD$\footnote{See Definition \ref{defstronglymeasurable} below.} (i.e., $V$ is  {\emph{far}} from HOD).
 \end{enumerate}
\end{theorem}

The \emph{$\HOD$ Hypothesis} (see \cite{WOEM1} and also \cite[Section 7.1]{MR4022642}) is the statement  that there exists a proper class of regular cardinals that are not $\omega$-strongly measurable in $\HOD$. 
Theorem \ref{HODdichotomy} shows that the existence of an extendible cardinal  implies that the $\HOD$ Hypothesis is equivalent 
to the statement that every singular cardinal above an extendible cardinal is singular in $\HOD$ and $\HOD$ computes its successor correctly, whereas 
a failure of the $\HOD$ Hypothesis is equivalent to the statement that every regular cardinal greater than  an extendible cardinal is $\omega$-strongly measurable in $\HOD$. 
Results of Goldberg in \cite{MR4693981} show that a  similar dichotomy already  holds above a strongly compact cardinal.

The $\HOD$ Dichotomy  placed set theory at a crossroads, for  there are only two alternatives, each yielding a radically different  universe of sets.     However, in contrast with Jensen's  $L$-Dichotomy, which asserts that either $V$ is very close to $L$ or very far from it, and is easily decided granted large cardinals, the $\HOD$ Dichotomy has a completely different character, for the existence of all large cardinals (compatible with ZFC) that have been studied until now is not known to be inconsistent with either alternative.

Woodin's strategy for resolving the HOD Dichotomy has been to prove  the existence of a  \emph{weak extender model for a supercompact cardinal} (see \cite{WOEM1} and \cite{MR3632568}), which yields the first alternative of the dichotomy:

\begin{theorem}[Woodin \cite{WOEM1}]
Suppose that $\kappa$ is an extendible cardinal. Then the following are equivalent:
 \begin{enumerate}
  \item The $\HOD$ Hypothesis.

  \item There is a weak extender model $N$ for the supercompactness of $\kappa$ which is contained in $\HOD$.
 \end{enumerate}
\end{theorem}

Since the expectation was that a canonical $L$-like inner model for a supercompact cardinal should exist, and it should be contained in $\HOD$,  the theorem  prompted the  following conjecture implying that the first alternative of the HOD Dichotomy is the true one:

\begin{definition} [Woodin \cite{WOEM1}]
The \emph{$\HOD$ Conjecture} asserts that the theory
$$\mbox{ZFC + ``There exists an extendible cardinal''}$$
proves the $\HOD$ Hypothesis.\footnote{This is weaker than the assertion that all sufficiently large regular cardinals are not measurable in HOD.}    
\end{definition}

The \emph{Weak $\HOD$ Conjecture} is defined similarly, but with the stronger theory $$\mbox{ZFC + ``There exists an extendible cardinal and a huge cardinal above it''}.$$ Many contemporary questions in set theory center on topics that are linked to the $\HOD$ Conjecture. For instance, it is shown in \cite{MR4022642} that the ``weak'' form of Woodin's well known \textit{Ultimate-L Conjecture} has the Weak HOD Conjecture as consequence.

Part of the appeal of the HOD Conjecture is that it is equivalent to a number-theoretic question of very low complexity, namely $\Sigma^0_1$, and it is commonly known that all true $\Sigma^0_1$ sentences are provable. Thus the HOD Conjecture must be provable if true. This is a very appealing aspect of the conjecture if the expectation is that it be true. The downside of this situation is that the negation of the HOD Conjecture is  literally a consistency assertion. Thus if it is false, one faces the problem that it can only be refuted from hypotheses stronger than the HOD Conjecture itself in terms of consistency strength, by G\"odel's incompleteness theorems. This is problematic as one must first convince onself that the hypotheses employed are consistent. Plausibly, the HOD Conjecture fails, with a proof of this only possible from incredibly strong assumptions. 
Candidates for such arguments have been found, but, in all cases, there is no consensus on the consistency of the assumptions used. Results of Woodin show that the consistency of $\mathsf{ZF}$ with large cardinals beyond the Axiom of Choice imply a failure of the HOD Conjecture (see \cite{MR4022642} and \cite{WOEM1}). Moreover, in \cite{MR4693981}, Goldberg proved that a failure of the HOD Conjecture follows from the consistency of $\mathsf{ZFC}$ together with a strongly compact cardinal and a non-trivial embedding from HOD into itself. 
Important examples of strong partial failures of the $\HOD$ Hypothesis obtained from assumptions that are generally considered to be consistent can be found in \cite{doi:10.1142/S0219061324500181}, \cite{Ben-Neria_Hayut_2023}, \cite{blue2024ad}, \cite{AlejandroGabeCompactnessHOD} and \cite{FarmerLambda2}.

Let us consider a  scenario in which large cardinal axioms cause the  $\HOD$ Conjecture to fail. In this scenario, we (i) exhibit instances of large-cardinal axioms $A_1, A_2, \hdots, A_n$, (ii) argue that each axiom $A_i$ is consistent relative to a commonly accepted system such as $\mathsf{ZFC} +{{\rm I}0}$, and (iii) prove that the joint consistency of the axioms $A_1, A_2, \hdots, A_n$ with $\mathsf{ZFC}$ refutes the HOD Conjecture. Why multiple axioms $A_i$? The current picture of the large cardinal landscape is that these are generated in a linearly ordered and incremental manner, e.g., weakly compact cardinals are ``much stronger'' than any axiom expressible in terms of inaccessibility; measurable cardinals are ``much stronger'' than any axiom expressible in terms of weak compactness, etc. If this picture were to continue, one would expect that the axioms $A_1, A_2, \hdots, A_n$ in the above scenario can be substituted by a single one, but it is not clear that it should. 
Nevertheless, we shall exhibit instances of large cardinal axioms $A_1$ and $A_2$, each of which has consistency strength strictly weaker than $\mathsf{ZFC} +{{\rm I}0}$, and whose joint consistency  disproves the HOD Conjecture. 
Moreover, we will use techniques developed by Woodin in his disproof of the HOD Conjecture from the 
consistency of $\mathsf{ZF}$ with large cardinals beyond the Axiom of Choice to show that this consistency assumption also implies the joint consistency of $A_1$ and $A_2$ with $\mathsf{ZFC}$.

%The difference between exacting and ultraexacting cardinals is whether one demands that the embeddings satisfy a version of the Hauser embedding property: that $j\upharpoonright V_\lambda$ belong to the domain of the embedding. While weakly compact and strongly unfoldable embeddings usually can be asked to satisfy the (usual) Hauser embedding property, this makes a difference in the context of exact and ultraexact cardinals.

%{corollary:LeastI0notUltraexacting}
\subsection{Summary of results}

We now state and summarize the main results of this article. Exacting  cardinals are the subject of \S\ref{SectExact}. We define exacting cardinals and argue that they are a natural large cardinal notion. First, we give an alternative characterization of these cardinals as a weak form of rank-Berkeley cardinals and therefore also of Reinhardt cardinals. Second, we show that this large cardinal property is also equivalent to a strong form of J\'onssonness that also reflects strong external properties to proper submodels of the same cardinality. We then continue by proving that exacting cardinals  are compatible with the Axiom of Choice:

\begin{theoremletter}\label{TheoremIntroI0Exact}
Suppose that I0 holds. Then, there is a set model $M$ of $\mathsf{ZFC}$ with an exacting cardinal. 
%then there is a transitive model $M$ of ZFC such that $\kappa\in M$ and $\kappa$ is $n$-exact in $M$ for every natural number $n$. 
 %The theory $\mathsf{ZFC}$ + I0 proves the consistency of $\mathsf{ZFC}$ + $\{``$there exists an $n$-exact cardinal'': $n \in\mathbb{N}\}$.
\end{theoremletter}

In particular, this answers \cite[Question 10.3]{BL}.
In  \S\ref{section6}, we  explore the consequences of exacting cardinals in the context of ordinal definability, and we prove the following:

\begin{theoremletter}\label{TheoremIntroHODC}
Suppose that the theory $\mathsf{ZFC}$ is consistent with the assumption that there is an extendible cardinal below an exacting cardinal.  
Then, the Weak HOD Conjecture and the Weak Ultimate-L Conjecture are false.
\end{theoremletter}

Thus, we have two natural notions of large cardinals, namely exacting and extendible, each consistent with $\mathsf{ZFC}$, and which together can be used to refute the HOD Conjecture. 
Theorem \ref{TheoremIntroHODC} and the fact that exacting cardinals imply $V \neq \HOD$ unveils a third -- previously unobserved -- dividing line in the large-cardinal hierarchy and illustrates the following phenomenon: as one considers stronger and stronger large-cardinal axioms, one negates weaker principles of ``regularity'' for the universe of sets -- first $V = L$, then $V = \HOD$, and finally the Axiom of Choice. However, unlike the first, the two latter dividing lines are not entirely determined by the consistency strength of the axioms, as seen from Theorem \ref{TheoremIntroI0Exact}, and by earlier work of Schlutzenberg \cite{FarmerLambda2}.

Motivated by our consistency proof for exacting cardinals, we introduce the notion of \textit{ultraexacting} cardinals. These are obtained from exacting cardinals by enlarging the domain of the elementary embedding in the simplest non-trivial way, namely by demanding that it contains the restriction of the embedding to $V_{\lambda}$, where $\lambda$ is the supremum of the critical sequence.
We show in \S\ref{Ultraexact} that ultraexacting cardinals are not stronger, consistency-wise, than the axiom I0:

\begin{theoremletter}\label{TheoremIntroConUE}
Suppose that $\mathsf{ZFC}$ + I0 is consistent. Then, $\mathsf{ZFC}$ + ``there is an ultraexacting cardinal'' is consistent. 
\end{theoremletter}

Nonetheless, ultraexacting cardinals challenge the linear--incremental picture of the large cardinal hierarchy by interacting in a very nontrivial way with other large cardinals.

\begin{theoremletter}\label{Thmxproperclass}
The theory $\mathsf{ZFC}$ + ``there is an ultraexacting cardinal $\lambda$ and $V_{\lambda+1}^\sharp$ exists'' proves the consistency of 
$\mathsf{ZFC}$ + ``there is a proper class of I0 embeddings.'' 
\end{theoremletter}

According to Theorem \ref{Thmxproperclass}, usual large cardinals do not add strength incrementally the typical way when in the presence of ultraexacting cardinals, and indeed these ``amplify'' the usual large cardinals into a natural extension of Woodin's Icarus hierarchy from \cite{MR2914848} (see \cite{MR3855762}). 
 The proof of  Theorem \ref{Thmxproperclass} relies on techniques developed in \S\ref{Ultraexact} that allow us to extend  embeddings witnessing the  ultraexactness of a cardinal to sets outside of the domain of the original embedding. These techniques can also be used to show that ultraexact cardinal interact with ordinal definability in a highly interesting way. For example, they allow us to show that the successors of ultraexacting cardinals are $\omega$-strongly measurable  in $\HOD$.\footnote{The first examples of successors of singular cardinals with this property were found by Woodin in \cite{WOEM1}} 
 Moreover, they can be used to show ultraexactingness implies non-trivial fragments of $\HOD$-Berkeleyness, a principle introduced in \cite{MR4022642} that implies that $V$ is very far from $\HOD$. 
 In addition, we use the developed theory to prove that ultraexacting  cardinals possess strong infinitary partition properties for ordinal definable colourings that were already studied by Goldberg in the context of the  $\HOD$ Conjecture in \cite{MR4693981}.

In  \S\ref{SectSSR} and \S\ref{SectChoiceless}, we explore alternative characterizations of ultraexacting cardinals and also provide evidence for the naturalness of this concept. In \S\ref{SectSSR}, we show that ultraexacting cardinals are equivalent to a natural principle of structural reflection.
The general principle of \textit{Structural Reflection}, for a class of structures $\mathcal{C}$ of the same type, asserts that there exists an ordinal $\alpha$ such that for every structure $A \in \mathcal{C}$ there is $B \in \mathcal{C}\cap V_\alpha$ and a non-trivial elementary embedding $\map{j}{B}{A}$ (see \cite{BagariaRefl}). Different forms of  structural reflection have been shown to yield a reformulation of large-cardinal notions of various kinds (see  \cite{Ba:CC,BCMR,BL, BL2,BV,BW,SRminus}, and the survey \cite{BagariaRefl}), covering all regions of the large-cardinal hierarchy. Based of this evidence, it has been argued in \cite{BagariaRefl} and \cite{BT} that Structural Reflection is a natural principle that underlies and unifies the whole of the large-cardinal hierarchy. We show %(Corollary \ref{coro4.6}) 
that ultraexacting cardinals are reasonable large-cardinal principles by their being equivalent to a form of Structural Reflection  we call  \emph{Square Root Exact Structural Reflection}. % (Definition \ref{DefSRESR}). 
 Further, in \S\ref{SectChoiceless}, we show that, in a choiceless setting, the existence of ultraexacting cardinals is implied by the existence of large cardinals beyond the Axiom of Choice, such as Reinhardt  cardinals, or Berkeley cardinals, and that the latter in turn imply global forms of Square Root Exact Structural Reflection.

 In light of Theorem \ref{TheoremIntroHODC}, it is not clear \emph{a priori} that, over $\mathsf{ZFC}$, the existence of ultraexacting cardinals should be jointly consistent with the existence of extendible cardinals below them. 
But in \S\ref{section6}, we derive the joint consistency of extendible cardinals with ultraexacting cardinals over $\mathsf{ZFC}$ from a relatively mild cardinal in the  hierarchy of \emph{large cardinals beyond the Axiom of Choice}.\footnote{Note that the theory used in the assumption of this theorem is stronger than the theory used in \cite{MR4022642} to derive a failure of the $\HOD$ Dichotomy from the consistency of large cardinals beyond choice. See Remark \ref{remark1} below.}

\begin{theoremletter}\label{TheoremConZF}
The theory $\mathsf{ZF}$ + ``there is a $C^{(3)}$-Reinhardt cardinal and a supercompact cardinal greater than the supremum of the critical sequence'' proves the consistency of the theory $\mathsf{ZFC}$ + ``there is an ultraexacting cardinal which is a limit of extendible cardinals.''
\end{theoremletter}

It seems unlikely that the hypothesis of Theorem \ref{TheoremConZF} can be weakened substantially, and indeed we conjecture that the existence of an ultraexacting cardinal above an extendible cardinal proves the consistency of $\mathsf{ZF}$ with rank-Berkeley cardinals. \\

In summary, exacting and ultraexacting cardinals have simultaneous characterizations in terms of elementary embeddings with sufficiently correct initial segments of the $V$-hierarchy as targets, as the weak forms of rank-Berkeley cardinals, as strong forms of J\'onsson cardinals, and in terms of principles of structural reflection, arguably instituting the notions as legitimate large-cardinal axioms. 
 They are consistent with the Axiom of Choice relative to I0, and also with extendible cardinals below them relative to mild choiceless large-cardinal axioms.
 Moreover, 
 they behave unusually in that their strength increases dramatically in the presence of other large cardinals, and together with extendible cardinals they are powerful enough to disprove the Weak HOD Conjecture and the Weak Ultimate-L Conjecture. \\ 
 
 These results show that the notions of exacting and ultraexacting cardinals  offer  a natural path to refute the HOD Conjecture through strong axioms of infinity. 
 In addition, even if the existence of an exacting cardinal above an extendible cardinal turns out to be inconsistent and the HOD Conjecture cannot be disproved in the given way,  such a conclusion would necessitate a profound shift in our understanding of large cardinals. 
 So, either way, the results of this paper suggest  that we have to change something about the way we think of large infinities.

%%%%%%

\section{Exacting Cardinals}\label{SectExact}
Our notation is standard. For background on set theory and large cardinals, we invite the reader to \cite{Kan:THI}. Although not necessary to read the article,  we refer the reader to \cite{MR3855762} for a comprehensive survey on I0 and to \cite{MR4022642} for large cardinals beyond the Axiom of Choice.

Recall that a cardinal $\kappa$ belongs to the class $C^{(n)}$ if it is \emph{$\Sigma_n$-correct}, {i.e.,} the set $V_\kappa$ is a $\Sigma_n$-elementary substructure of $V$, written $V_\kappa \prec_{\Sigma_n} V$.
The following definition from  \cite{BL} captures the large cardinal notion corresponding to the strongest principles of structural reflection studied in that paper.  Although we will later prove and make use of alternate definitions of exactness, we shall keep that of \cite{BL} as the ``official'' definition.

\begin{definition}[\cite{BL}]
\label{defexact}
  Let $n>0$ be a natural number and let  $\lambda$   be a limit cardinal. 
  \begin{enumerate}
      \item Given a cardinal $\lambda<\eta\in C^{(n)}$, an elementary submodel $X$ of $V_\eta$ with $V_\lambda\cup\{\lambda\}\subseteq X$ and a cardinal $\lambda<\zeta\in C^{(n+1)}$, an elementary embedding $\map{j}{X}{V_\zeta}$ is an \emph{$n$-exact embedding at $\lambda$} if $j(\lambda)=\lambda$ and $j\restriction\lambda\neq\id_\lambda$. 

      \item Given a strictly increasing sequence $\vec{\lambda}=\seq{\lambda_m}{m<\omega}$ of cardinals with supremum $\lambda$, a cardinal $\kappa<\lambda_0$ is \emph{$n$-exact for $\vec{\lambda}$} if for every $A\in V_{\lambda+1}$, there exists 
      an $n$-exact embedding $\map{j}{X}{V_\zeta}$ at $\lambda$ with $A\in\ran{j}$,      $j(\kappa)=\lambda_0$ and $j(\lambda_m)=\lambda_{m+1}$ for all $m<\omega$.  
     If we further require that $j(\crit{j})=\kappa$, then we say that $\kappa$ is \emph{parametrically $n$-exact for $\vec{\lambda}$}.
  \end{enumerate}
\end{definition}
Observe that parametrically exact cardinals are defined as the image of critical points of elementary embeddings, in the style of Magidor's ``parametrical'' characterization of supercompact cardinals (see \cite{Mag}). Notice also that the set $X$ from the definition of an exact embedding cannot be transitive, for otherwise, since $V_{\lambda +2}\in X$, the restriction map $\map{j\restriction V_{\lambda +2}}{V_{\lambda +2}}{V_{\lambda +2}}$ would be an elementary embedding, thus contradicting Kunen's inconsistency theorem. 
 Furthermore, Woodin's proofs of Kunen's result give us more insight into which sets are not contained in the domains of $1$-exact embeddings:

 \begin{proposition}\label{proposition:NotInDomainExactt}
   If   $\map{j}{X}{V_\zeta}$ is a $1$-exact embedding at a cardinal $\lambda$, then $\lambda^+\nsubseteq X$ and $[\lambda]^\omega\nsubseteq X$.  
 \end{proposition}

 \begin{proof}
  First, assume, towards a contradiction, that $\lambda^+\subseteq X$. Since $\zeta\in C^{(2)}$, we know that $\lambda^+<\zeta$ and $\lambda^+$ is definable in $V_\zeta$ from the parameter $\lambda$. This shows that $\lambda^+\in\ran{j}$ and we can now use the correctness properties of $X$ to conclude that $\lambda^+\in X$ with $j(\lambda^+)=\lambda^+$. The same argument now shows that $H_{\lambda^{++}}$ is an element of $X$. Since the existence of a partition of the set $S^{\lambda^+}_\omega=\Set{\gamma<\lambda^+}{\cof{\gamma}=\omega}$ into $\crit{j}$-many stationary sets can be stated in $H_{\lambda^{++}}$ by a formula that only uses the parameter $\lambda$, it follows that there exists such a partition $\vec{S}=\seq{S_\alpha}{\alpha<\crit{j}}$ that is an element of $X$. Elementarity then ensures that $j(\vec{S})=\seq{T_\beta}{\beta<j(\crit{j})}$ is a partition of $S^{\lambda^+}_\omega$ into $j(\crit{j})$-many stationary subsets. Define $C=\Set{\gamma\in S^{\lambda^+}_\omega}{j(\gamma)=\gamma}$.
  
  \begin{claim*}
    $C$ is an $\omega$-closed unbounded subset of $\lambda^+$. 
  \end{claim*}

  \begin{proof}[Proof of the Claim]
    Given $\gamma<\lambda^+$, let $\seq{\gamma_n}{n<\omega}$ denote the unique sequence with $\gamma_0=\gamma$ and $\gamma_{n+1}=j(\gamma_n+1)$ for all $n<\omega$. Set $\gamma_\omega=\sup_{n<\omega}\gamma_n\in S^{\lambda^+}_\omega\subseteq X$. The correctness properties of $X$ then imply that $X$ contains a strictly increasing sequence $\seq{\beta_m}{m<\omega}$ of ordinals that is cofinal in $\gamma_\omega$. Our construction now ensures that $j(\beta_m)<\gamma_\omega$ holds for all $m<\omega$. Since elementarity implies that $\seq{j(\beta_m)}{m<\omega}$ that is cofinal in $j(\gamma_\omega)$, it follows that $\gamma_\omega\in C$.  

    Next, fix an element $\gamma$ of $S^{\lambda^+}_\omega$ that is a limit point of $C$. As above, we know that $X$ contains a strictly increasing sequence $\seq{\beta_m}{m<\omega}$ that is cofinal in $\gamma$. Given $m<\omega$, there exists $\bar{\gamma}\in C$ with $\beta_m<\bar{\gamma}<\gamma$ and this implies that $j(\beta_m)<\bar{\gamma}<\gamma$. Since $\seq{j(\beta_m)}{m<\omega}$  is cofinal in $j(\gamma)$, we can now conclude that $\gamma\in C$. 
  \end{proof}

  The above claim shows that there exists $\gamma\in C\cap T_{\crit{j}}$. Then there is $\alpha<\crit{j}$ with $\gamma\in S_\alpha$ and we can conclude that $\gamma=j(\gamma)\in T_\alpha\cap T_{\crit{j}}=\emptyset$, a contradiction. 

  Now, assume that $[\lambda]^\omega\subseteq X$. Arguing as above, we can show that $X$ contains a surjection from $[\lambda]^\omega$ onto $\lambda^+$. Therefore, our assumption implies that $\lambda^+\subseteq X$, in contradiction to the above conclusion. 
 \end{proof}

The notion of exactness is motivated by its connection with Structural Reflection. Theorem \ref{theorem:ESRcorrespondence} below, quoted from \cite[Corollary 9.10]{BL},  presents this connection and clarifies the role of the parametric form of the definition.

 It is easy to see that, if $\map{j}{X}{V_\zeta}$ is a $1$-exact embedding at some limit cardinal $\lambda$, then $\map{j\restriction V_\lambda}{V_\lambda}{V_\lambda}$ is an I3-embedding, {i.e.},  a non-trivial elementary embedding from $V_\lambda$ to $V_\lambda$. 
 %
% Then the critical sequence of $j$ is cofinal in $\lambda$. 
 %
 From this it follows  that the existence of  a $1$-exact embedding at $\lambda$ implies that $\lambda$ is a limit of $n$-huge cardinals for all $n<\omega$ (see {\cite[p. 332]{Kan:THI}}). 

Our first goal, in \S\ref{SectExactBerkeley} and \S\ref{SectExactJonsson} below, is to further motivate the notion of exact embeddings by giving two alternate characterizations of their existence, one as a weak form of rank-Berkeleyness and another as a strong form of J\'onssonness.

\subsection{Exactness and rank-Berkeley cardinals}\label{SectExactBerkeley}
The following lemma establishes a fact about $n$-exact embeddings which might be surprising: the existence of such embeddings does not depend on the parameter $n$. 
%their existence does not depend on $n$ (for $1 < n$).
%Even though the notion of $n$-exact embedding appears to depend on $n$, it turns out this is not the case. 

\begin{lemma}\label{lemma:Equivalentexact}
Given a natural number $n>0$, the following statements are equivalent for every  limit cardinal $\lambda$ and every set $x$: 
\begin{enumerate}
    \item\label{item:Equivalentexact1} There is an $n$-exact embedding $\map{j}{X}{V_\zeta}$ at $\lambda$ with $x\in X$ and $j(x)=x$. 
    
    %\item There are elements $\eta$ and $\zeta$ of $C^{(2)}$ greater than $\lambda$, an elementary submodel $X$ of $V_\eta$ with $V_\lambda\cup\{\lambda,x\}\subseteq X$, and an elementary embedding $\map{j}{X}{V_\zeta}$ with $j(\lambda)=\lambda$, $j(x)=x$, and $j\restriction\lambda \neq \id_\lambda$. 

    \item\label{item:Equivalentexact2} For every $\zeta>\lambda$ with $x\in V_\zeta$, there is an elementary submodel $X$ of $V_\zeta$ with  $V_\lambda\cup\{\lambda,x\}\subseteq X$, and an elementary embedding $\map{j}{X}{V_\zeta}$ with $j(\lambda)=\lambda$, $j(x)=x$, and $j\restriction \lambda \neq \id_\lambda$. 
        
    \item\label{item:Equivalentexact3} For every $\zeta>\lambda$ with $x\in V_\zeta$ and every $\alpha<\lambda$, there is an elementary submodel $X$ of $V_\zeta$ with  $V_\lambda\cup\{\lambda,x\}\subseteq X$, and an elementary embedding $\map{j}{X}{V_\zeta}$ with $j(\lambda)=\lambda$, $j(x)=x$, $j\restriction\alpha=\id_\alpha$, and $j\restriction \lambda \neq \id_\lambda$. 
\end{enumerate}
\end{lemma}

\begin{proof}
  \ref{item:Equivalentexact1}$\Rightarrow$\ref{item:Equivalentexact3}: Assume, towards a contradiction that \ref{item:Equivalentexact1} holds and \ref{item:Equivalentexact3} fails,  and let $\xi>\lambda$ be  minimal with the property that $x\in V_\xi$ and there is an $\alpha<\lambda$ such that  for every  elementary submodel $X$ of $V_\xi$ with  $V_\lambda\cup\{\lambda,x\}\subseteq X$, there is no elementary embedding $\map{j}{X}{V_\xi}$ with $j(\lambda)=\lambda$, $j(x)=x$, $j\restriction\alpha=\id_\alpha$, and $j\restriction \lambda \neq \id_\lambda$. Then  the set $\{\xi\}$ is definable by a $\Sigma_2$-formula with parameters $\lambda$ and $x$. Now, let $\beta<\lambda$ be minimal such that for every  elementary submodel $X$ of $V_\xi$ with  $V_\lambda\cup\{\lambda,x\}\subseteq X$, there is no elementary embedding $\map{j}{X}{V_\xi}$ with $j(\lambda)=\lambda$, $j(x)=x$, $j\restriction\beta=\id_\beta$, and $j\restriction \lambda \neq \id_\lambda$. Then the set $\{\beta\}$ is also definable by a $\Sigma_2$-formula with parameters $\lambda$ and $x$. 
  Using \ref{item:Equivalentexact1}, we can now find a cardinal $\lambda<\eta\in C^{(n)}$, an elementary submodel $X$ of $V_\eta$ with $V_\lambda\cup\{\lambda,x\}\subseteq X$, a cardinal $\lambda<\zeta\in C^{(n+1)}$ and an elementary embedding $\map{j}{X}{V_\zeta}$ with $j(\lambda)=\lambda$, $j(x)=x$ and $j\restriction\lambda\neq\id_\lambda$. Since $n+1\geq 2$, the correctness properties properties of $V_\zeta$ ensure that the ordinals $\xi$ and $\beta$ are both smaller than $\zeta$, and the sets $\{\xi\}$ and $\{\beta\}$ are definable in $V_\zeta$ by $\Sigma_2$-formulas with parameters $\lambda$ and $x$. Moreover, since $j(\lambda)=\lambda$ and $j(x)=x$, we can use the fact that $\Sigma_2$-formulas are upwards absolute from $X$ to $V$ to conclude that $\xi$ and $\beta$ are both contained in $X$  with $j(\xi)=\xi$ and $j(\beta)=\beta$. Set $Y=X\cap V_\xi$ and $i=j\restriction Y$. Since elementarity implies that $V_\xi\in X$ with $j(V_\xi)=V_\xi$, we then know that $Y$ is an elementary submodel of $V_\xi$ with $V_\lambda\cup\{\lambda,x\}\subseteq Y$. Moreover, it follows that $\map{i}{Y}{V_\xi}$ is an elementary embedding with $i(\lambda)=\lambda$, $i(x)=x$, and $i\restriction\lambda \neq \id_\lambda$. By our assumptions, we now know that $i\restriction\beta\neq\id_\beta$. Using the fact that $i(\beta)=\beta$,  we can now conclude that $\beta>\omega$, $i(V_{\beta+2})=V_{\beta+2}$ and $\map{i\restriction V_{\beta+2}}{V_{\beta+2}}{V_{\beta+2}}$ is a non-trivial elementary embedding, contradicting the Kunen inconsistency.  
  \ref{item:Equivalentexact3}$\Rightarrow$\ref{item:Equivalentexact2}: Trivial.  
 \ref{item:Equivalentexact2}$\Rightarrow$\ref{item:Equivalentexact1}: Pick $\zeta\in C^{(n+1)}$ with $x\in V_\zeta$ and use \ref{item:Equivalentexact2} to obtain an $n$-exact embedding.  
\end{proof}

Note that the above implication \ref{item:Equivalentexact1}$\Rightarrow$\ref{item:Equivalentexact3} does not work when we prescribe the critical sequence of the embedding, because then the least counterexample is only definable by using the sequence as a parameter and it is no longer possible to show that this ordinal is fixed by the embedding, because the parameter is not. We will return to this in \S\ref{Ultraexact}.

In addition, observe that item \ref{item:Equivalentexact2} in the lemma above is $\Pi_2$ expressible, using $\lambda$ and $x$ as parameters. 
Thus, if $\lambda$ is the least ordinal for which \ref{item:Equivalentexact2} holds for a given set $x$, then there is no cardinal in $C^{(3)}$ below $\lambda$ and above the rank of $x$. 
In particular, there is no extendible cardinal below the least $\lambda$ for which there is a $1$-exact embedding at $\lambda$. Moreover, if $V$ is a model of ZFC which satisfies that there is a $1$-exact embedding at $\lambda$, with $\lambda$ being the least such, and there is an inaccessible cardinal $\kappa$ above $\lambda$, then, by \ref{item:Equivalentexact2} of the lemma above, the set $V_\kappa$ is a model of ZFC with a $1$-exact embedding at $\lambda$ and with no regular cardinal in $C^{(3)}$, hence with no extendible cardinals.

%Another interesting consequence of the above lemma is the fact that the existence of exact embeddings is downwards absolute to sufficiently large initial segments of $V$: 

%\begin{corollary}\label{corollary:ExactAbso}
 %   If  an $n$-exact embedding exists at a cardinal $\lambda$ for some $n>1$ and $\kappa$ is an inaccessible cardinal greater than $\lambda$, then there exists an $n$-exact embedding at  $\lambda$ in $V_\kappa$. 
%\end{corollary}

%\begin{proof}
 %  For each $\lambda<\zeta<\kappa$, Lemma \ref{lemma:Equivalentexact}  shows that, in $V$, there is an elementary submodel $X$ of $V_\zeta$ with  $V_\lambda\cup\{\lambda\}\subseteq X$, and an elementary embedding $\map{j}{X}{V_\zeta}$ with $j(\lambda)=\lambda$, and $j\restriction \lambda \neq \id_\lambda$. Then $X$ and $j$ are both contained in $V_\kappa$, and all of the listed properties of these sets are absolute between $V$ and $V_\kappa$. Another application of Lemma \ref{lemma:Equivalentexact}, this time within $V_\kappa$, now shows that there is an $n$-exact embedding  at  $\lambda$ in $V_\kappa$. 
%\end{proof}

In \cite{GoSch24}, working in ZF, Goldberg and Schlutzenberg defined a cardinal $\lambda$ to be \emph{rank-Berkeley} if for all $\zeta>\lambda$ and all $\alpha<\lambda$, there exists a non-trivial elementary embedding $\map{j}{V_\zeta}{V_\zeta}$ with $\alpha<\crit{j}<\lambda$ and $j(\lambda)=\lambda$. The equivalences given by Lemma \ref{lemma:Equivalentexact} now show that the existence of  $1$-exact embeddings corresponds to rank-Berkeleyness in the same way as strong unfoldability corresponds to supercompactness (as shown in \cite{luecke2021strong} and \cite{SRminus}). Only this time, instead of canonically weakening a large cardinal property implying $V\neq L$ to obtain a notion compatible with $V=L$, we weaken a property implying a failure of the Axiom of Choice in the same way to obtain a notion compatible with this axiom. 
%The following definition is motivated by the definitions ``weak'' large cardinals compatible with $V = L$ which one obtains by restricting the domains of elementary embeddings. 
Motivated by these analogies, we introduce a name for the isolated property:

\begin{definition}\label{DefinitionWRB}
A cardinal $\lambda$ is \emph{exacting} if for every $\zeta>\lambda$ there exist $X \prec V_\zeta$ with $V_\lambda \cup \{\lambda\}\subseteq X$ and an elementary embedding $\map{j}{X}{V_\zeta}$ with $j(\lambda)=\lambda$ and $j\restriction\lambda\neq\id_\lambda$. 
\end{definition}

If $\lambda$ is an exacting cardinal, then there is a non-trivial elementary embedding from $V_\lambda$ to $V_\lambda$. Therefore, the Kunen inconsistency implies that all exacting cardinal are elements of $C^{(1)}$ with countable cofinality. In particular, if   $X$ is as in Definition \ref{DefinitionWRB}, then we can assume that $\betrag{X}=\lambda$. Applying Lemma \ref{lemma:Equivalentexact} to $n = 1$ and $x = \varnothing$,  immediately yields the following equivalence:

\begin{corollary}\label{corollary:Exacting1ExactEmbedding}
 A cardinal $\lambda$ is exacting if and only if there is a $1$-exact embedding at $\lambda$.  \qed 
\end{corollary}

\noindent We now turn to model-theoretic characterizations of exacting cardinals.

\subsection{Exactness and J\'onsson cardinals}\label{SectExactJonsson}

Recall that a cardinal $\lambda$ is \emph{J\'onsson} if every structure in a countable first-order language whose domain has cardinality $\lambda$ has a proper elementary substructure of cardinality $\lambda$ (see {\cite[{p.} 93]{Kan:THI}}). 
%Here, structures are assumed to have countable vocabularies in which every relation and function is finitary. 
In the following, we will show that being an exacting cardinal can be regarded as a strong form of J\'onssonness   that ensures the existence of proper elementary substructures of the same cardinality that also possess certain external properties of the original structure. 
 In addition, we will also derive a dual version of this result that demands the existence of proper elementary superstructures of the same cardinality and which has no non-trivial equivalent in the classical case.

\begin{proposition}
 Let $\lambda$ be an exacting cardinal and let  $\Ce$ be a class of structures in a countable first-order language that is definable by a formula with parameters in  $V_{\lambda}\cup\{\lambda\}$. If $B$ is a structure of cardinality $\lambda$ in $\Ce$, then there exist structures $A$ and $C$ of cardinality $\lambda$ in $\Ce$ such that $A$ is isomorphic to a proper elementary substructure of $B$ and $B$ is isomorphic to a proper elementary substructure of $C$. 
\end{proposition}

\begin{proof}
    Fix a natural number $n>0$, a $\Sigma_n$-formula $\varphi(v_0,v_1,v_2)$ and $z\in V_\lambda$ such that $\Ce=\Set{A}{\varphi(A,\lambda,z)}$. Pick $\lambda<\zeta\in C^{(n+3)}$ and use Lemma \ref{lemma:Equivalentexact} to find an elementary submodel $X$ of $V_\zeta$ with  $V_\lambda\cup\{\lambda\}\subseteq X$, and an elementary embedding $\map{j}{X}{V_\zeta}$ with $j(\lambda)=\lambda$, $j(z)=z$, and $j\restriction \lambda \neq \id_\lambda$. 

    First, assume, towards a contradiction, that $\Ce$ contains a structure     of cardinality $\lambda$ that does not contain a proper elementary substructure of cardinality $\lambda$ that is isomorphic to a structure in $\Ce$. The correctness of $X$ for $\Sigma_{n+3}$ sentences then implies that $X$ contains a structure $B$ with this property. 
    %Since $j(z)=z$ and $\zeta\in C^{(n+3)}$, we then have that  $j(B)$ is also an element of $\Ce$ with the property that no proper elementary substructure of $j(B)$ of cardinality $\lambda$ is isomorphic to a structure in $\Ce$. By elementarity, we know that $X$ contains a bijection, $b$, between $\lambda$ and the domain of  $j(B)$. It follows that $j(b)$ is a bijection between $\lambda$ and the domain of $B$. 
    Moreover, since  $\lambda$ is a subset of $X$, it follows that the domain of $B$ is also a subset of $X$ and this implies that $j$ induces an elementary embedding of $B$ into $j(B)$. Let $b\in X$ be a bijection between $\lambda$ and the domain of $B$. Given $\alpha<\lambda$, we have $$j(b(\alpha)) ~ = ~ j(b)(j(\alpha)) ~ \in ~ j(b)[j[\lambda]]$$ and, since $j[\lambda]$ is a proper subset of $\lambda$, we can conclude that the  elementary  embedding of $B$ into $j(B)$ induced by $j$ is not surjective. In particular, the pointwise image of $B$ under this embedding is a proper elementary submodel of $j(B)$ of cardinality $\lambda$ that is isomorphic to the element $B$ of $\Ce$, contradicting the properties of $j(B)$. 

    Now, assume, towards a contradiction, that $\Ce$ contains a structure  of cardinality $\lambda$ that is not isomorphic to a proper elementary substructure of a structure of cardinality $\lambda$ in $\Ce$. As above, we can find a structure $B$ in $X$ with this property. We then know that $j(B)$ is also structure in $\Ce$ and $j$ induces an elementary embedding of $B$ into $j(B)$ that is not surjective, contradicting the properties of $B$.  
\end{proof}

 The next proposition shows that the existence of a $2$-exact embedding follows from each of the above conclusions for a single $\Sigma_3$-definable class of structures:

\begin{proposition}
   There exists a class $\Ce$  of structures in a finite first-order language such that the following statements hold: 
   \begin{enumerate}
       \item The class $\Ce$ is definable by a $\Sigma_3$-formula without parameters. 

       \item A cardinal is contained in the class $C^{(1)}$ if and only if it is the cardinality of a structure in $\Ce$. 

       \item If $\lambda$ is cardinal with the property that  there are structures $A$ and $B$ of cardinality $\lambda$ in $\Ce$ such that $A$ is isomorphic to a proper substructure of $B$, then  $\lambda$ is an exacting cardinal.  
   \end{enumerate}
\end{proposition}

\begin{proof}
  Let $\calL$ denote the first-order language that extends the language of set theory by a constant symbol and a unary function symbol, and let $\Ce$ denote the class of all $\calL$-models $\langle X,E,\lambda,s\rangle$ such that the following statements hold: 
  \begin{itemize}
      \item $\lambda$ is a cardinal in $C^{(1)}$. 

      \item There is a cardinal $\eta$ in $C^{(2)}$ such that $X$ is an elementary substructure of $V_\zeta$ of cardinality $\lambda$ with $V_\lambda\cup\{\lambda\}\subseteq X$. 

      \item $E={{\in}\mathrel{\restriction}(X\times X)}$. 

      \item $\map{s\restriction \lambda}{\lambda}{X}$ is a surjection. 
  \end{itemize}
 Then $\Ce$ is definable by a $\Sigma_3$-formula without parameters. Moreover, it is easy to see that the class $C^{(1)}$ coincides with the class of cardinalities of structures in  $\Ce$. 

 Now, assume that $\lambda$ is cardinal with the property that the class $\Ce$ contains structures $A=\langle X,\in,\lambda,s\rangle$ and $B=\langle Y,\in,\lambda,t\rangle$  such that $A$ is isomorphic to a proper substructure of $B$. Then there exist cardinals $\eta$ and $\zeta$ in $C^{(2)}$ such that $X$ is an elementary submodel of $V_\eta$ with $V_\lambda\cup\{\lambda\}\subseteq X$ and $Y$ is an elementary submodel of $V_\zeta$ with $V_\lambda\cup\{\lambda\}\subseteq Y$. Our assumption now yields an elementary embedding $\map{j}{X}{V_\zeta}$ with $j(\lambda)=\lambda$, $j[X]\subsetneq Y$ and $j(s(\alpha))=t(j(\alpha))$ for all $\alpha<\lambda$. 
 We can then find $\alpha<\lambda$ with $t(\alpha)\notin j[X]$. Since $t(j(\alpha))=j(s(\alpha))\in j[X]$, it follows that $j(\alpha)\neq\alpha$ and $j\restriction\lambda\neq\id_\lambda$. By Lemma \ref{lemma:Equivalentexact}, this shows that $\lambda$ is an exacting cardinal.  
\end{proof}

The following corollary gives an example of how the above results can be applied to obtain a characterization of being an exacting cardinal  as a strengthening  of being J\'onsson:

\begin{corollary}\label{CorollaryExactJonsson}
    The following are equivalent for each cardinal $\lambda\in C^{(1)}$:
      \begin{enumerate}
        \item The cardinal $\lambda$ is exacting. 

        \item For every class $\Ce$  of structures in a countable first-order language that is definable by a formula with parameters contained in $V_\lambda\cup\{\lambda\}$, every structure  of cardinality $\lambda$ in $\Ce$ contains a proper elementary substructure of cardinality $\lambda$ that is isomorphic to a structure in $\Ce$. 

        \item For every class $\Ce$ of  structures in a countable first-order language that is definable by a formula with parameters contained in $V_\lambda\cup\{\lambda\}$, every structure  of cardinality $\lambda$ in $\Ce$ is isomorphic to a proper elementary substructure of a structure of cardinality $\lambda$ in $\Ce$.  \qed 
      \end{enumerate}
\end{corollary}

Recall that the existence of a J\'onsson cardinal  implies $V\ne L$ (see \cite[8.13]{Kan:THI}). We will show in \S\ref{SectExactHOD} below that, similarly, the existence of an exacting cardinal implies $V \neq \HOD$. 
However, while 
the consistency strength of the existence of a J\'onsson cardinal of cofinality $\omega$ is that of a measurable cardinal (see \cite{MR534574}), and the existence of a measurable cardinal implies $V\neq L$,  the consistency of the existence of an exacting cardinal  can be established relative to large cardinals compatible with $V = \HOD$.

%%%%%%

\subsection{The consistency of exact cardinals}\label{SectConsistencyExact}

 In this section, we answer one of the main questions left open by the results of \cite[Question 10.3]{BL} by showing that the consistency of the existence of $n$-exact cardinals can be established from a well-studied very strong large cardinal assumption not known to be inconsistent with ZFC. 
Recall that I0 is the assertion that there exists a non-trivial elementary embedding $\map{j}{L(V_{\lambda+1})}{L(V_{\lambda+1})}$, for some $\lambda$, with $\crit{j}<\lambda$. We shall refer to such an embedding as an \emph{I0 embedding}. Note that since we work in the context of ZFC, an I0 embedding, being a proper class, is implicitly assumed to be definable, possibly with parameters (see the proof of {\cite[Lemma 5]{MR2914848}} for more details).

\begin{theorem}\label{theorem:ExactFromI0}
 If $\map{j}{L(V_{\lambda+1})}{L(V_{\lambda+1})}$ is an  I0 embedding with critical sequence $\vec{\lambda}=\seq{\lambda_m}{m<\omega}$, then there is a set-sized transitive  model $M$ of ZFC such that  $\vec{\lambda}\in M$ and, in $M$, the cardinal $\lambda_0$ is parametrically $n$-exact for $\seq{\lambda_{m+1}}{m<\omega}$ for every natural number $n>0$. 
\end{theorem}

\begin{proof}
  %By a result of Woodin \cite[Lemma 10]{MR2914848}, we may assume that   $j$ is, in fact, an ultrapower embedding given by an $L(V_{\lambda+1})$-ultrafilter over $V_{\lambda +1}$. 
  We start with an observation needed for the construction of the  model with the listed properties. 

\begin{claim*}
 There exists a wellordering $\lhd$ of $V_\lambda$ of order-type $\lambda$, with $j(\lhd)=\lhd$. 
\end{claim*}

\begin{proof}[Proof of the Claim] 
 Pick a wellordering $\lhd_0$ of $V_{\lambda_0}$, and let $\lhd_1=j(\lhd_0)\setminus \lhd_0$. Given $\lhd_m$ for some $0<m<\omega$, set  $\lhd_{m+1}=j(\lhd_m)$. Finally, let $\lhd=\bigcup_{m<\omega}\lhd_m$. Then $\lhd$ is as required.
\end{proof}

Now, set $\Gamma=V_\lambda\cup\{\vec{\lambda},\lhd\}$ and note that this set belongs to $L(V_{\lambda+1})$.  
 By using $\lhd$, in $L(\Gamma)$, we may easily wellorder $\Gamma$ in order-type $\lambda$, so  that $L(\Gamma)$ is a model of ZFC. Moreover,  since $j(\lhd)=\lhd$, we have that $j(\Gamma)=V_\lambda \cup \{ j(\vec{\lambda}), \lhd\}$, hence $L(\Gamma)=L(j(\Gamma))$, and so $\map{j\restriction L(\Gamma)}{L(\Gamma)}{L(\Gamma)}$ is an elementary embedding. Let $M = L_{\lambda^+}(\Gamma)$.
 
\begin{claim*}
$M$ satisfies ZFC and $\map{j\restriction M}{M}{M}$ is an elementary embedding. 
\end{claim*}

\begin{proof}[Proof of the Claim]
Since $\map{j}{L(V_{\lambda+1})}{L(V_{\lambda+1})}$ is an I0 embedding, it follows that  $\lambda^+$ is %$\omega$-strongly 
 measurable in $L(V_{\lambda+1})$ by a theorem of Woodin (see, for example,  %\cite[Theorem 2.5]{Kaf:2004}, or 
  \cite[Theorem 6.2]{MR3855762}). %for a proof). 
   Thus, $L(\Gamma)$ is a model of ZFC and there is an $L(\Gamma)$-ultrafilter on $\lambda^+$, so  $\lambda^+$ is   strongly inaccessible in $L(\Gamma)$, and so $M$ is a model of ZFC. Moreover, since $j(\lambda)=\lambda$ and $(\lambda^+)^{L(V_{\lambda+1})}=\lambda^+$, 
   %$j$ is an ultrapower embedding with critical point less than $\lambda$,   
   %
%In particular, in $L(V_{\lambda+1})$there is a stationary $S\subset \lambda^+$ such that the club filter on $\lambda^+$ restricted to $S$ is an ultrafilter. Let $U$ be this ultrafilter. We may then define an $L(\Gamma)$-ultrapower embedding
%\[k: L(\Gamma) \to \mathrm{Ult}(L(\Gamma),U)\]
%the usual way, where $\mathrm{Ult}(L(\Gamma),U)$ is the set of all functions $f: \lambda^+ \to L(\Gamma)$ with $f \in L(\Gamma)$, taken modulo $U$. Running the usual proof of {\L}o\'s's theorem (using that the axiom of choice holds in $L(\Gamma)$ for the existential step of the induction), we see that $k$ is an elementary embedding. 
%Since $U$ is $\lambda^+$-complete, $\lambda^+$ is the critical point of $k$ and in particular $k$ fixes $\Gamma$.
%Hence, $k$ is a nontrivial elementary embedding from $L(\Gamma)$ to itself fixing $\Gamma$.
%By a result of Kunen, this implies that $\Gamma^\sharp$ exists and thus that every $L(V_{\lambda+1})$-cardinal above $\lambda$ -- and in particular $\lambda^+$  -- is inaccessible in $L(\Gamma)$. Therefore $M$ is a model of ZFC. 
 %
 the cardinal $\lambda^+$ is a fixed point of $j$ and therefore  $\map{j\restriction M}{M}{M}$ is an elementary embedding.
\end{proof}

 Now, fix a natural number $n>0$ and assume, aiming for a contradiction, that, in $M$, the cardinal  $\lambda_0$ is not parametrically $n$-exact for $\seq{\lambda_{m+1}}{m<\omega}$. Pick an ordinal $\lambda<\zeta<\lambda^+$ such that $j(\zeta)=\zeta$ and $\zeta\in (C^{(n+1)})^M$.

 Working  in $M$, pick   $A\in V_{\lambda+1}$ with the property that there is no $n$-exact embedding $\map{i}{Y}{V_\zeta}$ with $Y$ an elementary submodel of $V_\zeta$ containing $V_\lambda\cup\{\lambda\}$, $A\in\ran{i}$, $i(\crit{i})=\lambda_0$,  and $i(\lambda_m)=\lambda_{m+1}$ for all $m<\omega$.  
 Without loss of generality, we may assume that $A\notin\lambda_0\cup\Set{\lambda_m}{m<\omega}$. 
 The elementarity of $j\restriction M$ then implies that, in $M$, for every  elementary submodel $Y$ of $V_\zeta$ with $V_\lambda\cup\{\lambda\}\subseteq Y$, there is no elementary embedding $\map{i}{Y}{V_\zeta}$ with $j(A)\in\ran{i}$, $i(\crit{i})=\lambda_1$,  and $i(\lambda_{m+1})=\lambda_{m+2}$ for all $m<\omega$. 
 
 Still working in $M$,  let $X_0$ be an elementary submodel of $V_\zeta$ of cardinality $\lambda$ with the property that $V_{\lambda}\cup \{\lambda,A\}\subseteq X_0$. Pick a bijection $\map{b_0}{\lambda}{X_0}$ satisfying $b_0(0)=A$,   $b_0(m+1)=\lambda_m$ for all $m<\omega$ and $b_0(\omega+\alpha)=\alpha$ for all $\alpha<\lambda_0$.
Set $X_1=j(X_0)$ and $b_1=j(b_0)$. Since $M$ is closed under $j$, we have $X_1, b_1 \in M$.
The set $X_1$ is  an elementary submodel of $V_\zeta$ of cardinality $\lambda$ with $V_\lambda\cup\{\lambda,j(A)\}\subseteq X_1$ and $\map{b_1}{\lambda}{X_1}$ is a bijection with $b_1(0)=j(A)$, $b_1(m+1)=\lambda_{m+1}$ for all $m<\omega$, and $b_1(\omega+\alpha)=\alpha$ for all $\alpha<\lambda_1$.   
 Moreover, note that  
 \begin{equation}\label{equation:commute}
  b_1\circ (j\restriction\lambda) ~ = ~ (j\restriction X_0)\circ b_0   
 \end{equation} 
 holds, {i.e.,} the following diagram commutes:
     \[
\xymatrix{ 
\lambda\ar[r]^{b_0}\ar[d]_{j} & X_0 \ar[d]^j\\
\lambda \ar[r]_{b_1}  & X_1 
}
\] 

From the fact that $j(X_0) = X_1$ it follows that $\map{j\upharpoonright X_0}{X_0}{X_1}$ is an elementary embedding. Using this,  consider now the following diagram for a fixed $m < \omega$,
     \[
\xymatrix@=4pc{ 
\lambda_m\ar[r]^{b_0\restriction \lambda_m}\ar[d]_{j\restriction\lambda_m} & b_0[\lambda_m] \ar[d]^{{j\restriction (b_0[\lambda_m])}}\\
\lambda_{m+1} \ar[r]_{b_1\restriction \lambda_{m+1}}  & X_1 
}
\] 
and observe that 
the map $j\restriction \lambda_m$ is the identity on $\lambda_0$ and that the map $j\restriction (b_0[\lambda_m])$ yields a partial elementary embedding from $X_0$ to $X_1$, in the sense that for every first-order formula $\phi(\vec x)$ and every tuple $\vec a \in b_0[\lambda_m]$, we have 
\[X_0 \models \phi(\vec a) \quad \mbox{if and only if} \quad  X_1 \models \phi( j(\vec a)).\]

Let us define $T$ to be the set of all functions $\map{t}{\lambda_m}{\lambda_{m+1}}$, for some $m<\omega$, satisfying the following properties:
\begin{enumerate}
\item $t\restriction\lambda_0=\id_{\lambda_0}$,
\item $t(\lambda_k) = \lambda_{k+1}$ for $k < m$,
\item the  function 
\begin{align*}
t_*: b_0[\lambda_m] &\longrightarrow X_1\\
x &\mapsto (b_1\circ t\circ b_0^{{-}1})(x)
\end{align*}
is a partial elementary embedding from $X_0$ to $X_1$, {i.e.,} the map $t$ has the property that the following diagram commutes:  

   \[
\xymatrix@=4pc{ 
\lambda_m\ar[r]^{b_0\restriction \lambda_m}\ar[d]_{t} & b_0[\lambda_m] \ar[d]^{t_\ast}\\
\lambda_{m+1} \ar[r]_{b_1\restriction \lambda_{m+1}}  & X_1 
}
\] 
\end{enumerate}

Note that \eqref{equation:commute} ensures that $$(j\restriction\lambda_m)_* ~ = ~ j\restriction(b_0[\lambda_m])$$ holds for all $m<\omega$, and therefore we have verified that $j\upharpoonright \lambda_m \in T$ for all $m < \omega$.

By ordering $T$ under end-extensions, we can turn $T$ into a tree of height at most $\omega$. 
Since $j\upharpoonright \lambda_m \in T$ for all $m < \omega$, the tree 
$T$ has an infinite branch in $V$ and hence it has an infinite branch $B$ in $M$ as well. Then $\bigcup B$ is a function from $\lambda$
 to $\lambda$ and, if we define $$\map{i ~ = ~ b_1\circ\left(\bigcup B\right)\circ b_0^{{-}1}}{X_0}{X_1}$$ then $i$ yields an elementary embedding of $X_0$ into $V_\zeta^M$ in $M$ with $j(A)\in\ran{i}$, $i\restriction\lambda_0=\id_{\lambda_0}$, and $i(\lambda_m)=\lambda_{m+1}$ for all $m<\omega$. This contradicts our earlier assumptions.
\end{proof} 

The existence of an exacting cardinal appears to be a sensible large-cardinal principle, for, as shown in \cite{BL} (see Theorem \ref{theorem:ESRcorrespondence} below), it is equivalent to a principle of Structural Reflection. This impression  is further reinforced by the equivalence between being an exacting cardinal and a strong form  of J\'onssonness, given in \S\ref{SectExactJonsson} above. However, as we will next show, the existence of these cardinals fundamentally differs from all standard large cardinal axioms, as they imply the existence of non ordinal-definable sets.

\subsection{V is not HOD}\label{SectExactHOD}

Recall that, given a class $\Gamma$, the class $\OD_\Gamma$ consists of all sets that are definable from parameters in $\Gamma\cup\On$,  and  $\HOD_\Gamma$ is the class of all sets $x$ with $\tc{\{x\}}\subseteq\OD_\Gamma$.

\begin{theorem}\label{VnotHOD}
    If $\lambda$ is an exacting cardinal, then $\lambda$ is a regular cardinal in $\HOD_{V_\lambda}$,  and therefore $\HOD_{V_\lambda}\neq V$.
\end{theorem}

\begin{proof}
 Assume, towards a contradiction, that $\lambda$ is singular in $\HOD_{V_\lambda}$. Then there is $x\in V_\lambda$ such that $\lambda$ is singular in $\HOD_{\{x\}}$. Pick $\lambda<\zeta\in C^{(3)}$ and $\alpha<\lambda$ with $x\in V_\alpha$ and $\cof{\lambda}^{\HOD_{\{x\}}}<\alpha$. Since $\lambda$ is exacting, we can find an elementary submodel $X$ of $V_\zeta$ with $V_\lambda\cup\{\lambda\}\subseteq X$ and an elementary embedding  $\map{j}{X}{V_\zeta}$ with $j(\lambda)=\lambda$ and $j\restriction\alpha=\id_\alpha$.

    Now, let $\map{c}{\cof{\lambda}^{\HOD_{\{x\}}}}{\lambda}$ be the least cofinal function from $\cof{\lambda}^{\HOD_{\{x\}}}$ to $\lambda$ in $\HOD_{\{x\}}$ with respect to the $\Sigma_2$-definable canonical  wellordering of $\HOD_{\{x\}}$. 
    Since $\zeta \in C^{(3)}$, it follows that, in $V_\zeta$, the function $c$ is also the least cofinal function from $\cof{\lambda}^{\HOD_{\{x\}}}$ to $\lambda$ in $\HOD_{\{x\}}$, and therefore it is uniquely defined by this property, in the  parameters $\lambda$,  $\cof{\lambda}^{\HOD_{\{x\}}}$ and $x$. Hence, we know that $c$ is an element of  $X$, and this set is defined in $X$ by the same property and with  the same parameters.
    Since $j$ fixes $\lambda$, $\cof{\lambda}^{\HOD_{\{x\}}}$ and $x$, it follows that $j(c)=c$. Moreover, since $j\restriction\alpha=\id_\alpha$, we also know that $j(c(\xi))=c(\xi)$ holds for all $\xi<\cof{\lambda}^{\HOD_{\{x\}}}$. 
    Now, note that, since $V_\lambda\subseteq X$, the Kunen inconsistency implies that $\lambda$ is the supremum of the critical sequence of $j$ and all but boundedly many ordinals below $\lambda$ are moved by $j$.  We now derived a contradiction, because we have also shown that  $\ran{c}$ is a cofinal subset of $\lambda$ with $j\restriction\ran{c}=\id_{\ran{c}}$.  
    \end{proof}

Theorem \ref{VnotHOD} shows that exacting cardinals cause a strong form of the Axiom of Choice to fail. More specifically, if $\lambda$ is an exacting cardinal, then no wellordering of $V$ is definable using parameters in $V_\lambda$.\footnote{Note that the existence of such a wellordering is equivalent to the assumption $V=\HOD_{V_\lambda}$.}
%in fact that if there is a $2$-exact embedding, then a strong form of the Axiom of Choice fails: the existence of a wellordering of $V$ definable without parameters.
 However, by the proof of Theorem \ref{theorem:ExactFromI0}, the existence of an exacting cardinal  $\lambda$ is consistent with the assumption that $V = \HOD_{V_{\lambda+1}}$ (since the model $M$ in the proof satisfies this) and therefore also with  the existence of a wellordering of $V$ that is definable from parameters in $V_{\lambda+1}$, so Theorem \ref{theorem:ExactFromI0} is best possible.

\begin{corollary}\label{corollary:noimplication}
    If ZFC together with the existence of an I0 embedding is consistent, then this theory does not prove the existence of an exacting cardinal. 
\end{corollary}

\begin{proof}
  Standard class forcing arguments (due to Jensen and MacAloon) show that if    ZFC is consistent with the existence of an I0 embedding, then this theory is also consistent with the statement that $V=\HOD$. For suppose $\map{j}{L(V_{\lambda +1})}{L(V_{\lambda +1})}$ is an I0 embedding. By the proof of {\cite[Lemma 5]{MR2914848}}, we  may assume that $j$ is induced by an ultrapower of $L(V_{\lambda+1})$ using an ultrafilter $U$ on $V_{\lambda+2}^{L(V_{\lambda+1})}$.   Let $\mu\in C^{(1)}$ be a cardinal greater than $\lambda$. Then one can first force the GCH above $\mu$, and then force $V=\HOD$ via a class forcing $\mathbb{P}$ that codes every set in the forcing extension into the power-set function on cardinals above $\mu$, so that the forcing notions are $\mu$-closed, and therefore do not change $V_{\mu}$. %Moreover, $\mathbb{P}$ is $\Delta_2$-definable, with $\mu$ as a parameter. Hence any definable I0-embedding in $V$ yields a definable, with the same parameter $p$,  I0-embedding in the forcing extension. 
  %
  %In $V$, define $$U:=\{ X\subseteq V_{\lambda +1}: X\in L(V_{\lambda +1}) \mbox{ and } j\restriction V_\lambda \in j(X)\}.$$
  %Then $U$ is a normal  $L(V_{\lambda +1})$-ultrafilter over $V_{\lambda +1}$. 
  %
  Since  these class forcings  do not add new elements of $V_\gamma$, 
  these forcings do not change $L(V_{\lambda+1})$ and, when we construct the ultrapower of $L(V_{\lambda +1})$ by $U$ in the forcing extension, then we again obtain the embedding $\map{j}{L(V_{\lambda +1})}{L(V_{\lambda +1})}$. 
  %$U$ remains a normal $L(V_{\lambda +1})$-ultrafilter over $V_{\lambda +1}$ . Now, by taking in the forcing extension the 
  %, the corresponding ultrapower embedding yields an I0 embedding $L(V_{\lambda +1})\to L(V_{\lambda +1})$, possibly different form $j$ (see \cite[Section 4]{MR3855762} for details). 
  This shows that the given class forcing extension is a model of ZFC containing an I0 embedding and satisfying $V=\HOD$. Finally, an application of Theorem \ref{VnotHOD} shows that this model does not contain exacting cardinals. 
\end{proof}

We have shown (Theorem \ref{VnotHOD}) that the existence of an exacting cardinal causes $\HOD$ to be incorrect about its regularity. 
Thus, in view of Woodin's HOD Dichotomy Theorem {\cite[Theorem 3.39]{MR3632568}}, the existence of an exacting cardinal is therefore particularly interesting in the presence of other large cardinals, {e.g.,} extendible cardinals, for the existence of an extendible cardinal below an exacting cardinal implies the failure of Woodin's HOD Hypothesis. This will be the theme of \S\ref{section6} below. 

%%%%%%%%%%%%%%%%%%%
%%%%%%%%%%%%%%%%%%%

\section{Ultraexacting Cardinals}\label{Ultraexact}

Proposition \ref{proposition:NotInDomainExactt} shows that we can strengthen the notion of an exact embedding at a cardinal $\lambda$ by demanding that certain elements of $V_{\lambda+1}$ are contained in the domain of the embedding. 
In the proof of Theorem \ref{theorem:ExactFromI0} above, the embedding $\map{i}{X_0}{X_1}$ constructed is, in general, different from the embedding $\map{j\restriction X_0}{X_0}{X_1}$ induced by the I0-embedding, and its existence is only guaranteed by an absoluteness argument. 
In particular, since we first chose the model $X_0$ and then constructed the embedding $i$, we cannot expect  the map $i\restriction V_\lambda$ to be an element of $X_0$. 
These observations motivate the notion introduced in the next definition:

\begin{definition}
\label{defultraexact}
  Let $n>0$ be a natural number and let  $\lambda$   be a limit cardinal. 
  \begin{enumerate}
      \item An \emph{$n$-ultraexact embedding at $\lambda$} is an $n$-exact embedding $\map{j}{X}{V_\zeta}$ at $\lambda$ with the property that $j\restriction V_\lambda\in X$.  

    \item Given a strictly increasing sequence $\vec{\lambda}=\seq{\lambda_m}{m<\omega}$ of cardinals with supremum $\lambda$, a cardinal $\kappa<\lambda_0$ is \emph{$n$-ultraexact for $\vec{\lambda}$} if for every $A\in V_{\lambda+1}$, there exists 
      an $n$-ultraexact embedding $\map{j}{X}{V_\zeta}$ at $\lambda$ with $A\in\ran{j}$,      $j(\kappa)=\lambda_0$ and $j(\lambda_m)=\lambda_{m+1}$ for all $m<\omega$.  
     If we further require that $j(\crit{j})=\kappa$, then we say that $\kappa$ is \emph{parametrically $n$-ultraexact for $\vec{\lambda}$}. 
  \end{enumerate}
\end{definition}

\subsection{Ultraexactness and rank-Berkeley cardinals}
 Lemma \ref{lemma:EquivalentUltra} below is an analogue of Lemma \ref{lemma:Equivalentexact} which establishes that, as in the case of $n$-exact embeddings, the existence of an $n$-ultraexact embedding at a limit ordinal $\lambda$ does not depend on the parameter $n$. %\jpa{What about $n = 1$? Is this weaker?}

\begin{lemma}\label{lemma:EquivalentUltra}
Given a natural number $n>0$, the following statements are equivalent for every  limit ordinal $\lambda$ and every set $x$: 
\begin{enumerate}
    \item There is an $n$-ultraexact embedding $\map{j}{X}{V_\zeta}$ at $\lambda$ with $x\in X$ and $j(x)=x$. 
    
    %\item There are elements $\eta$ and $\zeta$ of $C^{(2)}$ greater than $\lambda$, an elementary submodel $X$ of $V_\eta$ with $V_\lambda\cup\{\lambda,x\}\subseteq X$, and an elementary embedding $\map{j}{X}{V_\zeta}$ with $j(\lambda)=\lambda$, $j(x)=x$, $j\restriction\lambda \neq \id_\lambda$, and  $j\restriction V_\lambda\in X$. 
    
    \item\label{item:EquivalentUltra2} For every $\zeta>\lambda$ with $x\in V_\zeta$, there is an elementary submodel $X$ of $V_\zeta$ with  $V_\lambda\cup\{\lambda,x\}\subseteq X$, and an elementary embedding $\map{j}{X}{V_\zeta}$ with $j(\lambda)=\lambda$, $j(x)=x$, $j\restriction \lambda \neq \id_\lambda$ and  $j\restriction V_\lambda\in X$. 

    \item For every $\zeta>\lambda$ with $x\in V_\zeta$ and every $\alpha<\lambda$, there is an elementary submodel $X$ of $V_\zeta$ with  $V_\lambda\cup\{\lambda,x\}\subseteq X$, and an elementary embedding $\map{j}{X}{V_\zeta}$ with $j(\lambda)=\lambda$, $j(x)=x$, $j\restriction\alpha=\id_\alpha$, $j\restriction \lambda \neq \id_\lambda$ and  $j\restriction V_\lambda\in X$. 
\end{enumerate}
\end{lemma}

\begin{proof}
Argue as in the proof of Lemma \ref{lemma:Equivalentexact}, taking into account the extra requirement that $j\restriction V_\lambda \in X$.  
\end{proof}

As before, we can abstract from Lemma \ref{lemma:EquivalentUltra} a characterization of the existence of ultraexact embeddings  as a weak form of rank-Berkeleyness, leading to the following definition:

\begin{definition}
A cardinal $\lambda$ is \emph{ultraexacting} if for every $\zeta>\lambda$ there exist $X \prec V_\zeta$ with $V_{\lambda} \cup \{\lambda\}\subseteq X$ and an elementary embedding $\map{j}{X}{V_\zeta}$ such that $j(\lambda)=\lambda$, $j\restriction\lambda\neq\id_\lambda$, and  $j\restriction V_\lambda \in X$. 
\end{definition}

Thus, ultraexacting cardinals are defined just like  exacting cardinals, except for the requirement that the embeddings satisfy $j\upharpoonright V_\lambda \in X$. Observe that this condition holds trivially for embeddings obtained from a rank-Berkeley cardinal.  
%The requirement is somewhat reminiscent of the \textit{Hauser embedding property} for weakly compact or strongly unfoldable cardinals, which explains the terminology. 
Applying Lemma \ref{lemma:EquivalentUltra} to $n = 1$ and $x = \varnothing$, we obtain the following direct analog of Corollary \ref{corollary:Exacting1ExactEmbedding}:

\begin{corollary}
A cardinal $\lambda$ is ultraexacting if and only if there is a $1$-ultraexact embedding at $\lambda$. \qed
\end{corollary}

It is easy to see that, by Lemma \ref{lemma:EquivalentUltra}, the observations about the interplay between $1$-exact embeddings and extendible cardinals stated after Lemma \ref{lemma:Equivalentexact} also apply in the case of $1$-ultraexact embeddings.
In the following, we  discuss an aspect of $n$-ultraexact embeddings that has no direct analog for $n$-exact embeddings. Lemma \ref{lemma:EquivalentUltra} shows that,  if there is a cardinal that is $1$-ultraexact for some sequence of cardinals with supremum $\lambda$, then $\lambda$ is an ultraexacting cardinal. The next corollary uses the lemma to derive a strong converse of this implication. In particular, it shows that the existence of $n$-ultraexact cardinals is equivalent for all natural numbers $n>0$.

\begin{corollary}
  For each natural number $n>0$, every ultraexacting cardinal $\lambda$ is a limit of cardinals that are parametrically $n$-ultraexact for a sequence of cardinals  with supremum $\lambda$.  
\end{corollary}

\begin{proof}
 Fix $\alpha<\lambda$ and $\lambda<\eta\in C^{(n+2)}$. Our assumption on $\lambda$ yields an elementary submodel $X$ of $V_\eta$ with $V_\lambda\cup\{\lambda\}\subseteq X$ and an elementary embedding $\map{j}{X}{V_\eta}$  with $j(\lambda)=\lambda$, $j\restriction\alpha=\id_\alpha$, $j\restriction\lambda\neq\id_\lambda$, and $j\restriction V_\lambda\in X$. Let $\seq{\lambda_m}{m<\omega}$ denote the critical sequence of $j$. Then $\alpha\leq\lambda_0<\lambda_1$ and $\sup_{m<\omega}\lambda_{m+2}=\lambda$. 
  Assume, towards a contradiction, that $\lambda_1$ is not parametrically $n$-ultraexact for $\seq{\lambda_{m+2}}{m<\omega}$. Then there exists $A\in V_{\lambda+1}$ with the property that there is no $n$-ultraexact embedding $\map{i}{Y}{V_\zeta}$ at $\lambda$ with $A\in\ran{i}$, $i(\crit{i})=\lambda_1$,  $i(\lambda_1)=\lambda_2$ and $i(\lambda_{m+2})=\lambda_{m+3}$ for all $m<\omega$. Since $\lambda$, $\lambda_1$ and the sequence $\seq{\lambda_{m+2}}{m<\omega}$ are all contained in $\ran{j}$, the correctness properties of $V_\eta$ and $X$ allow us to find such an $A\in V_{\lambda+1}$ that is contained in $\ran{j}$. But, this yields a contradiction, because $j$ is an $n$-ultraexact embedding at $\lambda$ with $A\in\ran{j}$, $j(\crit{j})=\lambda_1$,  $j(\lambda_1)=\lambda_2$ and $j(\lambda_{m+2})=\lambda_{m+3}$ for all $m<\omega$.  
\end{proof}

 Note that the above argument cannot be adapted to $n$-exact embeddings $\map{j}{X}{V_\zeta}$ with critical sequence $\seq{\lambda_m}{m<\omega}$, because the sequence $\seq{\lambda_{m+2}}{m<\omega}$ might not be an element of $\ran{j}$. 
 Very similar problems arise in {\cite[Question 10.4]{BL}}.

 Analogous to the results of \S\ref{SectExactJonsson}, it is also possible to characterize the property of being an ultraexacting cardinal as a further strengthening of the property of being J\'onsson. However, the resulting characterizations appear less elegant as compared to the characterizations of exacting cardinals, and so we will not discuss them further. In contrast, it turns out that ultraexacting cardinals can be naturally characterized as a principle of Structural Reflection.  We will explore this connection in detail in  \S\ref{SectSSR} below.

\subsection{Extending ultraexact embeddings}\label{subsec3.1}

 Throughout   this section, we aim to show that the domains of ultraexact embeddings  contain extensions of these embeddings to certain sets that are not necessarily subsets of these domains. More specifically, given a $2$-ultraexact embedding $\map{j}{X}{V_\zeta}$, we want to show that for certain $x$ in $X$, there exists a map $k$ with domain $x$ that is an element of $X$ and satisfies $k\restriction(X\cap x)=j\restriction(X\cap x)$. Note that the domain of a $2$-ultraexact embedding always contains elements for which such extensions do not exist. For example, if $\map{j}{X}{V_\zeta}$ is such an embedding at some cardinal $\lambda$, then $V_{\lambda+2}\in X$ and there is no function $k$ in $X$ with domain $V_{\lambda+2}$ and $k\restriction(V_{\lambda+2}\cap X)=j\restriction(V_{\lambda+2}\cap X)$, because otherwise $X$ would think that $k$ is a non-trivial elementary embedding of $V_{\lambda+2}$ into itself and the correctness properties of $X$
 would ensure that this statement also holds in $V$.

 We start by extending ultraexact embeddings at some cardinal $\lambda$ to $V_{\lambda+1}$. For this purpose, we recall the canonical method for extending I3-embeddings $\map{j}{V_\lambda}{V_\lambda}$ to maps from $V_{\lambda+1}$ to $V_{\lambda+1}$. 
 %Given a limit ordinal $\lambda$ and a function $\map{f}{V_\lambda}{V_\lambda}$, we recall the definition of $f_+$ given by \ref{DefinitionFPlus}.

 \begin{definition}\label{DefinitionFPlus}
 Given   a limit ordinal $\lambda$ and a function $\map{f}{V_\lambda}{V_\lambda}$, we define 
\begin{align*}
  f_+ : V_{\lambda+1} &\longrightarrow V_{\lambda+1}\\
A &\mapsto \bigcup\Set{f(A\cap V_\alpha)}{\alpha<\lambda}. 
\end{align*} 
\end{definition}

It is a well-known fact that, if $\map{j}{V_\lambda}{V_\lambda}$ is an I3-embedding, then $\map{j_+}{V_{\lambda+1}}{V_{\lambda+1}}$ is $\Sigma_0$-elementary embedding that extends $j$ (see {\cite[Lemma 1.6]{MR2768690}}). We will now show that  ultraexact embeddings at $\lambda$ induce \emph{I1-embeddings}, {i.e.,} non-trivial, fully elementary embeddings from $V_{\lambda+1}$ to $V_{\lambda+1}$.

 \begin{lemma}\label{lemma:j_+}
     If $\map{j}{X}{V_\zeta}$ is a $1$-ultraexact embedding at $\lambda$, then the map $$\map{(j\restriction V_\lambda)_+}{V_{\lambda+1}}{V_{\lambda+1}}$$ is an  I1-embedding   belonging to  $X$ and satisfying
     \begin{equation}\label{equation:j_+=j}
      j\restriction(V_{\lambda+1}\cap X) ~ = ~ (j\restriction V_\lambda)_+\restriction(V_{\lambda+1}\cap X).
    \end{equation}
 \end{lemma}

 \begin{proof}
   Set $i=(j\restriction V_\lambda)_+$. First, notice that our assumptions on $X$ and $j$  ensure that both $V_{\lambda+1}$ and $i$ are elements of $X$ with $j(V_{\lambda+1})=V_{\lambda+1}$. 
   Moreover, since $X$ contains an $\omega$-sequence that is cofinal in $\lambda$, elementarity directly implies that $$i\restriction(V_{\lambda+1}\cap X) ~ = ~ j\restriction(V_{\lambda+1}\cap X).$$
   
   Now, pick $a_0,\ldots,a_{n-1}\in V_{\lambda+1}\cap X$ and a formula $\varphi(v_0,\ldots,v_{n-1})$ in the language of set theory with the property that, in $X$, the statement $\varphi(a_0,\ldots,a_{n-1})$ holds in $V_{\lambda+1}$. Then the elementarity of $j$ and our earlier observations  ensure that $\varphi(i(a_0),\ldots,i(a_{n-1}))$ holds in $V_{\lambda+1}$. Since $i(a_0),\ldots,i(a_{n-1}),V_{\lambda+1}\in X$, we can now conclude that, in $X$, the statement $\varphi(i(a_0),\ldots,i(a_{n-1}))$ holds in $V_{\lambda+1}$.

We have thus shown  that $i$ is an I1-embedding in $X$, and the correctness properties of $X$ imply that it is also an I1-embedding in $V$. 
 \end{proof}

 The following lemma is the first step for showing how one can  extend ultraexact embeddings to larger sets.

 \begin{lemma}\label{lemma:ExtendWithSurjection}
  Let  $\map{j}{X}{V_\zeta}$ be a $1$-ultraexact embedding at $\lambda$, and let $\gamma$ be an ordinal in $X$  such that there is a surjection  $\map{s}{V_{\lambda+1}}{\gamma}$  in $X$ with $j(s)\in X$. Then there exists  a unique function $\map{j_\gamma}{\gamma}{j(\gamma)}$ that is an element of  $X$ and satisfies 
      \begin{equation}\label{equation:j_s=j_for_gamma}
         j\restriction(X\cap\gamma) ~ = ~ j_\gamma\restriction(X\cap\gamma).
    \end{equation}
 \end{lemma}

 \begin{proof}
     Given $x,y\in V_{\lambda+1}\cap X$ with $s(x)=s(y)$, we can apply \eqref{equation:j_+=j} to see that $$j(s)((j\restriction V_\lambda)_+(x)) ~ = ~ j(s)(j(x)) ~ = ~ j(s(x)) ~ =  ~ j(s(y)) ~ =  ~ j(s)(j(y)) ~ =  ~ j(s)((j\restriction V_\lambda)_+(y))$$ holds. Hence, in $X$ it holds that  there is a well-defined function $\map{j_\gamma}{\gamma}{j(\gamma)}$ satisfying $$j_\gamma(s(x)) ~ =  ~ j(s)((j\restriction V_\lambda)_+(x))$$ for all $x\in V_{\lambda+1}$. By the correctness of $X$, the function  $\map{j_\gamma}{\gamma}{j(\gamma)}$ has  the same properties in $V$. 

     Now, fix $\beta\in X\cap\gamma$. Then there is $x\in V_{\lambda+1}\cap X$ with $s(x)=\beta$ and \eqref{equation:j_+=j} shows that $$j_\gamma(\beta) ~ = ~ j(s)((j\restriction V_\lambda)_+(x)) ~ = ~ j(s)(j(x)) ~ = ~ j(s(x)) ~ = ~ j(\beta).$$

     Finally, to show uniqueness, let $\map{k}{\gamma}{j(\gamma)}$ be a function in $X$ such that  $k\restriction(X\cap\gamma)=j\restriction(X\cap\gamma)$. Then  we have $k(\beta)=j(\beta)=j_\gamma(\beta)$ for all $\beta\in X\cap\gamma$ and elementarity allows us to conclude that $k=j_\gamma$. 
  \end{proof}

  %As an application of the two results above, we show that every ultraexact embedding can be extended to the successor of the corresponding limit cardinals: 

  %\begin{corollary}
   %   If $\map{j}{X}{V_\zeta}$ is a $2$-ultraexact embedding at $\lambda$, then there is a surjection $\map{s}{V_{\lambda+1}}{\lambda^+}$ in $X$ with $j(s)=s$ and hence $\map{j_s}{\lambda^+}{\lambda^+}$ is the   unique function $k$ from $\lambda^+$ to $\lambda^+$ in $X$ with $k\restriction(X\cap\lambda^+)=j\restriction(X\cap\lambda^+)$. 
  %\end{corollary}

  %\begin{proof}
   %   Let $\map{s}{V_{\lambda+1}}{\lambda^+}$ be the unique function that sends a well-ordering of an ordinal less than or equal to $\lambda$ to its order-type and all other elements of $V_{\lambda+1}$ to $0$. Then $s$ is an element of $X$ with $j(s)$. Moreover, the uniqueness of $j_s$ directly follows from Lemma \ref{lemma:ExtendWithSurjection}. 
  %\end{proof}

  The following lemma will play a crucial role in the applications of the results of this section.

\begin{lemma}\label{lemma:FixedPointsInXclub}
    If $\map{j}{X}{V_\zeta}$ is a $1$-ultraexact embedding at $\lambda$, the ordinal $\gamma\in X$ is   greater than $\lambda$, of uncountable cofinality and  such that   $j(\gamma)=\gamma$, and $\map{s}{V_{\lambda+1}}{\gamma}$ is a surjection in $X$  with $j(s)\in X$, then   the set  $$C_{j,\gamma} ~ := ~ \Set{\beta<\gamma}{j_\gamma(\beta)=\beta}$$ is an unbounded and $\omega$-closed subset of $\gamma$ that is an element of $X$. 
\end{lemma}

\begin{proof}
  First, note that the fact that  $\gamma$ and $j_\gamma$ (as given by Lemma \ref{lemma:ExtendWithSurjection}) are both elements of $X$ implies that $C_{j,\gamma}$ also belongs to $X$. Now suppose $\vec{\alpha}=\seq{\alpha_i}{i<\omega}$ is a strictly increasing sequence of ordinals in $C_{j,\gamma}$ that is an element of $X$. Set $\alpha=\sup_{i<\omega}\alpha_i\in X\cap\gamma$. Then the elementarity of $j$ and \eqref{equation:j_s=j_for_gamma} ensure that $$j_\gamma(\alpha) ~ = ~ j(\alpha) ~ = ~ \sup_{i<\omega}j(\alpha_i) ~ = ~ \sup_{i<\omega}j_\gamma(\alpha_i) ~ = ~ \sup_{i<\omega}\alpha_i ~ = ~ \alpha$$
  and therefore $\alpha  \in  C_{j,\gamma}$. This  shows that $C_{j,\gamma}$ is an $\omega$-closed subset of $\gamma$ in $X$, and the correctness of $X$ ensures that $C_{j,\gamma}$ also has this property in $V$. 

  Now, pick an ordinal $\beta\in X\cap\gamma$. Note that \eqref{equation:j_s=j_for_gamma} and the elementarity of $j$ ensure that the function $j_\gamma$ is strictly increasing. Since $j_\gamma\in X$, we know that $X$ contains a strictly increasing sequence $\seq{\alpha_i}{i<\omega}$ of ordinals in the interval $(\beta,\gamma)$ with the property that $\alpha_{i+1}>j_\gamma(\alpha_i)$ holds for all $i<\omega$. Set $\alpha=\sup_{i<\omega}\alpha_i$ and note that $\alpha \in X\cap\gamma$. As above, we have that $$j_\gamma(\alpha) ~ = ~ j(\alpha) ~  = ~ \sup_{i<\omega}j(\alpha_i) ~ = ~ \sup_{i<\omega}j_\gamma(\alpha_i) ~ = ~ \alpha$$
  and so $\alpha \in  C_{j,\gamma}\setminus\beta$. This shows that $C_{j,\gamma}$ is an unbounded subset of $\gamma$ in $X$, and the correctness  of $X$ again implies that this also true in $V$. 
\end{proof}

   Recall that, given some non-empty set $X$, the ordinal $\Theta_X$ is defined to be the supremum of all ordinals $\gamma$ such that there is a surjection from $X$ onto $\gamma$. As usual, we  are particularly interested in computing the ordinal $\Theta_X^M$ in inner models $M$  of ZF containing $X$. 
 In the following, given a set $E$, we shall  write $\Theta^{L(V_{\lambda+1},E)}$ instead of $\Theta_{V_{\lambda+1}}^{L(V_{\lambda+1},E)}$. An argument due to Solovay (see \cite[Exercise 28.19]{Kan:THI}) shows that $\Theta^{L(V_{\lambda+1},E)}$ is a regular cardinal in $L(V_{\lambda+1},E)$.

  \begin{lemma}\label{proposition:ThetaFixed}
       If $\map{j}{X}{V_\zeta}$ is a $1$-ultraexact embedding at $\lambda$ and $E\in V_{\lambda+2}\cap X$ is such that $j(E)=E$, then $\Theta^{L(V_{\lambda+1},E)}\in X$ and  $j(\Theta^{L(V_{\lambda+1},E)})=\Theta^{L(V_{\lambda+1},E)}$. 
  \end{lemma}

  \begin{proof}
   First, note that the ordinal $\Theta^{L(V_{\lambda+1},E)}$ can be defined by a $\Sigma_2$-formula with parameters  $E$ and $\lambda$. Since $\zeta$ is an element of $C^{(2)}$, it follows that $\Theta^{L(V_{\lambda+1},E)}<\zeta$ and this ordinal can be defined in $V_\zeta$ by the same $\Sigma_2$-formula. Since $j(\lambda)=\lambda$ and $j(E)=E$, the elementarity of $j$ and the fact that all $\Sigma_2$-statement are upwards absolute from $X$ to $V$ now ensure that $\Theta^{L(V_{\lambda+1},E)}$ is an element of $X$ with $j(\Theta^{L(V_{\lambda+1},E)})=\Theta^{L(V_{\lambda+1},E)}$. 
  \end{proof}

  \begin{lemma}\label{lemma:ExtendBelowTheta}
   If $\map{j}{X}{V_\zeta}$ is a $1$-ultraexact embedding at $\lambda$, $E\in V_{\lambda+2}\cap X$ with $j(E)=E$  and $\gamma\in X\cap\Theta^{L(V_{\lambda+1},E)}$ with $j(\gamma)\in X$, then there is a surjection $\map{s}{V_{\lambda+1}}{\gamma}$ in $X$ with $j(s)\in X$. 
  \end{lemma}

\begin{proof}
    Let $\map{S}{\On\times V_{\lambda+1}}{L(V_{\lambda+1},E)}$ be the canonical class surjection onto $L(V_{\lambda+1},E)$. Then the map $S$ is definable by a $\Sigma_1$-formula with parameters $V_{\lambda+1}$ and $E$. 
    Moreover, since $\gamma\in X\cap\Theta^{L(V_{\lambda+1},E)}$, the correctness properties of $X$ allow us to find $x\in V_{\lambda+1}\cap X$ with the property that there exists an ordinal $\beta$ such that $S(\beta,x)$ is a surjection from $V_{\lambda+1}$ onto $\gamma$. 
    Let $\alpha$ be the minimal ordinal with the property that $S(\alpha,x)$ is a surjection from $V_{\lambda+1}$ onto $\gamma$  and set $s=S(\alpha,x)$. 
    Then $\alpha$ is definable by a $\Sigma_1$-formula with parameters $E$, $V_{\lambda+1}$, $x$ and  $\gamma$. The correctness properties of $X$ then ensure that  $\alpha$ is an element of $X$, and therefore we know that $s$ is also an element of $X$. Moreover, since $j(E)=E$, we know that $j(\alpha)$ is the minimal ordinal $\beta$ such that $S(\beta,j(x))$ is a surjection from $V_{\lambda+1}$ onto $j(\gamma)$, and thus  $j(\alpha)$ is definable by a $\Sigma_1$-formula with parameters $E$, $V_{\lambda+1}$, $j(x)$ and  $j(\gamma)$. An application of \eqref{equation:j_+=j} (see Lemma \ref{lemma:j_+}) now gives that $j(x)=(j\restriction V_\lambda)_+(x)\in X$. Since we also have $j(\gamma)\in X$, we can again use the correctness properties of $X$ to conclude that $j(\alpha)$ is an element of $X$, and therefore so is $j(s)=S(j(\alpha),j(x))$. 
\end{proof}

  We will next show that, by possibly changing to another embedding, we can find ultraexact embeddings that extend to ordinals above $\lambda^+$.

  \begin{lemma}\label{lemma:NewEmbeddingSurjection}
    If  $\map{i}{Y}{V_\eta}$ is a $1$-ultraexact embedding at  $\lambda$ and $E\in V_{\lambda +2} \cap Y$ is such that $i(E)=E$, then for every $n>1$, every $\gamma<\Theta^{L(V_{\lambda+1},E)}$, and every $\alpha<\lambda$, there exists a surjection $\map{s}{V_{\lambda+1}}{\gamma}$ and an $n$-ultraexact embedding $\map{j}{X}{V_\zeta}$ at $\lambda$ with $E,s,j(s),\gamma\in X$, $j(E)=E$, $j(\gamma)=\gamma$ and $j\restriction\alpha=\id_\alpha$. 
  \end{lemma}

  \begin{proof}
   Assume, towards a contradiction, that the above conclusion fails for some natural number $n$, and let $\gamma < \Theta^{L(V_{\lambda+1},E)}$ be the least ordinal witnessing the failure. Then $\gamma$ is definable from the parameters $\lambda$ and $E$. 
   %with the property that there is a $\rho<\lambda$ such that    for every $n$-ultraexact embedding $\map{j}{X}{V_\zeta}$ at $\lambda$ with $E,\beta\in X$, $j(E)=E$, $j(\beta)=\beta$ and $\crit{j}>\rho$, the set $X$ does not contain a surjection   $\map{s}{V_{\lambda+1}}{\beta}$ with $j(s)\in X$. 
   %
   Let $\alpha<\lambda$ be an ordinal such that for every $n$-ultraexact embedding $\map{j}{X}{V_\zeta}$ at $\lambda$ with $E,\gamma\in X$, $j(E)=E$, $j(\gamma)=\gamma$ and $j\restriction\alpha=\id_\alpha$, the set $X$ does not contain a surjection   $\map{s}{V_{\lambda+1}}{\gamma}$ with $j(s)\in X$. 
   Now, let $m>n$ be a sufficiently large natural number and apply Lemma \ref{lemma:EquivalentUltra} to find an $m$-ultraexact embedding $\map{j}{X}{V_\zeta}$ with $E\in X$, $j(E)=E$ and $j\restriction\alpha=\id_\alpha$. Our choice of  $m$  ensures that  $\gamma\in X$ and  $j(\gamma)=\gamma$. This allows us to   apply    Lemma \ref{lemma:ExtendBelowTheta} to find a surjection $\map{s}{V_{\lambda+1}}{\gamma}$ in $X$ with $j(s)\in X$, a contradiction. 
 \end{proof}

 Equipped with the lemmata above we are now ready to derive, in the next three subsections, some strong consequences of the existence of ultraexact embeddings.

%%%%%%%%%%%%%%%%%%%%%%%%%%%%%%

\subsection{$\omega$-strongly measurable cardinals} \label{SectHOD}

We already proved in Theorem \ref{VnotHOD} that if there exists a $2$-exact embedding, then $V$ is not equal to $\HOD$. We will show next  that the existence of a $2$-ultraexact embedding at a cardinal $\lambda$ implies a much stronger divergence between $V$ and $\HOD$, as, e.g., $\lambda^+$ becomes  $\omega$-strongly measurable in $\HOD$. 
We start by recalling the definition of this notion. Important results analyzing  this property can be found in \cite{doi:10.1142/S0219061324500181}, \cite{Ben-Neria_Hayut_2023}, \cite{MR4693981} and \cite{WOEM1}.

\begin{definition}[Woodin, \cite{WOEM1}]
\label{defstronglymeasurable}
  An uncountable regular cardinal $\kappa$ is \emph{$\omega$-strongly measurable in $\HOD$} if there exists a cardinal $\delta<\kappa$ with $(2^\delta)^\HOD<\kappa$ such that there is no partition $\seq{S_\alpha}{\alpha<\delta}$ in $\HOD$ of the set $S^\kappa_\omega:=\{ \alpha <\kappa: \cof{\alpha}=\omega\}$  into stationary sets. 
\end{definition}

The next result and Theorem \ref{theorem:UltraexactFromI0} below should be compared with {\cite[Lemma 190]{WOEM1}}.

\begin{theorem}\label{theorem:stronglymeasurable}
  If $\map{i}{Y}{V_\xi}$ is a $2$-ultraexact embedding at a  cardinal $\lambda$ and $E\in V_{\lambda+2}\cap Y$ with $i(E)=E$, then every regular cardinal in the interval $(\lambda,\Theta^{L(V_{\lambda+1},E)})$  is $\omega$-strongly measurable in $\HOD$.    
\end{theorem}

\begin{proof}
  Pick a regular cardinal $\lambda<\kappa<\Theta^{L(V_{\lambda+1},E)}$ and apply Lemma \ref{lemma:NewEmbeddingSurjection} to find a surjection $\map{s}{V_{\lambda+1}}{\kappa}$ and  a $2$-ultraexact embedding $\map{j}{X}{V_\zeta}$ with $E,s,\kappa\in X$, $j(E)=E$, $j(\kappa)=\kappa$ and $j(s)\in X$.    Set $\delta=\crit{j}<\lambda$. Then $(2^\delta)^\HOD<\lambda<\kappa$.      
  Assume, towards a contradiction, that $\HOD$ contains a partition of $S^\kappa_\omega$ into $\delta$-many sets such that each set is a stationary subset of $\kappa$ in $V$.     Let $\vec{S}=\seq{S_\alpha}{\alpha<\delta}$ denote the least  partition with this property in the canonical wellordering of $\HOD$. 
  
  Since both $\HOD$ and its canonical wellordering are $\Sigma_2$-definable, the set $\vec{S}$ is  definable from the parameters $\kappa$ and $\delta$. 
  Hence, the correctness properties of $X$ ensure that  $\vec{S}$ is an element of $X$. 
  %in any $V_\alpha$ with $\alpha \in C^{(2)}$ that contains it.   since $X\preceq V_\eta$ for some $\eta \in C^{(2)}$, and $\vec{S}\in V_\eta$, it follows that $\vec{S}$ is an element of $X$. 

    The elementarity of $j$ and the fact that $\zeta \in C^{(3)}$ then imply that $j(\vec{S})$ is the least partition of $S^{\kappa}_\omega$ into $j(\delta)$-many stationary sets, in the canonical wellordering of $\HOD$. 
    
    Again, this shows that the set $j(\vec{S})$ is $\Sigma_2$-definable from the parameters $\kappa$ and $j(\delta)$, and,  since $j(\delta)\in X$, we can conclude  that $j(\vec{S})=\seq{T_\beta}{\beta<j(\delta)}$ is also an element of $X$.  
    
    Let $C=C_{j,\kappa} \in X$ be the unbounded and $\omega$-closed subset of $\kappa$ defined in Lemma \ref{lemma:FixedPointsInXclub}. 
    Since elementarity implies that $T_\delta$ is a stationary subset of $S^\kappa_\omega$, we know that $C\cap T_\delta\neq\emptyset$, and the fact that $C$ and $T_\delta$ belong to $X$ implies that  $\gamma=\min(C\cap T_\delta)\in S^\kappa_\omega\cap X$. But, then there is $\alpha<\delta$ with $\gamma\in S_\alpha$ and, since $\gamma \in C$, we can use \eqref{equation:j_s=j_for_gamma} to conclude that $$\gamma ~ = ~ ~ j_\kappa (\gamma) ~ = ~ j(\gamma) ~ \in ~ j(S_\alpha)\cap T_\delta ~ = ~ T_\alpha\cap T_\delta ~ = ~ \emptyset,$$ thus yielding a contradiction.  
\end{proof}

As a particular case of the theorem, we have the following:

\begin{corollary}\label{corollary:omegaHOD}
    If $\lambda$ is an ultraexacting cardinal, then  $\lambda^+$ is $\omega$-strongly measurable in $\HOD$. \qed    
\end{corollary}

Similar to the results of \S\ref{SectExactHOD}, the results of this section reveal interesting connections between the notion of ultraexacting cardinals and Woodin's HOD Dichotomy Theorem that will be explored in  \S\ref{section6} below.

%This leads us to formulate the following conjecture:

%\begin{conjecture}\label{ConjectureSupercompact}
%The following are equiconsistent:
   % \begin{enumerate}
         %\item ZFC + ``there is an elementary embedding $j: L(V_{\lambda+1}) \to L(V_{\lambda+1})$ with critical point $\kappa$ and a supercompact cardinal $\delta \leq\kappa$.'' 
%
         %\item ZFC + ``there is a $2$-ultraexact embedding at some limit cardinal with critical point $\kappa$ and a supercompact cardinal $\delta \leq\kappa$.''
%
         %\item
%ZFC($\vec\lambda,\kappa,\delta$) + $\{``\kappa$ is parametrically $n$-ultraexact for $\vec{\lambda}$''$: n\in\mathbb{N}\}$ + $\delta \leq\kappa$ + ``$\delta$ is supercompact.''
%\end{enumerate}
%\end{conjecture}

%It is not even clear that ultraexact cardinals are simultaneously consistent with supercompact cardinals in the presence of the Axiom of Choice. 
% \todo{P.L.: I think it is not even clear if they are consistent with the existence of an element of $C^{2}$ below them.} 
% If they are not, we would obtain a very interesting example of two large cardinal notions independently consistent, yet incompatible with each other.\todo{JPA: I'm not sure how to get a model of ZFC with a supercompact and ultraexact cardinals (from any assumption). Any thoughts? Choiceless cardinals don't seem to help either: in models with e.g., Reinhardt cardinals such that $V_\lambda$ satisfies ZFC, the fact that $V_\lambda$ satisfies ZFC makes $\lambda$ very definable, so ultraexact reflection from the point of view of $V$ does not translate to $V_\lambda$.}

%%%%%%%%%%%%%%%%%%%

\subsection{Fragments of HOD-Berkeleyness}
\label{SectHODB}
We will next show that the existence of ultraexacting cardinals implies non-trivial fragments of the $\HOD$-versions of large cardinals beyond choice, considered by Koellner, Woodin and the second author in {\cite[Section 8.2]{MR4022642}}.

\begin{definition}
    Given an inner model $N$ and ordinals $\lambda<\vartheta$, the ordinal $\lambda$ is \emph{$N$-$\vartheta$-Berkeley} if for every $\alpha<\lambda$ and every transitive set $M$ in $N$ that contains the ordinal $\lambda$ as an element and has cardinality less than $\vartheta$ in $N$, there exists a non-trivial elementary embedding $\map{j}{M}{M}$ with $\alpha<\crit{j}<\lambda$.  
\end{definition}

\begin{theorem}\label{theorem:FragBerkeley}
  If $\map{i}{Y}{V_\xi}$  is a $1$-ultraexact embedding at  $\lambda$ and  $E\in V_{\lambda+2}\cap Y$ is such that $i(E)=E$, then $\lambda$ is $\HOD$-$\Theta^{L(V_{\lambda+1},E)}$-Berkeley. 
\end{theorem}

\begin{proof}
  Assume, towards a contradiction, that the above conclusion fails. Set $\Theta:=\Theta^{L(V_{\lambda+1},E)}$ and let $M\in\HOD$ be   minimal in the canonical wellordering of $\HOD$ with the property that $M$ is transitive,  $\lambda\in M$, $\betrag{M}^\HOD<\Theta$ and   for some $\alpha<\lambda$, there is no non-trivial embedding $\map{k}{M}{M}$ with $\alpha<\crit{k}<\lambda$.  Let $\gamma:=\betrag{M}^{\HOD}$ and define $\alpha$ to be the least upper bound of all critical points of non-trivial elementary embeddings $\map{k}{M}{M}$ with $\crit{k}<\lambda$. Note that our assumption implies that $\alpha<\lambda$. 
  Now, apply Lemma \ref{lemma:NewEmbeddingSurjection} to find a surjection $\map{s}{V_{\lambda+1}}{\gamma}$ and a $3$-ultraexact embedding $\map{j}{X}{V_\zeta}$ at $\lambda$ with $E,s,\gamma\in X$, $j(E)=E$, $j(\gamma)=\gamma$,   $j(s)\in X$ and $j\restriction\alpha=\id_\alpha$. In this situation, Proposition \ref{proposition:ThetaFixed} ensures that $\Theta\in X$ and $j(\Theta)=\Theta$. 
  Moreover, since $M$ is definable by a $\Sigma_3$-formula with parameters   $E$, $\lambda$ and $\Theta$, it follows that $M\in X$ and $j(M)=M$. 
  Let $b\in\HOD$ be the minimal bijection between $\gamma$ and $M$ in the canonical wellordering of $\HOD$ that has the property that $b(\omega\cdot\alpha)=\alpha$ holds for all $\alpha<\lambda$. Then, again, we know that the set  $b$ is definable by a $\Sigma_3$-formula with  parameters $E$, $\gamma$ and $\lambda$, and we know that  $b\in X$ with $j(b)=b$.  
  Let $\map{j_\gamma}{\gamma}{\gamma}$ be the function given by  Lemma \ref{lemma:ExtendWithSurjection} and define $$\map{k ~ = ~ b\circ j_\gamma\circ b^{{-}1}}{M}{M}.$$ Since both $b$ and $j_\gamma$ belong to $X$, so does $k$, and it is easily checked  that  $k\restriction\lambda=j\restriction\lambda$.  
  Now, fix a formula $\varphi(v_0,\ldots,v_{n-1})$ and ordinals $\beta_0,\ldots,\beta_{n-1}\in X\cap\gamma$ such that $\varphi(b(\beta_0),\ldots,b(\beta_{n-1}))$ holds in $M$. The elementarity of $j$ then implies that the statement $\varphi(j(b(\beta_0)),\ldots,j(b(\beta_{n-1})))$  holds in $M$. 
  Since $$j(b(\beta_m)) ~ = ~ b(j(\beta_m)) ~ = ~ b(j_\gamma(\beta_m)) ~ = ~ k(b(\beta_m))$$ holds for all $m<n$, we now know that, in $X$, the map $k$ is an elementary embedding from $M$ to $M$. It then follows that  $\map{k}{M}{M}$ is a non-trivial elementary embedding with critical point $\crit{j}$. Since $\crit{j}>\rho$, this yields a contradiction.  
\end{proof}

Let us note that recent work by Blue and Sargsyan \cite{blue2024ad} shows that a similar result holds for $\omega_1$ in determinacy models. 

\medskip

%Observe that Theorem \ref{theorem:FragBerkeley} above yields Theorem  \ref{theorem:stronglymeasurable} as a corollary. Indeed, in ZFC, if $\lambda$ is a $\HOD$-$\vartheta$-Berkeley cardinal, then every regular cardinal in the interval $(\lambda, \vartheta)$ is $\omega$-strongly measurable in $\HOD$. For let $\mu$ be a regular cardinal in the interval, and let $j:\HOD_{\mu^+}\to \HOD_{\mu^+}$ be an elementary embedding with critical point, $\delta$, below $\lambda$. \todo{P.: I am not sure if this argument works, because we do not know if $\Theta$ is an inaccessible cardinal in $\HOD$ (note that this is not $\HOD^{L(V_{\lambda+1})}$ but $\HOD^V$) and therefore I do not see how to prove the cardinality requirement in the definition of $\HOD$-$\Theta$-Berkeleyness. It might be possible to get around this problem by considering smaller structures; but this result in mostly reproving Theorem  \ref{theorem:stronglymeasurable}.} Then $(2^\delta)^\HOD <\lambda$. So if $\langle S_\alpha : \alpha <\delta\rangle \in \HOD$ is a partition of $S^\mu_\omega$ into stationary sets, then one gets a contradiction by arguing as in Woodin's proof of Kunen's inconsistency theorem.

%%%%%%%%%%%%%%%%%%%

\subsection{Definable infinitary partition properties}
\label{Sect3.4}

 Following {\cite[Section 3.2]{MR4693981}}, we shall now consider  restrictions to ordinal definable functions of infinitary partition properties that are incompatible with the Axiom of Choice.

 \begin{definition}
   Given infinite cardinals $\mu<\kappa$, the cardinal $\kappa$ is  \emph{definably $\mu$-J\'{o}nsson} if for every ordinal definable function $\map{c}{[\kappa]^\mu}{\kappa}$, there exists $H\in[\kappa]^\kappa$ with the property that $c[[H]^\mu]$ is a proper subset of $\kappa$.    
 \end{definition}

 \begin{theorem}
   If $\map{i}{Y}{V_\xi}$  is a $1$-ultraexact embedding at $\lambda$ and $E$ is an element of $V_{\lambda+2}\cap Y$ with $i(E)=E$, then every cardinal in the interval $[\lambda,\Theta^{L(V_{\lambda+1},E)})$ is definably $\rho$-J\'{o}nsson for every infinite cardinal $\rho<\lambda$.  
 \end{theorem}

 \begin{proof}
   Assume, towards a contradiction, that a  cardinal  $\lambda\leq\kappa<\Theta^{L(V_{\lambda+1},E)})$ is not $\rho$-J\'{o}nsson for some   $\omega\leq\rho<\lambda$. 
   An application of Lemma \ref{lemma:NewEmbeddingSurjection} then produces  a surjection $\map{s}{V_{\lambda+1}}{\kappa}$ and  a $2$-ultraexact embedding $\map{j}{X}{V_\zeta}$ with $E,s,\kappa\in X$, $j(E)=E$, $j(\kappa)=\kappa$, $j(s)\in X$ and $j\restriction(\rho+1)=\id_{\rho+1}$. 
   Let  $\map{c}{[\kappa]^\rho}{\kappa}$ denote the  ordinal definable function  that is minimal in the canonical wellordering of $\OD$ with the property that $c[[H]^\rho]=\kappa$ holds for every $H\in[\kappa]^\kappa$. Then the set $c$ is definable by a $\Sigma_2$-formula with  parameters  $\rho$ and $\kappa$, and it follows that $c$ is an element of $X$ with $j(c)=c$.  
   Let $\map{j_\kappa}{\kappa}{\kappa}$ be the function in $X$ given by Lemma \ref{lemma:ExtendWithSurjection} and set $H=j_\kappa[\kappa]\in X\cap[\kappa]^\kappa$.  
   Since $H$ is a proper subset of $\kappa$, our assumption and the correctness properties of $X$ yield $b\in X\cap[H]^\rho$ with $c(b)\notin H$. 
   Set $a=j_\kappa^{{-}1}[b]\in X\cap[\kappa]^\rho$. Then \eqref{equation:j_s=j_for_gamma} and fact that $j\restriction(\rho+1)=\id_{\rho+1}$  ensure that $$j(a) ~ = ~ j[a] ~ = ~ j_\kappa[a] ~ = ~ b$$ holds.  This equality now allows us to conclude that $$c(b) ~ = ~ j(c)(j(a)) ~ = ~ j(c(a)) ~ = ~ j_\kappa(c(a)) ~ \in ~ H,$$ a contradiction.  
 \end{proof}

 \begin{corollary}
     If $\lambda$ is an ultraexacting cardinal, then both $\lambda$ and $\lambda^+$ are  definably $\omega$-J\'{o}nsson. \qed 
 \end{corollary}

By combining this  corollary  with Goldberg's  {\cite[Theorem 3.5]{MR4693981}}, we obtain yet another proof   that the existence of an ultraexacting cardinal above   an extendible cardinal (or just a strongly compact cardinal)   yields the failure of Woodin's $\HOD$ Hypothesis (see \S\ref{section6} below).

\medskip

In the next section, we will show that, assuming the existence of sufficiently many sharps, many ${\rm I}0$ embeddings exist at  ultraexacting cardinals. 
 % that the existence of an ultraexact embedding yields, in the presence of sharps for subsets  of $V_{\lambda +1}$, the existence of
Complementing this, the results of   \S\ref{section3.6} will show that, given an ${\rm I}0$ embedding, a mild forcing extension yields the existence of a cardinal that is parametrically $n$-ultraexact for all $n>0$. 
%It will then follow that  the existence of an ${\rm I}0$ embedding $j:L(V_{\lambda +1})\to L(V_{\lambda +1})$,  the existence of a $2$-ultraexact embedding at $\lambda$, and the existence of a cardinal that is parametrically $n$-ultraexact for a sequence of cardinals with supremum $\lambda$, for every $n$, are all equiconsistent, \todo{P.: Check this! It seems like we loose sharps in the second implication.} modulo the existence of $V_{\lambda +1}^\sharp$  (Theorem \ref{equiconssharps} below). 

%%%%%%%

\subsection{Icarus sets from ultraexact cardinals and sharps}\label{SubsectIcarus}

In \cite{Cramer:Inverse}, \cite{Di11}, and \cite{MR3855762}, an element  $E$ of $V_{\lambda+2}$ for some limit ordinal $\lambda$ is called an \emph{Icarus set} if there exists a non-trivial elementary embedding $\map{j}{L(V_{\lambda+1},E)}{L(V_{\lambda+1},E)}$ with $\crit{j}<\lambda$. We will consider the following strengthening of this notion that is closely related to Woodin's analysis of \emph{proper} elementary embeddings provided by the proof of {\cite[Lemma 5]{MR2914848}}.

 \begin{definition}
 \label{strongIcarus}
   Given a limit ordinal $\lambda$, an element $E$ of $V_{\lambda+2}$  is a \emph{strong Icarus set} if there exists a non-trivial elementary embedding $\map{j}{L(V_{\lambda+1},E)}{L(V_{\lambda+1},E)}$ with the property that $\crit{j}<\lambda$ and $j(E)=E$.
 \end{definition}

 Note that, as shown in the proof of {\cite[Lemma 5]{MR2914848}}, the existence of such an embedding can be expressed by a first-order statement with parameter $E$. We will show next that the existence of an  ultraexact embedding at $\lambda$, together with the existence of \emph{sharps}\footnote{Recall that for a transitive set $X$, the existence of $X^\sharp$, the \emph{sharp} of $X$,  is equivalent to the existence of an elementary embedding $\map{j}{L(X)}{L(X)}$ with critical point above the rank of $X$. For more details, see, for example, {\cite[Remark 2.4]{FarmerLambda2}}, {\cite[\S 4]{So78}}, or {\cite[p.92]{La04}}.} for elements of $V_{\lambda+2}$ 
 %\todo{We should give a better reference. Is there some book/article developing the general theory of sharps? I can't find it! The only thing I could find is some notes from Andrés Caicedo, at 
 %https://andrescaicedo.wordpress.com
 %
 %but they are not published, and not complete enough.} 
  causes many elements of $V_{\lambda+2}$ to be strong Icarus sets.

\begin{theorem}\label{theorem:UltraexactSharpsIcarus}
    Let $\map{j}{X}{V_\zeta}$ be a $2$-ultraexact embedding at a cardinal $\lambda$ and let $E\in V_{\lambda+2}\cap X$ be such that  $j(E)=E$. 
    If $(V_{\lambda+1},E)^\#$ exists, then there is an elementary embedding $\map{i}{L(V_{\lambda+1},E)}{L(V_{\lambda+1},E)}$ with $i\restriction V_\lambda=j\restriction V_\lambda$ and $i(E)=E$. 
\end{theorem}

The first step towards proving Theorem \ref{theorem:UltraexactSharpsIcarus} is the following lemma:

 \begin{lemma}\label{lemma:ExtendEmbL_Theta}
  Let $\map{j}{X}{V_\zeta}$ be a $2$-ultraexact embedding at $\lambda$ and let $E\in V_{\lambda+2}\cap X$ be such that  $j(E)=E$. Set $\Theta:=\Theta^{L(V_{\lambda+1},E)}$. If $X$ contains a surjection $\map{s}{V_{\lambda+1}}{\Theta}$ with $j(s)\in X$, then there is a unique map $\map{j_E}{L_\Theta(V_{\lambda+1},E)}{L_\Theta(V_{\lambda+1},E)}$ in $X$ with 
      \begin{equation}\label{equation:j_s=j_for_L_theta}
         j\restriction(L_\Theta(V_{\lambda+1},E)\cap X) ~ = ~ j_E\restriction(L_\Theta(V_{\lambda+1},E)\cap X).   
    \end{equation}
 \end{lemma}

 \begin{proof}
    First, note that Proposition \ref{proposition:ThetaFixed} ensures that $\Theta\in X$ and $j(\Theta)=\Theta$, which implies that $L_\Theta(V_{\lambda+1},E)\in X$ and  $j(L_\Theta(V_{\lambda+1},E))=L_\Theta(V_{\lambda+1},E)$. 
    Next, recall that the  general theory of relative constructibility for $L(V_{\lambda+1},E)$ allows us to find a surjection $\map{S}{\Theta\times V_{\lambda+1}}{L_\Theta(V_{\lambda+1},E)}$ that is definable over $L_\Theta(V_{\lambda+1},E)$ by a $\Sigma_1$-formula with parameters $E$ and $V_{\lambda+1}$. The correctness of $X$  then ensures that $S\in X$, $j(S)=S$, and $S$ has the same properties in $X$. 

    Next, pick $\alpha,\beta\in X\cap\Theta$ and $x,y\in V_{\lambda+1}\cap X$ with $S(\alpha,x)=S(\beta,y)$. Then \eqref{equation:j_+=j} and \eqref{equation:j_s=j_for_gamma} imply that 
    \begin{equation*}
      \begin{split}
          S(j_\Theta(\alpha),(j\restriction V_\lambda)_+(x)) ~ & = ~ j(S)(j(\alpha),j(x)) ~ = ~ j(S(\alpha,x)) ~ = ~ j(S(\beta,y)) \\
           & = j(S)(j(\beta),j(y)) ~ = ~ S(j_\Theta(\beta),(j\restriction V_\lambda)_+(y)). 
      \end{split}
    \end{equation*} 
    This shows that, in $X$, there is a function $\map{j_E}{L_\Theta(V_{\lambda+1},E)}{L_\Theta(V_{\lambda+1},E)}$ with the property that $$j_E(S(\alpha,x)) ~ = ~ S(j_\Theta(\alpha),(j\restriction V_\lambda)_+(x))$$ holds for all $\alpha<\Theta$ and $x\in V_{\lambda+1}$. Moreover, the correctness of $X$ implies  that $j_E$ has the same property in $V$. 

    Now, fix $z\in L_\Theta(V_{\lambda+1},E)\cap X$. Since the map $S$ is, in $X$, a surjection of $\Theta\times V_{\lambda+1}$ onto $L_\Theta(V_{\lambda+1},E)$, there are $\alpha\in X\cap\Theta$ and $x\in V_{\lambda+1}\cap X$ with $S(\alpha,x)=z$. Then, we have $$j_E(z) ~ = ~ S(j_\Theta(\alpha),(j\restriction V_\lambda)_+(x)) ~ = ~ j(S)(j(\alpha),j(x)) ~ = ~ j(S(\alpha,x)) ~ = ~ j(z).$$

    Finally, assume  that there is a function $\map{k}{L_\Theta(V_{\lambda+1},E)}{L(V_{\lambda+1},E)}$ in $X$ with $k\neq j_E$ and $k\restriction(L_\Theta(V_{\lambda+1},E)\cap X) = j\restriction(L_\Theta(V_{\lambda+1},E)\cap X)$. Since $k$ and $j_E$ are both elements of $X$, we can then find $z\in L_\Theta(V_{\lambda+1},E)\cap X$ with $k(z)\neq j_E(z)$. But, this yields a contradiction, because our assumptions imply that $k(z)=j(z)=j_E(z)$.  
 \end{proof}

The proof of the next lemma uses some arguments contained in the proof of {\cite[Lemma 5]{MR2914848}}, adapted to our setting.

\begin{lemma}\label{lemma:I0fromUltraexact}
    Let $\map{j}{X}{V_\zeta}$ be a $2$-ultraexact embedding at $\lambda$ and let $E\in V_{\lambda+2}\cap X$ be such that  $j(E)=E$. Set $\Theta:=\Theta^{L(V_{\lambda+1},E)}$. If $X$ contains a surjection $\map{s}{V_{\lambda+1}}{\Theta}$ with $j(s)\in X$, then  there is an elementary embedding of $L(V_{\lambda+1},E)$ into itself that extends $j_E$.  
\end{lemma}

\begin{proof}
 In the following, set $N:=L(V_{\lambda+1},E)$ and  $D:=V^N_{\lambda+2}$. Since $D$ is contained in  $L_\Theta(V_{\lambda+1},E)=\dom{j_E}$ (see \cite[Lemma 4.6]{MR3855762}), we can now define  $$U ~ := ~ \Set{A\in D}{j\restriction V_\lambda\in j_E(A)}.$$  
 Then, since $D,j\restriction V_\lambda,j_E\in X$, we have that $U\in X$. 
 Moreover, in $V$, and therefore also in $X$, the set $U$ is an $N$-ultrafilter on $V_{\lambda+1}$. 
 Note that the fact that $\crit{j}\notin\ran{j}$ implies that $j\restriction V_\lambda\notin\ran{j}$, and therefore we know that $$j\restriction V_\lambda ~ \notin ~ j_E(\{ a\}) ~ =  ~ j(\{ a\}) ~ = ~ \{ j(a)\}$$ holds for all $a\in V_{\lambda +1}\cap X$. This allows us to conclude that $U$ is non-principal in $X$, and therefore also in  $V$.   
 Let $\Ult{N}{U}$ denote the corresponding ultrapower of $N$ by $U$, and let $\map{i}{\langle N,\in\rangle}{\langle\Ult{N}{U},\in_U\rangle}$ be the induced ultrapower embedding.

 \begin{claim*}
     Let $\map{f,g}{V_{\lambda+1}}{N}$ be functions in $N\cap X$.
     \begin{enumerate}
         \item\label{item:UltrapowerProp1}  $[f]_U=[g]_U$ if and only if $j(f)(j\restriction V_\lambda)=j(g)(j\restriction V_\lambda)$. 

         \item\label{item:UltrapowerProp2}  $[f]_U\in_U [g]_U$ if and only if $j(f)(j\restriction V_\lambda)\in j(g)(j\restriction V_\lambda)$. 
     \end{enumerate}
 \end{claim*}

 \begin{proof}[Proof of the Claim]
  \ref{item:UltrapowerProp1} Set $A=\Set{x\in V_{\lambda+1}}{f(x)=g(x)}\in D\cap X$. Then $$[f]_U=[g]_U  ~ \Longleftrightarrow ~ A\in U ~ \Longleftrightarrow ~ j\restriction V_\lambda\in j_E(A)=j(A) ~ \Longleftrightarrow ~ j(f)(j\restriction V_\lambda)=j(g)(j\restriction V_\lambda),$$
  where the equality $j_E(A) = j(A)$ follows from equation \eqref{equation:j_s=j_for_L_theta} in the statement of Lemma \ref{lemma:ExtendEmbL_Theta}.
 
   \ref{item:UltrapowerProp2} Follows similarly.
 \end{proof}

 \begin{claim*}
   The ultrapower $\langle\Ult{N}{U},\in_U\rangle$ is well-founded. 
 \end{claim*}

 \begin{proof}[Proof of the Claim]
  Assume, towards a contradiction, that $\Ult{N}{U}$ is ill-founded. Then there is a sequence of functions $\seq{\map{f_m}{V_{\lambda+1}}{N}}{m<\omega}$, all of them in $N$, such that $[f_{m+1}]_U\in_U[f_m]_U$ holds for all $m<\omega$. The correctness properties of $X$ then ensure that there is such a sequence in $X$. But in this situation,  the previous claim shows that $$j(f_{m+1})(j\restriction V_\lambda) ~ \in ~  j(f_m)(j\restriction V_\lambda)$$ holds for all $m<\omega$, a contradiction.  
 \end{proof}

\begin{claim*}
     Let $\map{f,g}{V_{\lambda+1}}{V_{\lambda+1}}$ be functions in $N$.
     \begin{enumerate}
         \item  $[f]_U=[g]_U$ if and only if $j_E(f)(j\restriction V_\lambda)=j_E(g)(j\restriction V_\lambda)$. 

         \item  $[f]_U\in_U [g]_U$ if and only if $j_E(f)(j\restriction V_\lambda)\in j_E(g)(j\restriction V_\lambda)$. 
     \end{enumerate}
 \end{claim*}

 \begin{proof}[Proof of the Claim]
  Since $f,g\in L_\Theta(V_{\lambda+1},E)$,   our first claim and \eqref{equation:j_s=j_for_L_theta} ensure that the stated equivalences hold in the case where $f$ and $g$ are elements of $X$. Therefore, the correctness of $X$ implies that they also hold in $V$. 
 \end{proof}

 Given $y\in V_{\lambda+1}$, we let $A_y$ denote the set of all I3-embeddings $\map{k}{V_\lambda}{V_\lambda}$ with $k_+(k)=j\restriction V_\lambda$ and $y\in\ran{k_+}$. Then $A_y\in D$ for all $y\in V_{\lambda+1}$ and, if $y\in V_{\lambda+1}\cap X$, then $A_y\in X$.

 \begin{claim*}
   $A_y\in U$ for all $y\in V_{\lambda+1}$.   
 \end{claim*}

 \begin{proof}[Proof of the Claim]
  Fix $y\in V_{\lambda+1}\cap X$. Then \eqref{equation:j_s=j_for_L_theta} ensures that $j_E(A_y)=j(A_y)$ is the  set of all I3-embeddings $\map{k}{V_\lambda}{V_\lambda}$ with $k_+(k)=j(j\restriction V_\lambda)$ and $j(y)\in\ran{k_+}$. Since \eqref{equation:j_+=j} implies that  $(j\restriction V_\lambda)_+(j\restriction V_\lambda)=j(j\restriction V_\lambda)$ and $(j\restriction V_\lambda)_+(y)=j(y)$, we can now conclude that $j\restriction V_\lambda\in j(A_y)$ and hence $A_y\in U$.  This shows that the statement of the claim holds in $X$, and therefore the correctness  of $X$ implies that it also holds in $V$. 
 \end{proof}

 Now, given $y\in V_{\lambda+1}$, we define 
 \begin{align*}
f_y: V_{\lambda+1} &\to V_{\lambda+1},\\ 
k& \mapsto
      \begin{cases}
	    x, & \text{if $k\in A_y$ and $k_+(x)=y$}\\
            \emptyset, & \text{otherwise.}
		 \end{cases} 
 \end{align*}
 Again, we have $f_y\in L_\Theta(V_{\lambda+1},E)\subseteq N$ for all $y\in V_{\lambda+1}$ and, if $y\in V_{\lambda+1}\cap X$, then $f_y\in X$.

 \begin{claim*}
   If $y\in V_{\lambda+1}$, then   $j_E(f_y)(j\restriction V_\lambda)=y$. 
 \end{claim*}

 \begin{proof}[Proof of the Claim]
    Fix $y\in V_{\lambda+1}\cap X$. Then, as in the proof of our previous claim, we can use \eqref{equation:j_+=j} and \eqref{equation:j_s=j_for_L_theta} to show that $j\restriction V_\lambda\in j_E(A_y)$, $(j\restriction V_\lambda)_+(y)=j(y)$. Thus, since $j_E(f_y)(j\restriction V_\lambda)=j(f_y)(j\restriction V_\lambda)$, it follows that $j(f_y)(j\restriction V_\lambda)=y$. This shows the claim holds in $X$, hence by the correctness of $X$ it also holds in $V$. 
 \end{proof}

 We now work towards showing that $i$ is an elementary embedding.

 \begin{claim*}
     If $\map{h}{V_{\lambda+1}}{N}$ is a function in $N$ with $\Set{x\in V_{\lambda+1}}{h(x)\neq\emptyset}  \in  U$, then there is a function $\map{f}{V_{\lambda+1}}{N}$ in $N$ with $[f]_U\in_U[h]_U$. 
 \end{claim*}

 \begin{proof}[Proof of the Claim]
  First, assume that $\map{g}{V_{\lambda+1}}{N}$ is a function in $N\cap X$ with the property that $A=\Set{x\in V_{\lambda+1}}{\emptyset\neq g(x)\in V_{\lambda+2}}\in U$. Since $A\in N\cap X$, we know that $j\restriction  V_\lambda\in j(A)$ and hence $$\emptyset ~ \neq ~ j(g)(j\restriction V_\lambda) ~ = ~ j_E(g)(j\restriction V_\lambda) ~ \in ~ V_{\lambda+2}\cap X.$$  Then we can find $x\in X\cap j(g)(j\restriction V_\lambda)\subseteq V_{\lambda+1}$. Earlier claims then show that $$j_E(f_x)(j\restriction V_\lambda) ~ = ~ x ~ \in ~ j(g)(j\restriction V_\lambda)=j_E(g)(j\restriction V_\lambda)$$ and thus that $[f_x]_U\in_U[g]_U$. 

  Next, assume that  $\map{h}{V_{\lambda+1}}{N}$ is an arbitrary function in $N\cap X$ with the property that $\Set{x\in V_{\lambda+1}}{h(x)\neq\emptyset}  \in  U$. The structure theory of $N=L(V_{\lambda+1},E)$ provides a class surjection $\map{S}{\On\times V_{\lambda+1}}{N}$ that is definable over $N$ by a $\Sigma_1$-formula with parameters $V_{\lambda+1}$ and $E$. 
  Define 
\begin{align*}
   g:V_{\lambda+1} &\to D,  \\ 
   y & \mapsto \Set{x\in V_{\lambda+1}}{\exists\xi\in\On ~ S(\xi,x)\in h(y)}.
\end{align*}  
We have that $g\in N\cap X$, and, since $j(h)(j\restriction V_\lambda)\ne \emptyset$, we also have that $j(g)(j\restriction V_\lambda)\ne \emptyset$, and therefore $\Set{x\in V_{\lambda+1}}{g(x)\ne \emptyset} \in U$.  Then, as before, we can  find $x\in V_{\lambda+1}\cap X$ with $[f_x]_U\in_U[g]_U$.

  Now define functions $o$ and $f$ by setting 
\begin{align*}
o : V_{\lambda+1} &\to \On\\
y &\mapsto
      \begin{cases}
	    \min\Set{\xi\in\On}{S(\xi,f_x(y))\in h(y)}, & \text{if $f_x(y)\in g(y)$}\\
            0, & \text{otherwise}; 
		 \end{cases}  
\end{align*}  
and 
\begin{align*}
f : V_{\lambda+1} &\to N \\
y &\mapsto S(o(y),f_x(y)).
\end{align*}  
Then $f\in N\cap X$, and the fact that $[f_x]_U\in_U[g]_U$ ensures that there is $\xi\in\On$ with $S(\xi,j(f_x)(j\restriction V_\lambda))\in j(h)(j\restriction V_\lambda)$. In particular, we have that $$j(f)(j\restriction V_\lambda) ~ = ~ S(j(o)(j\restriction V_\lambda),j(f_x)(j\restriction V_\lambda)) ~ \in ~ j(h)(j\restriction V_\lambda)$$ and this allows us to conclude that $[f]_U\in_U[h]_U$ holds. 

Once more, the correctness of $X$ now yields the statement of the claim.
 \end{proof}

 The previous claim allows to prove \L{}os' Theorem for $\Ult{N}{U}$:

 \begin{claim*}
   The following two statements are equivalent for every formula $\varphi(v_0,\ldots,v_{n-1})$ in the language of set theory and all functions $\map{f_0,\ldots,f_{n-1}}{V_{\lambda+1}}{N}$ in $N$: 
   \begin{enumerate}
       \item $\langle\Ult{N}{U},\in_U\rangle\models\varphi([f_0]_U,\ldots,[f_{n-1}]_U)$ 

       \item $\Set{y\in V_{\lambda+1}}{\langle N,\in\rangle\models\varphi(f_0(y),\ldots,f_{n-1}(y))}\in U$. 
   \end{enumerate}
   Hence, the map $\map{i}{\langle N,\in\rangle}{\langle\Ult{N}{U},\in_U\rangle}$ is an elementary embedding. 
 \end{claim*}

 \begin{proof}[Proof of the Claim]
The only non-trivial case is the backward implication of the $\exists$-quantifier step of the induction. So let $\varphi(v_0,\ldots,v_n)$ be a formula in the language of set theory and  let $\map{f_0,\ldots,f_{n-1}}{V_{\lambda+1}}{N}$  be functions in $N$ with the property that $$A ~ = ~ \Set{y\in V_{\lambda+1}}{\langle N,\in\rangle\models \exists z ~ \varphi(f_0(y),\ldots,f_{n-1}(y),z)} ~ \in ~ U.$$ 
  If we now define the functions $\map{o}{V_{\lambda+1}}{\On}$ by 
  \begin{equation*}
      o(y) ~ = ~       \begin{cases}
	    \min\Set{\xi\in\On}{\langle N,\in\rangle\models \exists z\in V_\xi ~ \varphi(f_0(y),\ldots,f_{n-1}(y),z)}, & \text{if $y\in A$}\\
            0, & \text{otherwise} 
		 \end{cases}  
  \end{equation*}
 and 
  \begin{align*}
h : V_{\lambda+1} &\to N\\
y &\mapsto \Set{z\in V_{o(y)}}{\langle N,\in\rangle\models\varphi(f_0(y),\ldots,f_{n-1}(y),z)},
  \end{align*} 
  then $h$ is an element of $N$ with $h(y)\neq\emptyset$ for all $y\in A$ and hence our previous claim yields a function $\map{f}{V_{\lambda+1}}{N}$ in $N$ with $[f]_U\in_U[h]_U$. 
  Now assume, towards a contradiction, that $\varphi([f_0]_U,\ldots,[f_{n-1}]_U,[f]_U)$ does not hold in $\langle\Ult{N}{U},\in_U\rangle$. Then  our induction hypothesis ensures that the set $$B ~ := ~ \Set{y\in A}{\textit{$f(y)\in h(y)$ with $\langle N,\in\rangle\models\neg\varphi(f_0(y),\ldots,f_{n-1}(y),f(y))$}}$$ is an element of $U$ and in particular nonempty. By the definition of $h$, this yields a contradiction. 
 \end{proof}

  The following claim establishes a kind of normality for $U$.

 \begin{claim*}
  If $y\in V_{\lambda+1}$, then the following statements hold: 
  \begin{enumerate}
      \item\label{item:V_lambdaInUlt1} If $x\in y$, then $[f_x]_U\in_U[f_y]_U$. 

      \item\label{item:V_lambdaInUlt2} If $\map{f}{V_{\lambda+1}}{N}$ is an element of $N$ with $[f]_U\in_U[f_y]_U$, then there is $x\in y$ with $[f]_U=[f_x]_U$. 
  \end{enumerate}
 \end{claim*}

 \begin{proof}[Proof of the Claim]
  \ref{item:V_lambdaInUlt1} If $x\in y$, then we can apply earlier claims to conclude that both $j_E(f_x)(j\restriction V_\lambda)\in j_E(f_y)(j\restriction V_\lambda)$ and $[f_x]_U\in_U[f_y]_U$ hold.

    \ref{item:V_lambdaInUlt2} Assume that $y\in V_{\lambda+1}\cap X$ and let $\map{f}{V_{\lambda+1}}{N}$ be a function in $N\cap X$ with $[f]_U\in_U[f_y]_U$. Set $x=j(f)(j\restriction V_\lambda)=j_E(f)(j\restriction V_\lambda)\in  X$. Then, by earlier claims,  $x\in j_E(f_y)(j\restriction V_\lambda)=y$ and $j_E(f_x)(j\restriction V_\lambda)=x=j_E(f)(j\restriction V_\lambda)$.  Hence, by another previous claim $[f_x]_U=[f]_U$. This shows that  \ref{item:V_lambdaInUlt2} holds in $X$, and therefore also in $V$.  
 \end{proof}

 By earlier claims, we have that the ultrapower $\langle\Ult{N}{U},\in_U\rangle$ is well-founded and extensional, so we may let $\map{\pi}{\langle\Ult{N}{U},\in_U\rangle}{\langle M,\in\rangle}$ be the corresponding transitive collapse. The previous two claims then show that $\pi([f_y]_U)=y$ holds for all $y\in V_{\lambda+1}$,  and therefore  $V_{\lambda+1}\subseteq M$. 
 We have thus established that $(\pi \circ i):L(V_{\lambda +1}, E)\to M$ is an elementary embedding with $V_{\lambda +1}\subseteq M$. So it only remains to show that $M=L(V_{\lambda +1}, E)$ and that $(\pi \circ i)$ agrees with $j_E$ on $L_\Theta (V_{\lambda +1},E)$.

 \begin{claim*}
  $(\pi\circ i)\restriction V_{\lambda+1}=j_E\restriction V_{\lambda+1}$.    
 \end{claim*}

 \begin{proof}[Proof of the Claim]
     For each $y\in V_{\lambda+1}\cap X$,  define $$B_y ~ := ~ \Set{x\in V_{\lambda+1}}{f_{j_E(y)}(x)=y}.$$ 
Note that $B_y\in D\cap X$ and 
   $$j_E(B_y) ~ = ~ j(B_y) ~ = ~ \Set{x\in V_{\lambda+1}}{j(f_{j_E(y)})(x)=j(y)}.$$ 
   Since we know that $$j(f_{j_E(y)})(j\restriction V_\lambda) ~ = ~ j_E(f_{j_E(y)})(j\restriction V_\lambda) ~ = ~ j_E(y) ~ = ~ j(y),$$ it follows that $B_y\in U$. 

   By the correctness properties of $X$, we have that $B_y\in U$, for all $y\in V_{\lambda+1}$. 
   This directly implies that $i(y)=[f_{j_E(y)}]_U$ and $$(\pi\circ i)(y) ~ = ~ \pi([f_{j_E(y)}]_U) ~ = ~ j_E(y)$$ holds for all $y\in V_{\lambda+1}$, where the second equality, as we noted above,  follows from the preceding claim.
 \end{proof}

 The claim above  implies that $(\pi\circ i)(\lambda)=j_E(\lambda)=j(\lambda)=\lambda$, hence by the  elementarity of $\pi\circ i$, we have that  $(\pi\circ i)(V_{\lambda+1})=V_{\lambda+1}$.

 \begin{claim*}
     $(\pi\circ i)\restriction D=j_E\restriction D$.    
 \end{claim*}

 \begin{proof}[Proof of the Claim]
     Fix $A\in D\cap X$ and let $c_A$ denote the function on $V_{\lambda +1}$ with constant value $A$. Then $c_A\in L_\Theta(V_{\lambda+1},E)\cap X$ and $$j_E(c_A)(j\restriction V_\lambda)=j(c_A)(j\restriction V_\lambda) =j(A)=j_E(A).$$ 
    Since an earlier claim showed that  $y=j_E(f_y)(j\restriction V_\lambda)$ holds for all $y\in V_{\lambda +1}$, we can conclude that 
    $$ y\in j_E(A)  \Longleftrightarrow ~ j_E(f_y)(j\restriction V_\lambda) \in j_E (c_A)(j\restriction V_\lambda) ~ \Longleftrightarrow ~ j(f_y)(j\restriction V_\lambda) \in j (c_A)(j\restriction V_\lambda).$$
          If, moreover, $y\in X$, then a previous claim shows that  the latter is equivalent to
                        $$[f_y]_U \in_U [c_A]_U ~ \Longleftrightarrow ~ \pi([f_y]_U)\in(\pi\circ i)(A) ~ \Longleftrightarrow ~ y\in(\pi\circ i)(A)
      $$ 
    and therefore $j_E(A)=(\pi\circ i)(A)$ holds in $X$ (as the map $(\pi\circ i)\restriction D$ is an element of $X$). The statement of the claim now follows from the correctness of $X$. 
 \end{proof}

 In particular, we now have that $(\pi \circ i)(E)=E$, and the   elementarity of $\pi\circ i$ implies that $M=L(V_{\lambda+1},E)=N$.

 \begin{claim*}
      $(\pi\circ i)\restriction \Theta=j_E\restriction \Theta$.    
 \end{claim*} 

 \begin{proof}[Proof of the Claim]
     First, note that, if $A\in D\cap X$ codes a prewellordering of $V_{\lambda+1}$ of order-type $\gamma$, then \eqref{equation:j_s=j_for_L_theta} ensures that $j_E(A)$  codes a prewellordering of $V_{\lambda+1}$ of order-type $j_E(\gamma)$. The correctness  of $X$ then yields that if $A\in D$ codes a prewellordering of $V_{\lambda+1}$ of order-type $\gamma$, then $j_E(A)$  codes a prewellordering of $V_{\lambda+1}$ of order-type $j_E(\gamma)$.
     Since $(\pi\circ i)\restriction D=j_E\restriction D$, the elementarity of $\pi\circ i$ implies the statement of the claim.  
 \end{proof}

 \begin{claim*}
   $(\pi\circ i)\restriction L_\Theta(V_{\lambda+1},E)=j_E\restriction L_\Theta(V_{\lambda+1},E)$.  
 \end{claim*}

 \begin{proof}[Proof of the Claim]
   Let $\map{S}{\Theta\times V_{\lambda+1}}{L_\Theta(V_{\lambda+1},E)}$ be the canonical surjection that is $\Sigma_1$-definable over $L_\Theta(V_{\lambda+1},E)$ from the parameters $E$ and $V_{\lambda+1}$. Using \eqref{equation:j_s=j_for_L_theta} and the correctness of $X$, we then have that $j_E(S(\alpha,x))=S(j_E(\alpha),j_E(x))$ holds for all $\alpha<\Theta$ and $x\in V_{\lambda+1}$. Since the elementarity of $\map{(\pi\circ i)}{L(V_{\lambda+1},E)}{L(V_{\lambda+1},E)}$ implies that $(\pi\circ i)(S(\alpha,x))=S((\pi\circ i)(\alpha),(\pi\circ i)(x))$ holds for all $\alpha<\Theta$ and $x\in V_{\lambda+1}$, the statement of the claim follows directly from the previous claims.  
 \end{proof}

\begin{figure}[h]
\begin{center}
\begin{tikzcd}
{L_\Theta(V_{\lambda+1},E)} \arrow[rr, "\mathrm{id}"]                   &  & {L(V_{\lambda+1},E)}                 &                                                         \\
                                                                        &  &                                      & {\mathrm{Ult}(L(V_{\lambda+1},E),U)} \arrow[lu, "\pi"'] \\
{L_\Theta(V_{\lambda+1},E)} \arrow[uu, "j_E"] \arrow[rr, "\mathrm{id}"] &  & {L(V_{\lambda+1},E)} \arrow[ru, "i"] &                                                        
\end{tikzcd}
\end{center}
\caption{An elementary embedding from $L(V_{\lambda+1},E)$ to itself extending $j_E$.}
\end{figure}

 We have thus shown that $\map{(\pi\circ i)}{L(V_{\lambda+1},E)}{L(V_{\lambda+1},E)}$ is an elementary embedding that extends $j_E$, which  
 completes the proof of  Lemma \ref{lemma:I0fromUltraexact}.
\end{proof}

We are now ready to prove the main results of this section.

\begin{proof}[Proof of Theorem \ref{theorem:UltraexactSharpsIcarus}]
    Set $F=(V_{\lambda+1},E)^\#\in V_{\lambda+2}$. Then the definability of sharps and the correctness of $X$ imply that $F\in X$ and  $j(F)=F$. 

    \begin{claim*}
     $\Theta^{L(V_{\lambda+1},E)}<\Theta^{L(V_{\lambda+1},F)}$. 
    \end{claim*}

    \begin{proof}[Proof of the Claim]
        First, note that the fact that  $L(V_{\lambda +1},E)\subseteq L(V_{\lambda +1}, F)$ implies that  $\Theta^{L(V_{\lambda+1},E)}\leq \Theta^{L(V_{\lambda+1},F)}$. Since $\Theta^{L(V_{\lambda +1},F)}$ is a regular cardinal in $L(V_{\lambda +1},F)$, and $(V_{\lambda +1},E)^\#$ exists, we can define, in $L(V_{\lambda +1},F)$, an elementary embedding  $\map{i}{L(V_{\lambda +1}, E)}{L(V_{\lambda +1}, E)}$, with critical point above $\lambda +2$ and such that $i(\crit{i})=\Theta^{L(V_{\lambda +1}, F)}$. But, by elementarity, we must also have $i(\Theta^{L(V_{\lambda +1}, E)})=\Theta^{L(V_{\lambda +1}, E)}$. This implies that 
        $\Theta^{L(V_{\lambda+1},E)}\ne \Theta^{L(V_{\lambda+1},F)}$.
    \end{proof}

  Since Lemma \ref{proposition:ThetaFixed} shows that $\Theta^{L(V_{\lambda+1},E)}\in X$ with $j(\Theta^{L(V_{\lambda+1},E)})=\Theta^{L(V_{\lambda+1},E)}$, the above  claim allows us to use Lemma \ref{lemma:ExtendBelowTheta} to find a surjection $\map{s}{V_{\lambda+1}}{\Theta^{L(V_{\lambda+1},E)}}$ in $X$ with $j(s)\in X$. Then, an application of Lemma \ref{lemma:I0fromUltraexact} yields an elementary embedding from $L(V_{\lambda+1},E)$ into itself that extends $j\restriction V_\lambda$ and fixes $E$. 
\end{proof}

Let us remark that, as shown by Woodin in \cite[Lemma 28]{MR2914848}, if the inequality $\Theta^{L(V_{\lambda+1},E)}<\Theta^{L(V_{\lambda+1},F)}$ (as in the claim above) holds for some $F\in V_{\lambda +2}$ such that $E\in  L(V_{\lambda +1}, F)$, then $(V_{\lambda +1}, E)^\#$ exists and belongs to $L(V_{\lambda +1}, F)$. Thus, the assumption that $(V_{\lambda +1}, E)^\#$ exists in the statement of Theorem \ref{theorem:UltraexactSharpsIcarus} appears to be necessary.

\medskip

Using Lemma \ref{lemma:NewEmbeddingSurjection}, Theorem \ref{theorem:UltraexactSharpsIcarus} yields the following corollary:

\begin{corollary}\label{corollary:StrongIcarusUltraexactSharps}
   Let $\map{j}{X}{V_\zeta}$ be a $1$-ultraexact embedding at $\lambda$ and let $E\in V_{\lambda+2}\cap X$ be such that  $j(E)=E$.  If $F\in\OD_{V_\lambda\cup\{E\}}\cap V_{\lambda+2}$ and  $(V_{\lambda+1},F)^\#$ exists, then $F$ is a strong Icarus set. 
\end{corollary}  

\begin{proof}
  %Let $\map{i}{Y}{V_\eta}$ be a $2$-ultraexact embedding at a cardinal $\lambda$ and let $E\in V_{\lambda+2}\cap X$ be such that $i(E)=E$. 
  Assume, towards a contradiction, that there exists $F\in\OD_{V_\lambda\cup\{E\}}$ with the property that $(V_{\lambda+1},F)^\#$ exists and $F$ is not a strong Icarus set. Pick $z\in V_\lambda$ with the property that such a set exists in $\OD_{\{E,z\}}$.  In the following, let $F$ denote the least such element in the canonical wellordering of $\OD_{\{E,z\}}$. Then the set $F$ is definable by a formula with parameters $E$ and $z$. 
  
  Let $n>1$ be a natural number such that the set $F$ is definable by a $\Sigma_n$-formula with parameters $E$ and $z$. We can now apply Lemma \ref{lemma:NewEmbeddingSurjection} to find   an $n$-ultraexact embedding $\map{i}{X}{V_\zeta}$ at $\lambda$ with $E\in X$, $j(E)=E$ and $j(z)=z$. 
  Then  $F\in X$ and $j(F)=F$. Since   $(V_{\lambda+1},F)^\#$ exists, we may now use  Theorem \ref{theorem:UltraexactSharpsIcarus} to conclude that $F$ is a strong Icarus set, a contradiction. 
\end{proof}

 In \S\ref{section3.6} below, we will show that  the consistency of the existence of an ultraexacting cardinal can be established from the existence of an I0 embedding. In contrast, in the next theorem we shall apply  Theorem \ref{theorem:UltraexactSharpsIcarus} to show that the existence of an ultraexacting cardinal $\lambda$ together with the existence of $V_{\lambda +1}^\#$  has much higher consistency strength then the existence of an I0 embedding $\map{j}{L(V_{\lambda+1})}{L(V_{\lambda+1})}$, as they implies the existence of many I0 embeddings $\map{j}{L(V_{\bar{\lambda}+1})}{L(V_{\bar{\lambda}+1})}$ with $\bar{\lambda} <\lambda$. Moreover,  while Theorem \ref{theorem:UltraexactSharpsIcarus} yields that every ultraexacting cardinal $\lambda$ below a measurable cardinal is the least non-trivial fixed point of an I0 embedding, the next theorem shows that some of the cardinals $\bar{\lambda}<\lambda$ carrying I0 embeddings $\map{j}{L(V_{\bar{\lambda}+1})}{L(V_{\bar{\lambda}+1})}$  are not ultraexacting.
 %, which shows that  if $\map{j}{L(V_{\lambda +1})}{L(V_{\lambda +1})}$ is an I0 embedding with the least possible $\lambda$, and  a measurable cardinal above $\lambda$ exists, then $\lambda$ is not an ultraexacting cardinals. In particular, if $\kappa$ is the least measurable cardinal above such a $\lambda$, then $V_\kappa$ is a model of ZFC with an I0 embedding at $\lambda$ and no ultraexacting cardinals. 
 %
 %In contrast, according to Theorem \ref{TheoremIntroConI0_2}, a $2$-ultraexact embedding in the presence of sharps yields an elementary embedding $j:L(V_{\lambda+1}^\sharp) \to L(V_{\lambda+1}^\sharp)$. By work of Cramer \cite[3.9]{Cramer:Inverse}, such a $\lambda$ is a limit of ordinals $\bar\lambda$ for which there is an non-trivial elementary embedding $\bar j: L(V_{\bar\lambda+1})\to L(V_{\bar\lambda+1})$. Thus, $2$-ultraexact cardinals become much stronger consistency-wise in the presence of sharps or other large cardinals. In particular, the existence of a $2$-ultraexact embedding at a cardinal greater than an extendible cardinal implies the existence of a proper class of I0 embeddings.

\begin{theorem}\label{corollary:LeastI0notUltraexacting}
\label{noimplication}
Suppose that $\lambda$ is an ultraexacting cardinal and $V_{\lambda+1}^\#$  exists. Then $\lambda$ is a limit of cardinals $\bar{\lambda}$ such that there is an I0 embedding $\map{j}{L(V_{\bar{\lambda}+1})}{L(V_{\bar{\lambda}+1})}$. 
 %and $X^\#$ exists for every $X\in V_{\bar{\lambda}+2}$. 
 %
In particular, the set $V_\lambda$ is a model of $\mathsf{ZFC}$ satisfying ``there is a proper class of I0 embeddings.'' %in this situation, the set $V_\lambda$ is a transitive model of ZFC in which $X^\#$ exists for every set $X$ and there is a proper class of cardinals $\bar{\lambda}$  with the property that there is an I0-embedding $\map{j}{L(V_{\bar{\lambda}+1})}{L(V_{\bar{\lambda}+1})}$. 
 %  Assume that there is a cardinal $\rho$ with the property that exists an I0-embedding $\map{j}{L(V_{\rho+1})}{L(V_{\rho+1})}$ and . 
 %If $\lambda$ is the least cardinal with the property that there is an I0-embedding $\map{j}{L(V_{\lambda+1})}{L(V_{\lambda+1})}$, then both $V_{\lambda+1}^\#$ and $(V_{\lambda+1},V_{\lambda+1}^\#)^\#$ exist and  $\lambda$ is not an ultraexacting cardinal.  
\end{theorem} 

\begin{proof}
Let $\map{j}{X}{V_\zeta}$ be a $2$-ultraexact embedding at $\lambda$.
First, observe that $V_{\lambda+1}$ is definable from $\lambda$ and thus is fixed by $j$. Similarly, $V_{\lambda+1}$ and $V_{\lambda+1}^\sharp$ are fixed by $j$. 

The idea for what follows is to try to imitate the proof of Theorem \ref{theorem:UltraexactSharpsIcarus}. While we cannot expect to obtain an elementary embedding 
\[\map{i}{L(V_{\lambda+1},V_{\lambda+1}^\sharp)}{L(V_{\lambda+1},V_{\lambda+1}^\sharp)},\]
the idea is that we can obtain very close approximations to this in the form of embeddings
\[\map{i}{L_{\alpha}(V_{\lambda+1},V_{\lambda+1}^\sharp)}{L_{\alpha}(V_{\lambda+1},V_{\lambda+1}^\sharp)}.\]
By slightly increasing $\alpha$ if necessary, a reflection argument will yield many embeddings of this form for smaller $\bar\lambda$. This will in particular yield the statement of the theorem via Cramer's work on inverse limit reflection in \cite{Cramer:Inverse}. The main ingredient is the following claim:

\begin{claim*}
There exists an ordinal $\alpha$ such that the following statements hold:
\begin{enumerate}
    \item $\Theta^{L(V_{\lambda+1})} < \alpha < \Theta^{L(V_{\lambda+1},V_{\lambda+1}^\sharp)}$,
    \item There is a non-trivial elementary embedding $\map{i}{L_{\alpha}(V_{\lambda+1},V_{\lambda+1}^\sharp)}{L_{\alpha}(V_{\lambda+1},V_{\lambda+1}^\sharp)}$.
    %\item $\Theta^{L(V_\lambda)}$ is definable from $\lambda$ in $L_{\alpha}(V_{\lambda+1},V_{\lambda+1}^\sharp)$.
\end{enumerate}
\end{claim*}

\begin{proof}[Proof of the Claim]
 First, we repeat the argument of Lemma \ref{lemma:ExtendEmbL_Theta} with $\Theta := \Theta^{L_\gamma(V_{\lambda+1}, V_{\lambda+1}^\sharp)}_{V_{\lambda+1}}$, where $\gamma$ is least such that $L_\gamma(V_{\lambda+1}, V_{\lambda+1}^\sharp)$ satisfies a sufficiently strong fragment of $\mathsf{ZFC}$, such as $\mathsf{ZC}$. By correctness and the minimality of $\gamma$, we have $\gamma \in X$ and $j(\gamma) = \gamma$. Moreover, once more by correctness, we see that there is a surjection 
\[\map{s}{V_{\lambda+1}}{\Theta}\]
in $X$ with $j(s)\in X$. 
This is enough for the argument of Lemma \ref{lemma:ExtendEmbL_Theta} to go through and yields a map \[\map{j_{V_{\lambda+1}^\sharp,\gamma}}{L_\Theta(V_{\lambda+1}, V_{\lambda+1}^\sharp)}{L_\Theta(V_{\lambda+1}, V_{\lambda+1}^\sharp)}\]
that is an element of  $X$ and agrees with $j$ on arguments in $L_\Theta(V_{\lambda+1}, V_{\lambda+1}^\sharp) \cap X$. Let $N = L_\gamma(V_{\lambda+1}, V_{\lambda+1}^\sharp)$ and $D = V_{\lambda+2}^N$. 
Using the map $j_{V_{\lambda+1}^\sharp,\gamma}$, we can now define the ultrafilter $U$ as in the proof of Lemma \ref{lemma:I0fromUltraexact} and obtain an ultrapower embedding 
$\map{i_*}{N}{\Ult{N}{U}}$ such that $\Ult{N}{U}$ is wellfounded, $i_*$ fixes $V_{\lambda+1}$ and $V_{\lambda+1}^\sharp$, and $i_*$ satisfies {\L}o\'s's Theorem. Observe that, by choice of $\gamma$, we have 
\begin{align*}
    N \models \mathsf{ZC} ~ + ~ V = L(V_{\lambda+1},V_{\lambda+1}^\sharp) ~ + ~  \forall \bar\gamma\in\On ~ ``L_{\bar\gamma}(V_{\lambda+1},V_{\lambda+1}^\sharp) \not\models \mathsf{ZC}\text{''}
\end{align*} 
and so, by elementarity, the model $\Ult{N}{U}$ satisfies this schema too. Since $N$ is the only wellfounded model of this theory which contains $V_{\lambda+1}^\sharp$, we have $N=\Ult{N}{U}$. Therefore, the map $i$ can be obtained by restricting $i_*$ to some submodel of the form $L_{\alpha}(V_{\lambda+1},V_{\lambda+1}^\sharp)$ such that $\Theta^{L(V_{\lambda+1})} < \alpha < \Theta$, so the claim follows.\footnote{Although not necessary for the rest of the argument, the proof of the claim could be modified to yield arbitrarily large such $\alpha < \Theta^{L(V_{\lambda+1},V_{\lambda+1}^\sharp)}$ by replacing $\gamma$ with the least ordinal such that $L_\gamma(V_{\lambda+1},V_{\lambda+1}^\sharp)$ satisfies some given theory, and then applying a $\Sigma_1$-reflection argument in $L(V_{\lambda+1},V_{\lambda+1}^\sharp)$, using the fact that the existence of such embeddings $i$ is absolute to $L(V_{\lambda+1},V_{\lambda+1}^\sharp)$ (this is by an argument of Laver \cite{Lav01} and Woodin; see {\cite[Theorem 6.40]{MR3855762})}.}
\end{proof}

The theorem now follows from the above claim, as Cramer \cite[Theorem 3.9]{Cramer:Inverse} has shown that even the existence of an elementary embedding 
\[\map{i}{L_\omega(V_{\lambda+1}, V_{\lambda+1}^\sharp)}{L_\omega(V_{\lambda+1},V_{\lambda+1}^\sharp)}\]
implies the existence of unboundedly many $\bar\lambda < \lambda$ for which there is an elementary embedding 
\[\map{k}{L(V_{\bar\lambda+1})}{L(V_{\bar\lambda+1})},\]
as desired.
\end{proof}

Let us observe  that the argument above also yields the following result:

\begin{corollary}\label{CorollaryIntroConI0_2}
 If there is an ultraexacting cardinal and $X^\#$ exists for every set $X$, then there exists a transitive set-sized model of the theory $$\mathrm{ZFC} ~ + ~ \textit{``There is an ordinal $\lambda$ such that $V_{\lambda+1}^\#$ exists and is a strong Icarus set"}.$$
 %The theory $\mathsf{ZFC}$ + ``there is a $2$-ultraexact embedding at $\lambda$'' + ``$X^\#$ exists for every set $X$'' proves the consistency of $\mathsf{ZFC} +$ ``There is an elementary embedding $j: L(V_{\lambda+1}^\sharp) \to L(V_{\lambda+1}^\sharp)$ with critical point $\kappa<\lambda$".
\end{corollary}

\begin{proof}
  By Theorem \ref{corollary:LeastI0notUltraexacting} and the proof of {\cite[Lemma 5]{MR2914848}},     our assumptions allow us to find a cardinal $\lambda$ and a filter $U$ on $V_{\lambda+2}^{L(V_{\lambda+1},V_{\lambda+1}^\#)}$ that induces a      non-trivial elementary embedding $\map{j}{L(V_{\lambda+1},V_{\lambda+1}^\#)}{L(V_{\lambda+1},V_{\lambda+1}^\#)}$ with critical point below $\lambda$.  Then we can pick a set $X$ such that $U,V_{\lambda+1},V_{\lambda+1}^\#\in L(X)$, the class $L(X)$ is a model of ZFC and, in $L(X)$, the filter $U$ induces a non-trivial elementary embedding of $L(V_{\lambda+1},V_{\lambda+1}^\#)$ into itself with critical point below $\lambda$.    Since $X^\#$ exists, we can find an ordinal $\alpha>\lambda$ such that  all of these statements hold with respect to  $L_\alpha(X)$.  
\end{proof}

Analogously to the observations made after Lemma \ref{lemma:Equivalentexact} above, item \ref{item:EquivalentUltra2} in Lemma \ref{lemma:EquivalentUltra} shows that, if $\lambda$ is the least ultraexacting cardinal and $\kappa$ is the least inaccessible cardinal above $\lambda$, then, in $V_\kappa$, the cardinal $\lambda$ is ultraexacting, there are no inaccessible cardinals greater than $\lambda$ and the class $C^{(3)}$ does not contain regular cardinals. In particular, there are no extendible cardinals in $V_\kappa$. 
%
%As we observed before in the case of $2$-exact embeddings (see the paragraph before \S \ref{subsec3.1}), if $V$ is a model of ZFC with a $2$-ultraexact embedding at $\lambda$, with $\lambda$ being the least such, and there is an inaccessible cardinal $\kappa$ above $\lambda$, then  $V_\kappa$ is a model of ZFC with a $2$-ultraexact embedding at $\lambda$ and with no regular cardinal in $C^{(3)}$, hence with no inaccessible cardinals. 
 %
 In the other direction, we can use the results of this section to show that the existence of an ultraexacting cardinal above an extendible cardinal turns out to have very strong consequences:

 \begin{corollary}
    If $\delta$ is an extendible cardinal below an  ultraexacting cardinal, then $\delta$ is a limit of ultraexacting cardinals and there is a proper class of cardinals $\lambda$  with the property that $V_{\lambda+1}^\#$ exists and is a strong Icarus set. 
 \end{corollary}

 \begin{proof}
  Since $\delta$ is an element of $C^{(3)}$, Lemma \ref{lemma:EquivalentUltra} directly implies that $\delta$ is a limit of ultraexacting cardinals. Using Corollary \ref{corollary:StrongIcarusUltraexactSharps}, the proof of {\cite[Lemma 5]{MR2914848}} and the fact that $X^\#$ exists for every set $X$, we know that for every $\gamma<\delta$, there exists a measurable cardinal $\gamma<\kappa<\delta$, a cardinal $\gamma<\lambda<\kappa$ and a filter $U$ on $V_{\lambda+2}^{L(V_{\lambda+1},V_{\lambda+1}^\#)}$ with the property that, in $V_\kappa$, the filter  $U$ induces an non-trivial embedding $\map{j}{L(V_{\lambda+1},V_{\lambda+1}^\#)}{L(V_{\lambda+1},V_{\lambda+1}^\#)}$ with critical point below $\lambda$. Using the extendibility of $\delta$, this implies that there is a proper class of cardinals $\lambda$ with the property that there exists a measurable cardinal $\kappa>\lambda$ and a filter $U$ on $V_{\lambda+2}^{L(V_{\lambda+1},V_{\lambda+1}^\#)}$ that induces a non-trivial embedding $\map{j}{L(V_{\lambda+1},V_{\lambda+1}^\#)}{L(V_{\lambda+1},V_{\lambda+1}^\#)}$  in $V_\kappa$ and, by forming $\On$-many iterated ultrapowers of $V$ using   normal ultrafilters on the given measurable cardinals. This implies the conclusion of the corollary. 
 \end{proof}

 %Now notice that by Lemma \ref{lemma:EquivalentUltra} (ii), the  statement that there is a $2$-ultraexact embedding $j:X\to V_\zeta$ at some cardinal $\lambda$ greater than $\eta$. % , with $\zeta \in C^{(2)}$, is $\Sigma_3$ in the parameter $\eta$, and so it also reflects to  $V_\delta$. 
%Therefore, having a 2-ultraexact embedding at a cardinal $\lambda$ with an extendible cardinal $\delta$ below $\lambda$ yields (via the proof of Theorem \ref{theorem:UltraexactSharpsIcarus}, and since an extendible cardinal implies the existence of sharps for all transitive sets)  a %model $V_\delta$ of ZFC with proper class-many 
%proper class of I0-embeddings. 

We close this section by showing that, in the presence of sharps, cardinals $\kappa$ that are  parametrically $n$-ultraexact cardinals for a sequence $\vec{\lambda}$ yield  I0-embeddings with $\vec{\lambda}$ as the critical sequence above $\kappa$.

\begin{theorem}
  Let $n>1$ be a natural number and let $\vec{\lambda}=\seq{\lambda_m}{m<\omega}$ be a strictly increasing sequence of cardinals with supremum $\lambda$. Assume $\kappa$ is parametrically $n$-ultraexact for $\vec{\lambda}$, and $E$ is an element of $V_{\lambda+1}$ that is definable by a $\Sigma_n$-formula with parameters in $V_\kappa\cup\{\lambda\}$. If $(V_{\lambda +1},E)^\#$ exists,   then there is a non-trivial elementary embedding $$\map{i}{L(V_{\lambda+1},E)}{L(V_{\lambda+1},E)}$$ with $i(E)=E$, $i(\crit{i})=\kappa$, $i(\kappa)=\lambda_0$ and $i(\lambda_m)=\lambda_{m+1}$ for all $m<\omega$.  
\end{theorem}

\begin{proof}
    Pick an element $z$ of $V_\kappa$ such that $E$ is definable by a $\Sigma_n$-formula with parameters $\lambda$ and $z$. 
    As $\kappa$ is parametrically $n$-ultraexact for $\vec{\lambda}$, we have an $n$-ultraexact embedding $\map{j}{X}{V_\zeta}$ for $\lambda$, with $z\in\ran{j}$, $j(\crit{j})=\kappa$,  $j(\kappa)=\lambda_0$ and $j(\lambda_m)=\lambda_{m+1}$ for all $m<\omega$. Note that since $z\in V_{\crit{j}}$, we have $j(z)=z$. 
    Moreover, by the correctness of $X$ we also have  that $E\in X$ and $j(E)=E$. An application of Theorem  \ref{theorem:UltraexactSharpsIcarus} now yields an elementary embedding $\map{i}{L(V_{\lambda+1},E)}{L(V_{\lambda+1},E)}$ with $i(E)=E$ and $i\restriction V_\lambda ~ = ~ j\restriction V_\lambda$. Hence, $i(\crit{i})=\kappa$, $i(\kappa)=\lambda_0$ and $i(\lambda_m)=\lambda_{m+1}$ for all $m<\omega$.  
\end{proof}

%%%%%

\subsection{The consistency of ultraexact cardinals}
\label{section3.6}

In  this section, we establish the consistency of ultraexact cardinals from the consistency of I0-embeddings. This gives the 
remaining implication in our equiconsistency result, which in particular implies that the consistency strength of $m$-ultraexactness does not depend on $m$. The following argument should be compared with {\cite[Lemma 190]{WOEM1}}. 
Theorem \ref{theorem:UltraexactFromI0} below and other similar results should be regarded as a theorem schema indexed by $n > 0$.

\begin{theorem}\label{theorem:UltraexactFromI0}
  Assume that for some cardinal $\lambda$ and $E \in V_{\lambda+2}$, there is a non-trivial  elementary embedding 
\[\map{j}{L(V_{\lambda+1},E)}{L(V_{\lambda+1},E)}\] 
 with critical sequence $\seq{\lambda_m}{m<\omega}$ cofinal in $\lambda$ and $j(E)=E$. If $G$ is $\Add{\lambda^+}{1}$-generic over $V$, then $L(V_{\lambda+1},E,G)$ is a model of ZFC and,  in $L(V_{\lambda+1},E,G)$, for every natural number $n>0$ and every $A\in V_{\lambda+1}$, there is an $n$-ultraexact embedding $\map{i}{X}{V_\zeta}$ at $\lambda$ with $E\in X$, $i(E)=E$, $A\in\ran{i}$, $i(\crit{i})=\lambda_0$ and $i(\lambda_m)=\lambda_{m+1}$ for all $m<\omega$.   
\end{theorem}

\begin{proof}
  First, note that $\Add{\lambda^+}{1}\subseteq L(V_{\lambda+1},E)$ and hence $G$ is also $\Add{\lambda^+}{1}$-generic over $L(V_{\lambda+1},E)$. Moreover, since the partial order $\Add{\lambda^+}{1}$ is ${<}\lambda^+$-closed in  $V$, it follows that $V$ and $V[G]$ contain the same $\lambda$-sequences of elements of $V$, and $L(V_{\lambda+1},E)$ and $L(V_{\lambda+1},E,G)$ contain the same $\lambda$-sequences of elements of $L(V_{\lambda+1},E)$.  By genericity, the filter  $G$ codes a wellordering of $P(\lambda)$ of order-type $\lambda^+$, and  it follows that $V_{\lambda+1}$ and $E\subseteq V_{\lambda+1}$ can be wellordered in $L(V_{\lambda+1},E,G)$, which thus satisfies the Axiom of Choice. In addition, these observations show that $H_{\lambda^+}^{L(V_{\lambda+1},E,G)}\subseteq L(V_{\lambda+1},E)$.

  \begin{claim*}
      In $L(V_{\lambda+1},E,G)$, if $\zeta>\lambda$ is a cardinal and $X$ is an elementary submodel of $V_\zeta$ of cardinality $\lambda$ with $V_\lambda\cup\{\lambda,E\}\subseteq X$, then there is an elementary embedding $\map{i}{X}{V_{j(\zeta)}}$ with $i\restriction V_\lambda=j\restriction V_\lambda$ and $i(E)=E$. 
  \end{claim*}

  \begin{proof}[Proof of the Claim]
    Let $\map{\pi}{X}{M}$ be the transitive collapse of $X$ and define $E_0=\pi(E)\in M$. 
    Then the inverse map $\map{\pi^{{-}1}}{M}{V_\zeta}$ is an elementary embedding with $\pi^{{-}1}\restriction V_\lambda=\id_{V_\lambda}$ and  
    $M\in H_{\lambda^+}^{L(V_{\lambda+1},E,G)}\subseteq L(V_{\lambda+1},E)$.  Moreover, the fact that  $\Add{\lambda^+}{1}$ is weakly homogeneous in $L(V_{\lambda+1},E)$  implies that 
    \begin{equation*}
       \begin{split}
           \mathbbm{1}_{\Add{\lambda^+}{1}}\Vdash ``&\textit{There is an elementary embedding $\map{k}{\check{M}}{V_{\check{\zeta}}}$} \\ 
           & ~ \textit{satisfying $k\restriction V_{\check{\lambda}}=\id_{V_{\check{\lambda}}}$ and $k(\check{E}_0)=\check{E}$ }"
       \end{split} 
    \end{equation*} 
    holds in $L(V_{\lambda+1},E)$. Since $j(\lambda)=\lambda$, $j(E)=E$ and $j(\Add{\lambda^+}{1})=\Add{\lambda^+}{1}$,  an application of $j$ ensures that,  in $L(V_{\lambda+1},E,G)$, there is an elementary embedding $\map{k}{j(M)}{V_{j(\zeta)}}$ with $k\restriction V_\lambda=\id_{V_\lambda}$ and $k(j(E_0))=E$. 
    But, this implies that the embedding $$\map{i ~ = ~ k\circ (j\upharpoonright M)\circ\pi}{X}{V_{j(\zeta)}^{L(V_{\lambda+1},E,G)}}$$ given by the  diagram below is an elementary embedding with $i\restriction V_\lambda=j\restriction V_\lambda$ and $i(E)=E$.  
    \begin{figure}[h]
    \begin{center}
    \begin{tikzcd} 
M \arrow[rr, "j\upharpoonright M"]                                        &  & j(M) \arrow[rr, "k"] &  & V_{j(\zeta)}^{L(V_{\lambda+1},E,G)} \\
                                                                          &  &                      &  &              \\
X \arrow[uu, "\pi"] \arrow[rrrruu, "i", dashed] \arrow[rr, "\mathrm{id}"] &  & V_\zeta^{L(V_{\lambda+1},E,G)}              &  & 
\end{tikzcd}      
    \end{center}
    \end{figure}         

    Finally, since the elementary embedding $\map{j\restriction M}{M}{j(M)}$ belongs to $L(V_{\lambda+1},E)$,  the map $i$ is an element of $L(V_{\lambda+1},E,G)$. This proves the claim.
  \end{proof}

  Now, fix $n > 0$ and assume, aiming for a contradiction, that there is $A\in V_{\lambda+1}$ with the property that, in $L(V_{\lambda+1},E,G)$, there is no  $n$-ultraexact embedding $\map{k}{X}{V_\zeta}$ at $\lambda$ with $E\in X$, $k(E)=E$, $A\in\ran{k}$, $k(\crit{k})=\lambda_1$ and $k(\lambda_{m+1})=\lambda_{m+2}$ for all $m<\omega$.  
  Fix a cardinal $\zeta>\lambda$ such that $j(\zeta)=\zeta$ and $V_\zeta^{L(V_{\lambda+1},E,G)}$ is sufficiently elementary in $L(V_{\lambda+1},E,G)$. 
  Now, work in $L(V_{\lambda+1},E,G)$ and pick an elementary submodel $X$ of $V_\zeta$ of cardinality $\lambda$ with $V_\lambda\cup\{E,j\restriction V_\lambda\}\subseteq X$. The above claim then yields an elementary embedding $\map{i}{X}{V_\zeta}$ with $i\restriction V_\lambda=j\restriction V_\lambda$ and $i(E)=E$. 
  Since $V_\zeta$ was chosen to be sufficiently elementary in $V$ and the set $i[X]$ is an elementary submodel of $V_\zeta$ that contains $E$ and the sequence $\seq{\lambda_{m+1}}{m<\omega}$, our assumption and the correctness properties of $V_\zeta$ yield $A\in i[X]$ with the property that there is no  $n$-ultraexact embedding $\map{k}{X}{V_\zeta}$ at $\lambda$ with $E\in X$, $k(E)=E$, $A\in\ran{k}$, $k(\crit{k})=\lambda_1$ and $k(\lambda_{m+1})=\lambda_{m+2}$ for all $m<\omega$.  
  However, the fact that $i\restriction V_\lambda=j\restriction V_\lambda\in X$ implies that $\map{i}{X}{V_\zeta}$ is an $n$-ultraexact embedding at $\lambda$ with  $A\in\ran{i}$, $i(\crit{i})=j(\crit{j})=j(\lambda_0)=\lambda_1$ and $i(\lambda_{m+1})=j(\lambda_{m+1})=\lambda_{m+2}$ for all $m<\omega$, a contradiction.   

  We have thus shown that, in $L(V_{\lambda+1},E,G)$, for every $A\in V_{\lambda+1}$,    there is an  $n$-ultraexact embedding $\map{i}{X}{V_\zeta}$ at $\lambda$ with $E\in X$, $i(E)=E$, $A\in\ran{i}$, $i(\crit{i})=\lambda_1$ and $i(\lambda_{m+1})=\lambda_{m+2}$ for all $m<\omega$.  Note that, in  $L(V_{\lambda+1},E,G)$, this statement can be expressed by a formula with parameters $E$ and $\seq{\lambda_{m+1}}{m<\omega}$. Since these parameters are contained in $L(V_{\lambda+1},E)$, the weak homogeneity of $\Add{\lambda^+}{1}$ in $L(V_{\lambda+1},E)$ implies that every condition in this partial order forces this statement to hold.   In this situation, the elementarity of $j$ ensures that every condition in $\Add{\lambda^+}{1}$ forces that for every $A\in V_{\lambda+1}$,    there is an  $n$-ultraexact embedding $\map{i}{X}{V_\zeta}$ at $\lambda$ with $E\in X$, $i(E)=E$, $A\in\ran{i}$, $i(\crit{i})=\lambda_0$ and $i(\lambda_m)=\lambda_{m+1}$ for all $m<\omega$. So, in particular, this statement holds in $L(V_{\lambda+1},E,G)$. 
\end{proof}

By applying the above theorem to the case $E=\emptyset$, we directly obtain the following corollary:

\begin{corollary}
 If $\map{j}{L(V_{\lambda+1})}{L(V_{\lambda+1})}$ is an I0-embedding with critical sequence $\seq{\lambda_m}{m<\omega}$ and $G$ is $\Add{\lambda^+}{1}$-generic over $V$, then, in $L(V_{\lambda+1},E,G)$, the cardinal $\lambda_0$ is parametrically $n$-ultraexact for $\seq{\lambda_{m+1}}{m<\omega}$, for all  $n>0$. \qed 
\end{corollary}

While the combination of Theorems  \ref{theorem:UltraexactSharpsIcarus} and  \ref{theorem:UltraexactFromI0} does not directly give us an equiconsistency statement, we end this section by discussing how these results can be used to obtain results in this direction. 
 For this purpose, we start by inductively defining  for all natural numbers $n$ and cardinals $\lambda$, when a set $E$ is the \emph{$n$-th iterated sharp of $V_{\lambda+1}$}. First, a set $E$ is the $0$-th iterated sharp of $V_{\lambda+1}$ if and only if $E=\emptyset$. Next, a set $E$ is the $(n+1)$-th iterated sharp of $V_{\lambda+1}$ if there exists a set $D$ that is the $n$-th iterated sharp of $V_{\lambda+1}$ and satisfies $E=(V_{\lambda+1},D)^\#$. Notice that, if the $n$-th iterated sharp of $V_{\lambda+1}$ exists, then it is uniquely determined, it is contained in $\OD\cap V_{\lambda+2}$ and the $m$-th iterated sharp of $V_{\lambda+1}$ exists for all $m<n$. 
 Corollary  \ref{corollary:StrongIcarusUltraexactSharps} now shows that if the $(n+1)$-th iterated sharp of $V_{\lambda+1}$ exists for some ultraexacting cardinal $\lambda$, then the  $n$-th iterated sharp of $V_{\lambda+1}$ is a strong Icarus set. In the other direction, Theorem \ref{theorem:UltraexactFromI0} shows that if $\lambda$ is a cardinal with the property that  the $n$-th iterated sharp of $V_{\lambda+1}$ exists and is a strong Icarus set, then, in an inner model of a generic extension of $V$, the cardinal $\lambda$ is ultraexacting and the $n$-th iterated sharp of $V_{\lambda+1}$ exists. 
 These observations show that the following statements are equivalent: 
 \begin{enumerate}
     \item For every natural number $n$, the axioms of ZFC are consistent with the existence of a cardinal $\lambda$ with the property that the  $n$-th iterated sharp of $V_{\lambda+1}$ is a strong Icarus set. 

     \item For every natural number $n$, the axioms of ZFC  are consistent with the existence of an ultraexacting cardinal for which the $n$-th iterated sharp of $V_{\lambda+1}$ exists. 
 \end{enumerate}

\section{Square Root Reflection and Ultraexact Cardinals}\label{SectSSR}

 In this section, we show that the existence of ultraexact cardinals  corresponds exactly with the validity of very strong structural reflection principles. We start by reviewing the principles of \emph{Exact Structural Reflection $\ESR_\Ce(\vec{\lambda})$} introduced in \cite{BL} and their relationship to the existence of exact cardinals. 
 %Below, the \textit{rank} of a structure is defined as the cardinality of its universe.

 \begin{definition}[\cite{BL}]
 \label{defESR}
 Let $\calL$ be a first-order language containing unary predicate symbols $\vec{P}=\seq{\dot{P}_m}{m<\omega}$. 
 \begin{enumerate}
     \item Given a sequence $\vec{\lambda}=\seq{\lambda_m}{m<\omega}$ of cardinals with supremum $\lambda$, an $\calL$-structure $A$ has \emph{type $\vec{\lambda}$ (with respect to $\vec{P}$)} if the universe  of $A$ has rank $\lambda$ and $\rank{\dot{P}_m^{A}}=\lambda_m$ for all $m<\omega$. 
     
     \item Given a class $\Ce$ of $\calL$-structures and  a strictly increasing sequence $\vec{\lambda}=\seq{\lambda_m}{m<\omega}$ of cardinals, we let $\ESR_{\Ce}(\vec{\lambda})$ denote the statement that for every structure $B$ in $\Ce$ of type  $\seq{\lambda_{m+1}}{m<\omega}$, there exists an elementary embedding of a structure $A$ in $\Ce$ of type $\seq{\lambda_m}{m<\omega}$ into $B$. 
     
     \item Given a definability class $\Gamma$ and a class $P$, we let $\Gamma(P)$-$\ESR(\vec{\lambda})$ denote the statement that $\ESR_\Ce (\vec{\lambda})$ holds for every class $\Ce$ of structures of the same type that is $\Gamma$-definable with parameters in $P$.  
 \end{enumerate} 
\end{definition}

The following statements are a direct consequence of  {\cite[Corollary 9.10]{BL}} and its proof (just letting $\rho=\lambda$ in the proof of the implication (ii) $\Rightarrow$  (i) in \cite[Lemma 9.9]{BL}). In the original result, the cardinal $\lambda$ does not appear as a parameter in the definition of the classes of structures, but it is easily seen that the given proof also works if we allow this cardinal as a parameter.

\begin{theorem}[\cite{BL}]\label{theorem:ESRcorrespondence}
 Let $n>0$ be a natural number and let $\vec{\lambda}=\seq{\lambda_m}{m<\omega}$ be a strictly increasing sequence of cardinals with supremum $\lambda$. 
  \begin{enumerate}
   \item The cardinal $\lambda_0$ is  $n$-exact for $\seq{\lambda_{m+1}}{m<\omega}$ if and only if  $\Sigma_{n+1}(\{\lambda\})$-$\ESR(\vec{\lambda})$ holds. 
   
   \item If $\lambda_0$ is  parametrically $n$-exact for $\seq{\lambda_{i+1}}{i<\eta}$, then $\Sigma_{n+1}(V_{\lambda_0}\cup\{\lambda\})$-$\ESR(\vec{\lambda})$ holds.  
  \end{enumerate}
\end{theorem}

 We now strengthen the above reflection principles by permuting the quantifiers in their statements to obtain principles of structural reflection that directly correspond to ultraexact cardinals. These stronger reflection principles are based on the concept of \emph{square roots} of I3-embeddings that play an important role in the study of rank-into-rank embeddings (see \cite[Section 2]{MR1489305}).

\begin{definition}
 Given a limit ordinal $\lambda$  and a function $\map{f}{V_\lambda}{V_\lambda}$, a \emph{square root of $f$} is a function $\map{r}{V_\lambda}{V_\lambda}$ with $r_+(r)=f$.\footnote{The function $\map{r_+}{V_{\lambda+1}}{V_{\lambda+1}}$ is defined in Definition \ref{DefinitionFPlus}.} 
\end{definition}
 
 %Below, we recall from \S \ref{SectUltraExactJonsson} that $r_+:V_{\lambda +1}\to V_{\lambda+1}$ is defined by $r_+(x)=\bigcup \{ r(x\cap V_\alpha):\alpha <\lambda\}$ and that a square root of $\map{f}{V_\lambda}{V_\lambda}$ is a function $\map{r}{V_\lambda}{V_\lambda}$ with $r_+(r)=f$. 

 \begin{definition}
 \label{DefSRESR}
   \begin{enumerate}
       \item Given  a first-order language $\calL$ containing unary predicate symbols $\seq{\dot{P}_m}{m<\omega}$, and given a class $\Ce$ of $\calL$-structures together with  a strictly increasing sequence $\vec{\lambda}=\seq{\lambda_m}{m<\omega}$ of cardinals with supremum $\lambda$, 
       we let $\sqrt{\ESR}_\Ce(\vec{\lambda})$ (the \emph{Square Root Exact Structural Reflection principle for $\vec{\lambda}$ and $\Ce$}) denote the statement that there is a function $\map{f}{V_\lambda}{V_\lambda}$ with the property that  for every structure $B$ in $\Ce$ of type  $\seq{\lambda_{m+1}}{m<\omega}$, there exists a structure $A$ in $\Ce$ of type $\seq{\lambda_m}{m<\omega}$ and a square root $r$ of $f$ such that the restriction of $r$ to the universe of $A$ is  an elementary embedding of $A$ into $B$. 

        \item Given a definability class $\Gamma$ and a class $P$, we let $\Gamma(P)$-$\sqrt{\ESR}(\vec{\lambda})$ denote the statement that $\sqrt{\ESR}_\Ce (\vec{\lambda})$ holds for every class $\Ce$ of structures of the same type that is $\Gamma$-definable with parameters in $P$.  
   \end{enumerate}  
 \end{definition}

Clearly, for every sequence $\vec{\lambda}$ and every class $\Ce$, the principle $\sqrt{\ESR}_\Ce (\vec{\lambda})$ implies the principle $\ESR_\Ce (\vec{\lambda})$. 
In the following, we shall prove the analogs of Theorem \ref{theorem:ESRcorrespondence} for the \emph{square root Structural Reflection principle $\sqrt{\ESR}$}.

\begin{lemma}\label{lemma:UltraexactEmbeddingsSQSR}
  Let $n>0$ be a natural number, let $\vec{\lambda}=\seq{\lambda_m}{m<\omega}$ be a strictly increasing sequence of cardinals with supremum $\lambda$, and let $\map{j}{X}{V_\zeta}$ be an $n$-ultraexact embedding at $\lambda$ with $\vec{\lambda}\in X$, $\lambda_0\in\ran{j}$ and $j(\lambda_m)=\lambda_{m+1}$ for all $m<\omega$. 
    If    $F=\Set{x\in X}{j(x)=x}$, then the principle $\Sigma_{n+1}(F)$-$\sqrt{\ESR}(\vec{\lambda})$ holds. 
\end{lemma}

\begin{proof}
 Assume, towards a contradiction, that there is a class $\Ce$ of structures of the same type such that $\sqrt{\ESR}_\Ce(\vec{\lambda})$ fails for $\Ce$ and $\Ce$ is definable by a $\Sigma_{n+1}$-formula $\varphi(v_0,v_1)$ and a parameter $z\in F$. 
 Since the map $\map{j(j\restriction V_\lambda)}{V_\lambda}{V_\lambda}$ does not witness that $\sqrt{\ESR}_\Ce(\vec{\lambda})$ holds, we can find  a structure $B\in\Ce$ of type $\seq{\lambda_{m+1}}{m<\omega}$ with the property that for every structure $A\in\Ce$ of type $\vec{\lambda}$, there is no square root $r$ of $j(j\restriction V_\lambda)$ such that the restriction of $r$ to the universe of $A$ is  an elementary embedding of $A$ into $B$.  
 Since $\zeta\in C^{(n+1)}$ by the definition of $n$-ultraexactness, and since $B\in  V_\zeta$, we have that, in $V_\zeta$,  $\varphi(B,z)$ holds and  for every structure $A$  of the given signature and  type $\vec{\lambda}$ such that $\varphi(A,z)$ holds, there is no square root $r$ of $j(j\restriction V_\lambda)$ such that the restriction of $r$ to the universe of $A$ is  an elementary embedding of $A$ into $B$. 
 Pick $\mu_0\in X$ with $j(\mu_0)=\lambda_0$ and set $\mu_{m+1}=\lambda_m$ for all $m<\omega$. Then $\vec{\mu}=\seq{\mu_m}{m<\omega}\in X$ and $j(\vec{\mu})=\vec{\lambda}$. 
 The elementarity of $j$ then implies that, in $X$, there is a structure $A$ of the given signature and type $\vec{\lambda}$  such that $\varphi(A,z)$ holds and  for every structure $A_0$  of the given signature and  type $\vec{\mu}$ such that $\varphi(A_0,z)$ holds, there is no square root $r$ of $j\restriction V_\lambda$ such that the restriction of $r$ to the universe of $A_0$ is  an elementary embedding of $A_0$ into $A$. 
Applying $j$, we see that, in $V_\zeta$, $A$ is a structure in $\Ce$ of type $\vec{\lambda}$, $j(A)$ is a structure in $\Ce$ of type $\seq{\lambda_{m+1}}{m<\omega}$ and there is no square root $r$ of $j(j\restriction V_\lambda)$ such that the restriction of $r$ to the universe of $A$ is  an elementary embedding of $A$ into $j(A)$. By the correctness of $V_\zeta$ this holds in $V$ as well.
However, this is a contradiction, as $j\restriction V_\lambda$ is a square root of $j(j\restriction V_\lambda)$  and the restriction of $j\restriction V_\lambda$  to the universe of $A$ is an elementary embedding of $A$ into $j(A)$.  
\end{proof}

\begin{corollary}
\label{coro4.5}
 Let $n>0$ be a natural number and let  $\vec{\lambda}=\seq{\lambda_m}{m<\omega}$ be a strictly increasing sequence of cardinals with supremum $\lambda$.
    \begin{enumerate}
        \item\label{item:SQSRfromultraexact1} If $\lambda_0$ is $n$-ultraexact for $\seq{\lambda_{m+1}}{m<\omega}$, then $\Sigma_{n+1}(\{\lambda\})$-$\sqrt{\ESR}(\vec{\lambda})$ holds.

        \item\label{item:SQSRfromultraexact2} If $\lambda_0$ is parametrically $n$-ultraexact for $\seq{\lambda_{m+1}}{m<\omega}$, then $\Sigma_{n+1}(V_{\lambda_0}\cup\{\lambda\})$-$\sqrt{\ESR}(\vec{\lambda})$ holds. 
    \end{enumerate}
\end{corollary}
\begin{proof}
    \ref{item:SQSRfromultraexact1} If the cardinal $\lambda_0$ is $n$-ultraexact for $\seq{\lambda_{m+1}}{m<\omega}$, then there is an $n$-ultraexact embedding $\map{j}{X}{V_\zeta}$ at $\lambda$ with $\lambda_0,\vec{\lambda}\in\ran{j}$  and $j(\lambda_m)=\lambda_{m+1}$ for all $m<\omega$. Then $\vec{\lambda}\in X$ and  Lemma \ref{lemma:UltraexactEmbeddingsSQSR} shows that $\Sigma_{n+1}(\{\lambda\})$-$\sqrt{\ESR}(\vec{\lambda})$ holds. 
    
    \ref{item:SQSRfromultraexact2} Now, assume $\lambda_0$ is parametrically $n$-ultraexact for $\seq{\lambda_{m+1}}{m<\omega}$ and $z\in V_{\lambda_0}$. Then we find an $n$-ultraexact embedding $\map{j}{X}{V_\zeta}$ at $\lambda$ with $\vec{\lambda},z\in\ran{j}$,  $j(\crit{j})=\lambda_0$  and $j(\lambda_m)=\lambda_{m+1}$ for all $m<\omega$. Then $\vec{\lambda}\in X$, $j(z)=z$ and  Lemma \ref{lemma:UltraexactEmbeddingsSQSR} shows that $\Sigma_{n+1}(\{\lambda,z\})$-$\sqrt{\ESR}(\vec{\lambda})$ holds.  
\end{proof}

We shall prove next the converse to (i) of \ref{coro4.5}.
 In the following, we let $\calL$ denote the first-order language that extends the language of set theory by a constant symbol $\dot{c}$, a unary function symbol $\dot{f}$, 
 %a unary predicate symbol, 
  a binary predicate symbol $\dot{E}$, and a unary predicate symbol $\dot{P}_m$ for every $m<\omega$.  
 Given a natural number $n>0$,  we define $\calU_n$ to be the class of $\calL$-structures $A$ with the property that there exists a strictly increasing sequence $\vec{\lambda}=\seq{\lambda_m}{m<\omega}$ of cardinals with supremum $\lambda$ such that the following statements hold: 
 \begin{itemize}
     \item  The reduct of $A$ to the language of set theory is equal to $\langle V_\lambda,\in\rangle$.  

      \item There is $\lambda<\zeta\in C^{(n)}$, an elementary submodel $X$ of $V_\zeta$ with $V_\lambda\cup\{\lambda,\dot{f}^A\}\subseteq X$ and a bijection $\map{\tau}{X}{V_\lambda}$ with $\tau(\lambda)=\langle 0,0\rangle$, $\tau(x)=\langle 1,x\rangle$ for all $x\in V_\lambda$ and 
      \begin{equation}\label{equation:Ecodes}
          x\in y ~ \Longleftrightarrow ~ \tau(x) ~ \dot{E}^A ~ \tau(y)
      \end{equation} 
      for all $x,y\in X$.  
 \end{itemize}
 It is easy to check that the class $\calU_n$ is definable by a $\Sigma_{n+1}$-formula without parameters.

 \begin{lemma}
 \label{sqrtimpliesultraexact}
   If $\vec{\lambda}=\seq{\lambda_m}{m<\omega}$    is a strictly increasing sequence of cardinals  such that $\sqrt{\ESR}_{\calU_n}(\vec{\lambda})$ holds, then $\lambda_0$ is $n$-ultraexact for $\seq{\lambda_{m+1}}{m<\omega}$. 
 \end{lemma}

 \begin{proof}
  Set $\lambda=\sup_{m<\omega}\lambda_m$ and let $\map{f}{V_\lambda}{V_\lambda}$ be a function witnessing that $\sqrt{\ESR}_{\calU_n}(\vec{\lambda})$  holds. 
  Fix $D\in V_{\lambda+1}$. Now, pick $\lambda<\zeta\in C^{(n+1)}$, an elementary submodel $Y$ of $V_\zeta$ of cardinality $\lambda$ with $V_\lambda\cup\{\lambda,D,f\}\subseteq Y$, and a bijection  $\map{\pi}{Y}{V_\lambda}$ with $\pi(\lambda)=\langle 0,0\rangle$,  $\pi(x)=\langle 1,x\rangle$ for all $x\in V_\lambda$. 
  Then there is an $\calL$-structure $B$ extending $\langle V_\lambda,\in\rangle$ with $\dot{c}^B=\pi(D)$, $\dot{f}^B=f$, $\dot{E}^B=\Set{\langle \pi(x),\pi(y)\rangle}{x,y\in Y, ~ x\in y}$ and $\dot{P}_m^B=\lambda_{m+1}$ for all $m<\omega$. 
  It follows that $B$ is an element of $\calU_n$ of type $\seq{\lambda_{m+1}}{m<\omega}$. Hence, there is a structure $A$ in $\calU_n$ of type $\seq{\lambda_m}{m<\omega}$ and a square root $\map{r}{V_\lambda}{V_\lambda}$ of $f$ that is an elementary embedding of $A$ into $B$. Observe that in particular $r$ maps $\dot c^A$ to $\dot c^B$. 
  
  Thus $r$ is an I3-embedding with $r(\lambda_m)=\lambda_{m+1}$ for all $m<\omega$. Also, we have that $r(\langle m,x\rangle)=\langle m,r(x)\rangle$ holds for all $x\in V_\lambda$ and $m<\omega$. 
  Pick a cardinal $\lambda<\eta\in C^{(n)}$,  let $X$ be an elementary submodel of $V_\eta$ of size $\lambda$, with $V_\lambda\cup\{\lambda,\dot{f}^A\}\subseteq X$,  and let $\map{\tau}{X}{V_\lambda}$ be a bijection such that $\tau(\lambda)=\langle 0,0\rangle$, $\tau(x)=\langle 1,x\rangle$ for all $x\in V_\lambda$, and \eqref{equation:Ecodes} holds for all $x,y\in X$. Now define $$\map{j ~ := ~ \pi^{{-}1}\circ r\circ \tau}{X}{V_\zeta}.$$
   \begin{figure}[h]
    \begin{center}
    \begin{tikzcd} 
V_\lambda \arrow[rr, "\pi^{-1}"]                                        &  & Y \arrow[rr, "{\rm id}"] &  & V_\zeta \\
                                                                          &  &                      &  &              \\
V_\lambda \arrow[uu, "r"]  \arrow[rr, leftarrow,  "\tau"] &  & X  \arrow[rruu, "j"]            &  & 
\end{tikzcd}      
    \end{center}
    \end{figure}

    \noindent
We claim that $j$ is an $n$-ultraexact embedding at $\lambda$. First note that, letting $x$ be such that $\tau (x)=\dot{c}^A$, we have that 
 $$j(x)=  \pi^{-1}(r(\dot{c}^A)))=\pi^{-1}(\dot{c}^B))=D$$
 and so $D\in\ran{j}$. Moreover, $$j(\lambda) ~ = ~ (\pi^{{-}1}\circ r\circ \tau)(\lambda) ~ = ~ (\pi^{{-}1}\circ r)(\langle 0,0\rangle) ~ = ~ \pi^{{-}1}(\langle 0,0\rangle) ~ = ~ \lambda$$ and 
  \begin{equation}\label{equation:j+rAgreeVlambda}
      j(x) ~ = ~ (\pi^{{-}1}\circ r\circ \tau)(x) ~ = ~ (\pi^{{-}1}\circ r)(\langle 1,x\rangle ) ~ = ~ \pi^{{-}1}(\langle 1,r(x)\rangle) ~ = ~ r(x)
  \end{equation} holds for all $x\in V_\lambda$. In particular, $j(\lambda_m)=\lambda_{m+1}$ for all $m<\omega$. Further, $j$ is an elementary embedding:  for every $a\in X$ and every formula $\varphi(x)$ in the language of set theory, 
  $$X\!\models \varphi (x) \quad \mbox{iff} \quad \langle V_\lambda , \dot{E}^A\rangle \!\models \varphi (\tau (a))\quad \mbox{iff} \quad \langle V_\lambda , \dot{E}^B\rangle \models \!\varphi (r(\tau(a)))\quad \mbox{iff}\quad Y\!\models \varphi (\pi^{-1}(r(\tau(a))))$$
  and the latter holds if and only if $V_\zeta \models \varphi(j(a))$.

  \begin{claim*}
      $j\restriction V_\lambda=\dot{f}^A$. 
  \end{claim*}

  \begin{proof}[Proof of the Claim]
    Assume, towards a contradiction, that the  claim fails and pick $m<\omega$ with   $j\restriction V_{\lambda_m}\neq\dot{f}^A\restriction V_{\lambda_m}$. Elementarity and \eqref{equation:j+rAgreeVlambda} then imply that $$r(r\restriction V_{\lambda_m}) ~ = ~ r(j\restriction V_{\lambda_m}) ~ \neq ~ r(\dot{f}^A\restriction V_{\lambda_m}) ~ = ~ \dot{f}^B\restriction V_{\lambda_{m+1}} ~ = ~ f\restriction V_{\lambda_{m+1}}.$$
    Since the fact that $r$ is a square root of $f$ implies that $$r(r\restriction V_{\lambda_m}) ~ = ~ r_+(r)\restriction V_{\lambda_{m+1}} ~ = ~ f\restriction V_{\lambda_{m+1}},$$ we derived a contradiction.  
  \end{proof}

   Hence, we have $j\restriction V_\lambda = \dot{f}^A\in X$, showing that $j$ is an $n$-ultraexact embedding at $\lambda$. Since $D$ was arbitrary and belongs to the range of $j$ and since $j(\lambda_m) = r(\lambda_m) = \lambda_{m+1}$ for all $m$, it follows that $\lambda_0$ is $n$-ultraexact for $\langle \lambda_{m+1} | m < \omega\rangle$, as desired.
 \end{proof}

\begin{corollary}
\label{coro4.6}
 Let $n>0$ be a natural number and let  $\vec{\lambda}=\seq{\lambda_m}{m<\omega}$ be a strictly increasing sequence of cardinals. The cardinal  $\lambda_0$ is $n$-ultraexact for $\seq{\lambda_{m+1}}{m<\omega}$ if and only if $\Sigma_{n+1}$-$\sqrt{\ESR}(\vec{\lambda})$ holds.
 \end{corollary}

In view of the corollary above, we shall refer to square root Exact Structural Reflection also as  \emph{Ultraexact Structural Reflection}.

%%%%%%%%%%%%%%%%%%%%%%%%%%%%%%%%%
%%%%%%%%%%%%%%%%%%%%%%%%%%%%%%%%%

\section{Square Root Reflection and Choiceless Cardinals}\label{SectChoiceless}

As we already observed (Corollary \ref{corollary:noimplication} above), the existence of an ${\rm I}0$ embedding does not imply  the existence of a $2$-exact embedding, the reason being that the former is consistent with $V=\HOD$, while the latter is not. Hence an ${\rm I}0$ embedding does not imply Ultraexact Structural Reflection. In contrast, we will next show, in ZF, that Ultraexact Structural Reflection does follow from the existence of some large cardinals beyond the Axiom of Choice, such as Reinhardt, super-Reinhardt, or Berkeley cardinals.

\subsection{Ultraexact Structural Reflection and Reinhardt cardinals}

We will show that a Reinhardt cardinal $\lambda$ is parametrically $n$-ultraexact for every $n$, and moreover Ultraexact Structural Reflection for its critical sequence holds for all classes of structures that are definable with parameters in $V_\kappa$. Since Reinhardt and super-Reinhardt cardinals are given by proper class elementary embeddings, which need not be definable, we shall work in ${\rm NBG}-{\rm AC}$, von Neumann-Bernays-G\" odel class theory without the Axiom of Choice.

\begin{definition}[${\rm NBG}-{\rm AC}$]
A cardinal  is \emph{Reinhardt} if it is the critical point of a non-trivial elementary embedding $\map{j}{V}{V}$. 
 \end{definition}

 Note that the assertion that ``\emph{$j$ is elementary}" in the definition above can be formulated as ``\emph{$j$ is cofinal and $\Delta_0$-elementary}".

 \begin{theorem}[${\rm NBG}-{\rm AC}$] \label{theorem:ultraexactfromReinhardt}
If
$\map{j}{V}{V}$ is a non-trivial elementary embedding with critical sequence $\seq{\lambda_m}{m < \omega}$, then $\lambda_0$ is parametrically $n$-ultraexact for $\seq{\lambda_{m+1}}{m< \omega}$, for all $n>0$. 
\end{theorem}

\begin{proof}
 Pick $\zeta\in C^{(n+2)}$ greater than the supremum of the critical sequence, and such that $j(\zeta)=\zeta$ . Then, the map $\map{j\restriction V_\zeta}{V_\zeta}{V_\zeta}$ is an $n$-ultraexact embedding. 

Now, we can argue as in the last part of the proof of Theorem \ref{theorem:UltraexactFromI0}. Namely, 
since $j[V_\zeta]$ is an elementary submodel of $V_\zeta$ that contains the sequence $\seq{\lambda_{m+1}}{m<\omega}$ and $V_\zeta$ was chosen to be sufficiently elementary in $V$, the hypothesis that $\lambda_1$ is not parametrically $n$-ultraexact for $\seq{\lambda_{m+2}}{m<\omega}$ yields a set   $A \in j[V_\zeta]\cap V_{\lambda+1}$ with the property that there is no $n$-ultraexact embedding $\map{k}{X}{V_\zeta}$ at $\lambda$ with $A\in\ran{k}$, $k(\crit{k})=\lambda_1$  and $k(\lambda_{m+1})=\lambda_{m+2}$ for all $m<\omega$. 
But, the function $\map{j\restriction V_\zeta}{V_\zeta}{V_\zeta}$ has all of these properties, yielding a contradiction.  
We have thus shown that $\lambda_1$ is parametrically $n$-ultraexact for the sequence $\seq{\lambda_{m+2}}{m<\omega}$, for all $n>0$. Hence, the elementarity of $j$ implies that  $\lambda_0$ is parametrically $n$-ultraexact for the sequence $\seq{\lambda_{m+1}}{m<\omega}$, for all $n>0$.
\end{proof}

By checking that the proof of Lemma \ref{lemma:UltraexactEmbeddingsSQSR} does not make use of the Axiom of Choice, it can be seen that the following result is a corollary of the above theorem. But, we shall  give a simpler, direct proof. In the sequel, let  $\mathcal{L}$ be as in Definition \ref{defESR}.

\begin{corollary}[${\rm NBG}-{\rm AC}$] \label{theorem:Reinhardt}
 If $\map{j}{V}{V}$ is a non-trivial elementary embedding with critical sequence $\seq{\lambda_m}{m < \omega}$ and $\lambda=\sup_{m<\omega}\lambda_m$, then  $\sqrt{\ESR}_\Ce (\vec{\lambda})$ holds for all classes $\Ce$ of $\mathcal{L}$-structures that are definable with parameters in $V_{\lambda_0}\cup\{\lambda\}$. 
  \end{corollary}

\begin{proof}
  Fix a formula $\varphi(v_0,v_1,v_2)$ and  $z\in V_{\lambda_0}$ such that the class $\Ce=\Set{A}{\varphi(A,\lambda,z)}$ consists of $\calL$-structures. 
 Pick a structure $B$ in $\Ce$ of type $\seq{\lambda_{m+1}}{m<\omega}$. 
 Then the elementarity of $j$ implies that $\varphi(j(B),\lambda,z)$ holds. 
 Thus, we know that $j(B)$ is an  $\calL$-structure of type $\seq{\lambda_{m+2}}{m<\omega}$ and  the restriction map $\map{j\restriction B}{B}{j(B)}$ is an elementary embedding of  structures in $\Ce$. Let $f=\map{j\restriction V_\lambda}{V_\lambda}{V_\lambda}$, and note that $f$ is a square root of $j(f)$. 
  By  elementarity, we have that there is an $\calL$-structure $A$ of type $\vec{\lambda}$ with the property  that $\varphi(A,\lambda,z)$ holds and there exists an elementary embedding $\map{i}{A}{B}$ that is the restriction to $A$ of a function that is a square root of $f$. This shows that $\sqrt{\ESR}_\Ce (\vec{\lambda})$ holds.  
  \end{proof}

Under the stronger assumption of the existence of a super Reinhardt cardinal, we obtain a stronger result. Let us first recall the definition of super Reinhardt cardinal:

\begin{definition}[${\rm NBG}-{\rm AC}$]
A cardinal $\kappa$ is \emph{super Reinhardt} if for every ordinal $\lambda$ there exists  an elementary embedding $\map{j}{V}{V}$ with critical point $\kappa$ and   $j(\kappa)>\lambda$.
 \end{definition}

By a straightforward adaptation of the proof of Theorem \ref{theorem:Reinhardt}, we have the following:

\begin{proposition}[${\rm NBG}-{\rm AC}$]
  If $\kappa$ is a super Reinhardt cardinal, then  there exists a proper class of  sequences $\vec{\lambda}=\seq{\lambda_m}{m<\omega}$ of cardinals   for which  $\sqrt{\ESR}_\Ce(\vec{\lambda})$ holds for all classes $\Ce$ of $\mathcal{L}$-structures that are definable with parameters from $V_\kappa$. 
\end{proposition}

We finish with an observation on structural reflection in the presence of Reinhardt cardinals, which should be compared with the characterization of proto-Berkeley cardinals given by Proposition \ref{charprotoBC} below. Let $\calL_{\in,\dot{P}}$ denote the first-order language that extends the language of set theory by a unary predicate symbol $\dot{P}$.

\begin{proposition}[${\rm NBG}-{\rm AC}$]
 If $\kappa$ is a Reinhardt cardinal, then for every  first-order language $\calL$ extending $\calL_{\in,\dot{P}}$, every class $\Ce$ of $\calL$-structures that is definable  with parameters in  $V_\kappa$, and every structure $B$ in $\Ce$ with $\rank{\dot{P}^B}=\kappa$, there exists a structure $A$ in $\Ce$ with $\rank{\dot{P}^A}<\kappa$ and an elementary embedding of $A$ into $B$.  
\end{proposition}

\begin{proof}
 Let $\map{j}{V}{V}$ be elementary with critical point $\kappa$.  
 Fix a first-order language $\calL$ extending $\calL_{\in,\dot{P}}$, some $y\in V_\kappa$, a $\Sigma_n$-formula $\varphi(v_0,v_1)$ such that the class $\Ce=\Set{A}{\varphi(A,y)}$ consists of $\calL$-structures, and a structure $B\in\Ce$ with $\rank{\dot{P}^B}=\lambda$. 
  %
  %Let $\theta\in C^{(n)}$ be greater than $\kappa$, with $B\in V_\theta$, and such that $\theta$ is a fixed point of $j$, so that $j\restriction \theta :V_\theta \to V_\theta$ is an elementary embedding with critical point $\kappa$. 
  Then $\varphi(B,y)$ holds in $V_\theta$ and therefore the elementarity of $j$ ensures that $\varphi(j(B),y)$ holds in $V_\theta$. 
  %
  %Thus, we have that
  %
  Then $j(B)$ is an  $\calL$-structure with $\rank{\dot{P}^{j(B)}}=j(\kappa)>\kappa=\rank{\dot{P}^B}$,  and the map $\map{j\restriction B}{B}{j(B)}$ is an elementary embedding of $\calL$-structures.  %that is an element of $V_\theta$. 
  The elementarity of $j$ then yields  that %, in $V_\theta$, 
   there is an $\calL$-structure $A$ such that $\rank{\dot{P}^A}<\kappa$, $\varphi(A,y)$ holds and there exists an elementary embedding $\map{j}{A}{B}$. %Moreover, by the $\Sigma_n$-correctness of $V_\theta$, 
   It follows that the structure $A$ is an element of $\Ce$ with the desired properties.  
\end{proof}

\subsection{Ultraexact Structural Reflection and Berkeley cardinals}
Let us recall the definition of Berkeley cardinals:

\begin{definition}[\cite{MR4022642}]
 \begin{enumerate}
     \item An ordinal $\delta$ is a \emph{proto-Berkeley cardinal} if for all 
transitive sets $M$ with $\delta\in M$,  there exists a non-trivial elementary embedding $\map{j}{M}{M}$ with $\crit{j}<\delta$. 

  \item An ordinal $\delta$ is a \emph{Berkeley cardinal} if for all 
transitive sets $M$ with $\delta\in M$, for every $\eta <\delta$ there exists a non-trivial elementary embedding $\map{j}{M}{M}$ with $\eta<\crit{j}<\delta$.
 \end{enumerate}
\end{definition}

As shown in \cite{MR4022642}, under ZF, the least proto-Berkeley cardinal, if it exists, is a Berkeley cardinal. Arguing similarly as in the proof of Theorem \ref{theorem:ultraexactfromReinhardt}, we have the following:

\begin{theorem}[${\rm ZF}$]\label{theorem:ultraexactfromBC}
 Given a natural number $n>0$, if  $\delta$ is a Berkeley cardinal, then unboundedly many cardinals below $\delta$ are parametrically $n$-ultraexact cardinals for some sequence of cardinals below $\delta$. 
\end{theorem}

\begin{proof} 
  Fix any $\alpha <\delta$ and pick $\delta<\zeta\in C^{(n+2)}$. Then the results of \cite{MR4022642} show that there is an  elementary embedding $\map{j}{V_\zeta}{V_\zeta}$ with $\alpha<\crit{j}<\delta$ and  $j(\delta)=\delta$. It follows that  $j$ is an $n$-ultraexact embedding. If $\seq{\lambda_m}{m<\omega}$ is its critical sequence, then $\lambda_m<\delta$ for all $m<\omega$ and    we can argue as in the proof of Theorem \ref{theorem:ultraexactfromReinhardt} to show that $\lambda_0$ is parametrically $n$-ultraexact for the sequence $\seq{\lambda_{m+1}}{m<\omega}$. 
\end{proof}

Let us show next that the least Berkeley cardinal is the limit of a sequence of cardinals witnessing full Ultraexact Structural Reflection.

\begin{theorem}[ZF]
 Given a natural number $n>0$, if $\delta$ is the least Berkeley cardinal, then for every set $z$, there exists a strictly increasing sequence $\vec{\lambda}=\seq{\lambda_m}{m<\omega}$ of cardinals with supremum  less than $\delta$ and the property that $\sqrt{\ESR}_\Ce(\vec{\lambda})$ holds for every class $\Ce$ of $\calL$-structures that is definable by a $\Sigma_n$-formula with parameter $z$.  
\end{theorem}

\begin{proof}
  Using {\cite[Lemma 2.1.19]{Cu17}}, we can find a cardinal $\delta<\zeta\in C^{(n+2)}$ with $z\in V_\zeta$ and a non-trivial elementary embedding $\map{j}{V_\zeta}{V_\zeta}$ with critical sequence $\vec{\lambda}=\seq{\lambda_m}{m<\omega}$ such that $\crit{(j)}<\delta$, $j(\delta)=\delta$, $j(z)=z$ and $\lambda=\sup_{m<\omega}\lambda_m<\delta$. Set $\map{f:=j\restriction V_\lambda}{V_\lambda}{V_\lambda}$.

  Now, let  $\Ce$ be a class of $\calL$-structures that is definable by a $\Sigma_n$-formula with parameter $z$. Pick a $\Sigma_n$-formula $\varphi(v_0,v_1)$  with  $\Ce=\Set{A}{\varphi(A,z)}$.  Fix  a structure $B \in \Ce$ of type $\seq{\lambda_{m+1}}{m<\omega}$. Then $B$ is an element of $V\zeta$ and our setup ensures that $\varphi(j(B),z)$ holds in $V$. Hence, the  $\calL$-structure  $j(B)$ is an element of $\Ce$ of type $\seq{\lambda_{m+2}}{m<\omega}$ and  the restriction map $\map{j\restriction B}{B}{j(B)}$ is an elementary embedding of $\calL$-structures with $j\restriction B  \in V_\theta$. Moreover, notice that $j\restriction B$ is the restriction to $B$ of a function, namely $j\restriction V_\lambda$, that is a square root of $j(f)$. Since all of these statements are absolute between $V$ and $V_\zeta$, we can use the elementarity of $j$ and the correctness properties of $V_\zeta$ to find an  $\calL$-structure $A$ of type $\vec{\lambda}$ with the property  that $\varphi(A,z)$ holds and there exists an elementary embedding $\map{i}{A}{B}$ that is the restriction to $A$ of a function that is a square root of $f$. This allows us to conclude that $\sqrt{\ESR}_\Ce (\vec{\lambda})$ holds. 
  \end{proof}

%  As an immediate consequence of the proof of the theorem, we have the following:

%\begin{corollary}[ZF]
% Let $\delta$ be the least Berkeley cardinal. Given    $\eta <\delta$ and given any set $z$, there exists a strictly increasing sequence $\vec{\lambda}=\seq{\lambda_i}{i<\omega}$ of cardinals greater than $\eta$ and with supremum less than $\delta$ such that  $\sqrt{\ESR}_\Ce(\vec{\lambda})$ holds for every class $\Ce$ of $\calL$-structures  that is definable with parameter $z$. 
%\end{corollary}

Thus, in view of Theorem \ref{theorem:UltraexactFromI0}, and Corollary \ref{coro4.6}, the least Berkeley cardinal cannot be characterized in terms of Ultraexact Structural Reflection.
Nevertheless, we will  show next that proto-Berkeley cardinals, and therefore also the least Berkeley cardinal, can be characterized as a rather natural form of Structural Reflection.
%In the following, we let $\calL_{\in,\dot{P}}$ denote the first-order language that extends the language of set theory by a unary predicate symbol $\dot{P}$. 

\begin{proposition}[\rm ZF]
\label{charprotoBC}
 The following statements are equivalent for every limit cardinal $\delta$: 
 \begin{enumerate}
     \item $\delta$ is a proto-Berkeley cardinal. 
     
     \item For every natural number $n$ and every set $z$, there is a cardinal  $\lambda<\delta$ such that for every  first-order language $\calL$ extending $\calL_{\in,\dot{P}}$, every class $\Ce$ of $\calL$-structures that is definable  by a $\Sigma_n$-formula with parameters in  $V_\lambda\cup\{z\}$, and every structure $B$ in $\Ce$ with $\rank{\dot{P}^B}=\lambda$, there exists a structure $A$ in $\Ce$ with $\rank{\dot{P}^A}<\lambda$ and an elementary embedding of $A$ into $B$. 
     
     \item For every set $z$ and every class $\Ce$ of $\calL_{\in,\dot{P}}$-structures that is definable by a $\Sigma_0$-formula with parameter $z$, there exists an ordinal $\lambda<\delta$ such that for every structure $B$ in $\Ce$ with $\rank{\dot{P}^B}=\lambda$, there exists a structure $A$ in $\Ce$ with $\rank{\dot{P}^A}<\lambda$ and an elementary embedding of $A$ into $B$. 
 \end{enumerate}
\end{proposition}

\begin{proof}
 Assume that (i) holds and let us prove (ii). So fix a natural number $n$ and a set $z$. 
 By {\cite[Lemma 3.4]{MR4022642}},  for every ordinal $\theta>\delta$  with $z\in V_\theta$, there exists a minimal cardinal  $\lambda_\theta<\delta$ with the property that there exists a non-trivial elementary embedding $\map{j}{V_\theta}{V_\theta}$ with $\crit{j}=\lambda_\theta$  and $j(z)=z$.
It follows that there exists a cardinal $\lambda<\delta$ such that $\lambda=\lambda_\theta$ for a proper class of cardinals $\theta>\delta$ in $C^{(n+2)}$ with $z\in V_\theta$.
 
  Now, fix a first-order language $\calL$ extending $\calL_{\in,\dot{P}}$, let $y\in V_\lambda$, and let $\varphi(v_0,v_1,v_2)$ be a $\Sigma_n$-formula  such that the class $\Ce=\Set{A}{\varphi(A,y,z)}$ consists of $\calL$-structures. Let $B\in\Ce$ be such that  $\rank{\dot{P}^B}=\lambda$. 
  Then there exists a cardinal $\theta>\delta$ in $C^{(n)}$ with $B,z\in V_\theta$ and $\lambda=\lambda_\theta$. 
So there exists a non-trivial elementary embedding $\map{j}{V_\theta}{V_\theta}$ with $\crit{j}=\lambda$  and $j(z)=z$. 
  Moreover, we know that $\varphi(B,y,z)$ holds in $V_\theta$ and therefore the elementarity of $j$ ensures that $\varphi(j(B),y,z)$ holds in $V_\theta$ as well. 
  Thus, we have that $j(B)$ is an  $\calL$-structure with $\rank{\dot{P}^{j(B)}}=j(\lambda)>\lambda=\rank{\dot{P}^B}$ and the restriction map $\map{j\restriction B}{B}{j(B)}$ is an elementary embedding of $\calL$-structures that is an element of $V_\theta$. 
  The elementarity of $j$  then gives  that, in $V_\theta$, there is an $\calL$-structure $A$ with $\rank{\dot{P}^A}<\lambda$ and such   that $\varphi(A,y,z)$ holds, and there exists an elementary embedding $\map{j}{A}{B}$. But then the $\Sigma_n$-correctness of $V_\theta$ yields that $\varphi(A,y,z)$ holds in $V$, and so $A$ is an element of $\Ce$. 
  
  Now, assume that (iii) holds and let us prove (i). So, fix a transitive set $M$ with $\delta\in M$. Define $\Ce$ to be the class of all $\calL_{\in,\dot{P}}$-structures $A$ with universe $M$ and the property that $\dot{P}^A$ is an ordinal in $M$. Then $\Ce$ is definable by a $\Sigma_0$-formula with parameter $M$. 
  By our assumption, there exists an ordinal $\lambda<\delta$ with the property that for every  $B$ in $\Ce$ with $\rank{\dot{P}^B}=\lambda$, there exists a structure $A$ in $\Ce$ with $\rank{\dot{P}^A}<\lambda$ and an elementary embedding of $A$ into $B$. 
  
  Let $B$ denote the unique structure in $\Ce$ with $\dot{P}^B=\lambda$. %
  Then there exists a structure $A\in\Ce$, an ordinal $\eta<\lambda$,  and an elementary embedding $\map{j}{M}{M}$ such that $\dot{P}^A=\eta$ and the map $j$ is an elementary embedding of $A$ into $B$. 
  But then $\eta\in\dot{P}^B\setminus\dot{P}^A$ and therefore the elementarity of $j$ implies that $\eta<\lambda\leq j(\eta)\in M\cap\On$. 
  Thus, $j$ is a non-trivial elementary embedding with critical point less than $\delta$. 
\end{proof}

Since the least proto-Berkeley cardinal is a Berkeley cardinal, an immediate consequence of the proposition above is that the least Berkeley cardinal is characterized by being the least cardinal for which (ii) or (iii) above hold.

%%%%%%%%%%%%%%%%%%%%%%%%%%%%%%%%%
%%%%%%%%%%%%%%%%%%%%%%%%%%%%%%%%%

%\section{A failure of Woodin's HOD Conjecture}
\section{The consistency of an exacting cardinal above an extendible cardinal}
\label{section6}

%We define the class theory $\mathrm{ZF}_2$ as in {\cite[Section 1.1]{MR4022642}}. This is the setting in which we want to consider \emph{super Reinhardt cardinals}, {i.e.,} cardinals $\kappa$ with the property that for every ordinal $\alpha$, there exists a non-trivial elementary embedding $\map{j}{V}{V}$ with $\crit{j}=\kappa$ and $j(\kappa)>\alpha$. It is then easy to see that if $\kappa$ is a super Reinhardt cardinal, then $V_\kappa$ is an elementary submodel of $V$ (with respect to first-order formulas). Moreover, note that, if $\map{j}{V}{V}$ is a non-trivial elementary embedding with critical point $\kappa$ and $V_\kappa\prec V$, then we also have $V_{j^n(\kappa)}\prec V$ for all natural numbers $n$. Thus, if $\lambda$ is the supremum of the critical sequence of an elementary embedding $\map{j}{V}{V}$ with $V_{\crit{j}}\prec V$, then $V_\lambda\prec V$ also holds. 

%

Let us recall the notion of $C^{(n)}$-Reinhardt cardinal:

\begin{definition}[${\rm NBG}-{\rm AC}$]
 A cardinal $\kappa$ is \emph{$C^{(n)}$-Reinhardt} if there is a non-trivial elementary embedding $\map{j}{V}{V}$ with critical point $\kappa\in C^{(n)}$.   
\end{definition}

\begin{lemma}[${\rm NBG}-{\rm AC}$]
 Let $\map{j}{V}{V}$ be an elementary embedding witnessing  that a cardinal $\kappa$ is $C^{(n)}$-Reinhardt and let $\seq{\kappa_m}{m<\omega}$ be the critical sequence of $j$. Then, we have $\kappa_m \in C^{(n)}$ for all $m<\omega$ and thus also $\sup_{m<\omega}\kappa_m \in C^{(n)}$. 
\end{lemma}

\begin{proof}
By a straightforward induction, the fact that the class $C^{(n)}$ is definable by a formula without parameters and the elementarity of $j$ imply that, if $\kappa_m$ is an element of $C^{(n)}$ for some $m<\omega$, then $\kappa_{m+1}=j(\kappa_m)$ is an element of $C^{(n)}$. 
% . Suppose we have shown that $\kappa_n \in C^{(n)}$ and, towards a contradiction, suppose that $\kappa_{n+1} \not \in C^{(n)}$. Thus, there is some parameter $x \in V_{\kappa_{n+1}}$ and some $\Sigma_n$ formula $\varphi$ such that $\varphi(x)$ holds but $V_{\kappa_{n+1}} \not\models\varphi(x)$. This is expressible using $\kappa_{n+1}$ as a parameter, so by taking the preimage by $j$ we find some parameter $x \in V_{\kappa_{n}}$ such that $\varphi(x)$ holds but $V_{\kappa_{n}} \not\models\varphi(x)$, which is a contradiction.
\end{proof}

Let us also recall the definitions of 
 %supercompact, extendible, and $n$-huge cardinals 
 standard large cardinal notions in the absence of the Axiom of Choice. First, following Woodin {\cite[Definition p. 323]{WOEM1}}, we fix the following form of inaccessibility (see \cite{MR2310342} for a discussion of inaccessibility in the absence of the Axiom of Choice):

 \begin{definition}[ZF]
    A cardinal $\kappa$ is \emph{strongly inaccessible} if for every ordinal $\alpha<\kappa$, there is no function $\map{f}{V_\alpha}{\kappa}$ whose  range is unbounded in $\kappa$.   
\end{definition}

A standard argument then shows that the strong inaccessibility of a cardinal $\kappa$ is preserved by forcing with partial order that are elements of $V_\kappa$.  
 
 Next, again following Woodin {\cite[Definition 220]{WOEM1}}, we define:

\begin{definition}[ZF]
    A cardinal $\kappa$ is \emph{supercompact} if for every ordinal $\alpha>\kappa$, there is an ordinal $\beta>\alpha$, a transitive set $N$ with ${}^{V_\alpha}N\subseteq N$ and an elementary embedding $\map{j}{V_\beta}{N}$ with $\crit{j}=\kappa$ and $j(\kappa)>\alpha$.  
\end{definition}

It is easy to see that, in  ${\rm ZFC}$, this definition is equivalent to the standard definition of supercompactness. Moreover, standard arguments show that ${\rm ZF}$   proves that supercompact cardinals are strongly inaccessible elements of $C^{(2)}$. 

Finally, again in ${\rm ZF}$, a cardinal $\kappa$ is \emph{extendible} if for all $\alpha>\kappa$ there is an elementary embedding $\map{j}{V_\alpha}{V_\beta}$ for some $\beta$, such that $\kappa =\crit{j}$ and $j(\kappa)>\alpha$. It is easily seen that ${\rm ZF}$ proves that extendible cardinals are supercompact. 
%
%, now in ${\rm NBG}-{\rm AC}$, a cardinal $\kappa$ is \emph{$n$-huge} if there a transitive class $M$ and an elementary embedding $\map{j}{V}{M}$ such that $\crit{j}=\kappa$ and $^{j^n(\kappa)}M\subseteq M$.  Thus, a cardinal is \emph{huge} if it is $1$-huge. (Under AC, $n$-huge cardinals have equivalent formulations in terms of the existence of ultrafilters \cite[Theorem 24.8]{Kan:THI}.) 
%
Moreover, standard arguments show that  ${\rm NBG}-{\rm AC}$ proves that  every cardinal $\lambda$ that is the supremum of the critical sequence of a non-trivial elementary embedding $\map{j}{V}{V}$ is a limit of cardinals that are extendible in $V_\lambda$. In particular, every element of $C^{(3)}$ with this property is a limit of extendible cardinals.  

\medskip

The goal of this section will be to prove the following result: 

\begin{theorem}[${\rm NBG}-{\rm AC}$]
\label{theorem:ConsUltraexactAboveSupercompact}
    If there  is a $C^{(3)}$-Reinhardt cardinal and a supercompact cardinal greater than the supremum of the critical sequence, then there exists a set-sized model of $\mathrm ZFC$ and the statement 
    \begin{center}
        ``There is an ultraexacting cardinal that is a limit of extendible cardinals". 
    \end{center} 
\end{theorem}

%A short argument shows that super Reinhardt cardinals are supercompact (see {\cite[Remark 221.(3)]{WOEM1}}). In combination with earlier remarks,  if $\kappa$ is a super Reinhardt cardinal and $\alpha>\kappa$, then there is a supercompact cardinal $\delta$ greater than  $\alpha$  and such that $V_\delta\prec V$. This fact will be used later, in the proof of Theorem \ref{theorem:ConsUltraexactAboveSupercompact}.

Recall that, given an infinite cardinal  $\lambda$, the \emph{Dependent Choice principle} $\lambda$-DC states that for every non-empty set $D$ and every binary relation $R$ with the property that for all $s\in{}^{{<}\lambda}D \setminus \{ \emptyset \}$, there exists $d\in D$ with $s \mathrel{R} d$,  there exists a function $\map{f}{\lambda}{D}$ with the property that ${(f\restriction\alpha)} \mathrel{R} f(\alpha)$ holds for all $\alpha<\lambda$.  It is easy to see that the Axiom of Choice is equivalent to the statement that $\lambda$-DC holds for every cardinal $\lambda$.
 In addition, given an infinite cardinal $\kappa$, we let ${<}\kappa$-DC denote the statement that $\lambda$-DC holds for every infinite cardinal $\lambda<\kappa$. Another easy argument, in ZF, then shows that if $\kappa$ is a singular cardinal with the property that ${<}\kappa$-DC holds, then $\kappa$-DC holds. 

%Also recall that an embedding $e:\mathbb{Q}\to \mathbb{P}$, where $\mathbb{P}$ and $\mathbb{Q}$ are partial orderings, is  a \emph{projection} if
%the preimage of any dense open subset of $\mathbb{P}$ is predense in $\mathbb{Q}$. 
%It can be easily verified that 
 %$e:\mathbb{Q}\to \mathbb{P}$ is a projection  if and only if, writing $\dot{G}$ for the standard name for the $\mathbb{Q}$-generic filter over $V$,  $\mathbb{Q}$ forces that  the set $\{ p\in \mathbb{P}: \exists q\in \dot{G}\; (e(q)\leq p)\}$ is a $\mathbb{P}$-generic filter over $V$.

 In order to prove Theorem \ref{theorem:ConsUltraexactAboveSupercompact}, we will use the following version of a theorem of Woodin  (see {\cite[Theorem 226]{WOEM1}}):

\begin{theorem}[Woodin, ZF]\label{theorem:WoodinForceZFC}
  If $\delta$ is  a supercompact cardinal, then there is a partial order $\QQQ\subseteq V_\delta$ such that the following hold: 
  \begin{enumerate}
      \item $\QQQ$ is  homogeneous. 

      \item $\QQQ$ is $\Sigma_3$-definable, without parameters, over $V_\delta$. 

      \item\label{item:Woodin3} If $G$ is $\QQQ$-generic over $V$, then $V[G]_\delta$ is a model of ZFC, and every supercompact (respectively, extendible) cardinal  smaller than $\delta$ in $V$ is supercompact (respectively, extendible) in $V[G]_\delta$. 

      \item\label{item:Woodin4} If $\lambda<\delta$ is a cardinal in $C^{(3)}$ and $\map{j}{V_{\lambda+1}}{V_{\lambda+1}}$ is a non-trivial elementary embedding such that $\lambda$ is the supremum of the critical sequence of $j$, then there is a complete suborder $\PPP$ of $\QQQ$ with  $\PPP\subseteq V_{\lambda+1}$ and a $\PPP$-name $\dot{\RRR}$ for a partial order such that the following statements hold: 
      \begin{enumerate}
          %\item $\PPP$ is homogeneous. 

          \item\label{item:Woodin4-2} $\PPP$ is homogeneous and $\Sigma_3$-definable, without parameters, over $V_{\lambda+1}$. 

          \item\label{item:Woodin4-3} There is a dense embedding of $\QQQ$ into $\PPP*\dot{\RRR}$ that maps every condition $p$ in $\PPP$ to $(p,\mathbbm{1}_{\dot{\RRR}})$.  %$\QQQ \cong \PPP*\dot{\RRR}$. 

          \item\label{item:Woodin4-4} $\mathbbm{1}_{\PPP}\Vdash ``\textit{$\dot{\RRR}$ is  homogeneous and ${<}\check{\lambda}^+$-closed}"$

          %\item\label{item:Woodin4-5} There is a dense embedding of $\QQQ$ into  $\PPP*\dot{\RRR}$ that induces $e$. 

          \item\label{item:Woodin4-6} $\mathbbm{1}_{\PPP}\Vdash \text{$\check{\lambda}$-DC}$

          \item\label{item:Woodin4-7} There is a condition $p$ in $\PPP$ with the property that whenever $G_0$ is $\PPP$-generic over $V$ with $p\in G_0$, then $j[G_0]\subseteq G_0$ holds. 
      \end{enumerate}
  \end{enumerate}
\end{theorem}

\begin{proof}[Sketch of the proof]
For each pair $(\gamma,\eta)$ of regular infinite cardinals such that $\gamma <\eta$, we let  $\QQQ^\eta_\gamma$ denote the partial ordering of all partial functions $$\pmap{p}{\gamma \times \eta}{V_\eta}{part}$$ with domain of cardinality less than $\gamma$ and the property that $p(\alpha ,\beta)\in V_{1+\beta}$ holds for all $(\alpha ,\beta)$ in the domain of $p$. The ordering is by extension, {i.e.,} we have $p\leq_{\QQQ^\eta_\gamma} q$ if and only if $q\subseteq p$. 
 The following statements are easily checked:
  \begin{enumerate}
   \item $\QQQ^\eta_\gamma$ is ${<}\gamma$-closed.  
   
   \item For all $0<\beta <\eta$, forcing with $\QQQ^\eta_\gamma$ adds a function from $\gamma$ onto $V_{\beta}$. 
   
   \item $\QQQ^\eta_\gamma$ is definable in $V_\eta$ by a    $\Sigma_1$-formula  with  parameter $\gamma$. 
  \end{enumerate}

  The partial order $\QQQ$ is the direct limit of an iteration $\seq{\QQQ_\beta}{\beta <\delta}$, defined relative to a sequence $\seq{\kappa_\beta}{\beta <\delta}$ of regular cardinals that is cofinal in $\delta$, where $\QQQ_0=\emptyset$ and $\kappa_0$ is the least regular cardinal $\mu$ such that ${<}\mu$-DC holds and $\mu$-DC fails;   and for every  $\beta <\delta$, we have $\QQQ_{\beta +1}=\QQQ_\beta \ast \dot{\QQQ}^{\kappa_{\beta +1}}_{\kappa_\beta}$, where $\kappa_{\beta +1}$ is the least inaccessible cardinal $\nu$ greater than $\kappa_\beta$ such that if $\QQQ_\beta$ forces ${<}\kappa_\beta$-DC to hold, then $\QQQ_{\beta +1}$ forces ${<}\nu$-DC to hold (see \cite[Theorem 225]{WOEM1} for details). If $\beta$ is a limit ordinal and $\gamma=\sup_{\alpha <\beta}\kappa_\alpha$ is strongly inaccessible, then $\kappa_\beta =\gamma$ and $\QQQ_\beta$ is the direct limit of the iteration up to $\beta$. Otherwise, the inverse limit is taken, and we set $\kappa_\beta =\gamma^+$. The iteration is well-defined, assuming that $\kappa_0$ exists and is less than $\delta$. It is then easily seen that $\QQQ$ is  homogeneous, definable in $V_\delta$ by a $\Sigma_3$-formula  without parameters,  it forces ${<}\delta$-DC to hold, and preserves the strong inaccessibility of $\delta$. Hence, if $G$ is $\QQQ$-generic over $V$, then $V[G]_\delta$ is a model of $ZFC$.
 
 Now suppose $\delta_0 <\delta$ is a supercompact cardinal, and let us show it is supercompact in $V[G]_\delta$.  Woodin  (see {\cite[Lemma 222]{WOEM1}}) showed that there exist $\bar{\delta}_0 < \bar{\gamma} < \delta_0 <\gamma <\delta$, with $\kappa_\gamma =\gamma$, and $\bar{\QQQ}$, with an elementary embedding 
$$\map{i}{V_{\bar{\gamma}+\omega}}{V_{\gamma +\omega}}$$ 
having critical point $\bar{\delta}_0$ and such that $i(\bar{\delta}_0)=\delta_0$ and $i(\bar{\QQQ})=\QQQ_\gamma$. As $\bar{\QQQ}\restriction \bar{\delta}_0 =\QQQ_{\bar{\delta}_0}$, the map $i$ lifts to an elementary embedding 
$$\map{i}{V_{\bar{\gamma}+\omega}[G\restriction \bar{\delta}_0]}{V_{\gamma+\omega}[G\restriction \delta_0]}.$$ 
Since $\kappa_\gamma =\gamma$, and $V_{\bar{\gamma}}$ is sufficiently correct in $V$, it follows that $\bar{\QQQ}=\QQQ_{\bar{\gamma}}$. Moreover, we have $\QQQ_\gamma =\QQQ_{\delta_0}\ast \dot{\RRR}$, where $\dot{\RRR}$ is forced by $\QQQ_{\delta_0}$ to be homogeneous and ${<}\delta_0$-closed in $V_{\gamma +\omega}[\dot{G}\restriction \delta_0]$. Let $\RRR=\dot{\RRR}[G\restriction \delta_0]$, and let $\bar{\RRR}\in V_{\bar{\gamma}+\omega}[G\restriction \bar{\delta}_0]$ be such that $i(\bar{\RRR})=\RRR$. So, we have $V[G\restriction \bar{\gamma}]=V[G\restriction \bar{\delta}_0][g]$, where $g\subseteq \bar{\RRR}$ is $V[G\restriction \bar{\delta}_0]$-generic. As $\RRR$ is ${<\delta_0}$-closed in $V[G\restriction \delta_0]$, there is a condition $p\in \RRR$ below $i(r)$, for all $r\in g$. By homogeneity, we may assume that $p\in G$. Therefore, the map $i$ lifts to an elementary embedding 
$$\map{i}{V_{\bar{\gamma}+\omega}[G\restriction \bar{\gamma}]}{V_{\gamma +\omega}[G\restriction \gamma]}.$$ 
Since $\kappa_{\bar{\gamma}}=\bar{\gamma}$ and $\kappa_\gamma =\gamma$, we have that $V_{\bar{\gamma}}[G\restriction \bar{\gamma}]=V[G]_{\bar{\gamma}}$ and $V_\gamma[G\restriction \gamma]=V[G]_\gamma$. Thus, in $V[G]$, for each $\gamma$ such that $\delta_0 <\gamma <\delta$, there exist $\bar{\delta}_0, \bar{\gamma} <\delta_0$  and an elementary embedding 
$$\map{i}{V[G]_{\bar{\gamma}}}{V[G]_\gamma}$$
with critical point $\bar{\delta}_0$ such that $i(\bar{\delta}_0)=\delta_0$. This shows that $\delta_0$ is supercompact in $V[G]_\delta$. 

The proof that every extendible cardinal smaller than $\delta$ in $V$ is extendible in $V[G]_\delta$ is similar, and easier. This proves (iii). 

To prove (iv), assume $\lambda<\delta$ is a cardinal in $C^{(3)}$ and $\map{j}{V_{\lambda+1}}{V_{\lambda+1}}$ is a non-trivial elementary embedding with the property that $\lambda$ is the supremum of the critical sequence of $j$. Since $\lambda \in C^{(3)}$ and the supercompactness of $\delta$ implies that $\delta \in C^{(2)}$, we have that $V_\lambda \prec_{\Sigma_3} V_\delta$, and this implies that $\seq{\QQQ_\beta }{\beta <\lambda}\subseteq V_\lambda$. Moreover,  since $\lambda$ has countable cofinality, it follows that $\QQQ_\lambda \subseteq V_{\lambda +1}$ is the inverse limit of $\seq{\QQQ_\beta}{\beta <\lambda}$. Moreover,  the remaining part of the iteration is  homogeneous and $\lambda$-closed in $V^{\QQQ_\lambda}$. So, set $\PPP = \QQQ_\lambda$ and let $\dot{\RRR}$ be a $\PPP$-name for the tail of the iteration. Then clauses (a)-(c) clearly hold.

To prove clause (d), as before one can show that $\PPP$ forces, over $V$, that ${<}\lambda$-DC holds and this implies that $\lambda$-DC holds. Now, note that,  since $\PPP$ forces that $\dot{\RRR}$ is ${<}\lambda^+$-closed, we know that
$${}^\lambda V[G\restriction \lambda] \subseteq V[G\restriction \lambda]$$ holds in $V[G]$, and so it follows that $\lambda$-DC also holds in $V[G\restriction\lambda]$.

Finally, to prove clause (e), let $\seq{\lambda_m}{m<\omega}$ be the critical sequence of $j$. For every $m<\omega$, the forcing  $\PPP_{\lambda_{m+1}}=\QQQ_{\lambda_{m+1}}$ may be seen as a two-step iteration $\QQQ_{\lambda_m}\ast \dot{\QQQ}_{\lambda_m,\lambda_{m+1}}$ with $\QQQ_{\lambda_m}\subseteq V_{\lambda_m}$. Then, as $\dot{\QQQ}_{\lambda_m,\lambda_{m+1}}$ is forced by $\QQQ_{\lambda_m}$ to be ${<}\lambda_m$-closed, if $g_m$ is $\QQQ_{\lambda_m}$-generic over $V$, then in $V[g_m]$ there is a condition $p_{m+1}$ in $\dot{\QQQ}_{\lambda_m,\lambda_{m+1}}^{g_n}$ that is below $(j(r)\restriction [\lambda_m,\lambda_{m+1}))^{g_m}$, for all $r\in g_m$. Thus, starting with any condition $p_0\in \PPP_{\lambda_0}$, we can successively find a $\PPP_{\lambda_m}$-name $\dot{p}_m$ for a condition in $\dot{\QQQ}_{\lambda_m,\lambda_{m+1}}$, so that the condition $p\in \PPP$ given by the sequence $\langle p_0\rangle^\frown\seq{\dot{p}_{m+1}}{m < \omega}$ has the property that for every $G_0$ that is $\PPP$-generic over $V$ with $p\in G_0$,  $j[G_0]\subseteq G_0$ holds. For suppose $q\in G_0$. Pick $r\leq_\PPP p,q$ with  $r\in G_0$. Then $$p \leq_\PPP j(r)= \langle j(r\restriction[0,\lambda_0))\rangle^\frown \seq{j(r\restriction [\lambda_m,\lambda_{m+1}))}{n<\omega}$$ and therefore $j(r)\in G_0$, hence $j(q)\in G_0$.  
\end{proof}

 We will next show that forcing with Woodin's partial order $\QQQ$ over a model with a  $C^{(3)}$-Reinhardt cardinal and a supercompact cardinal above the supremum $\lambda$ of its critical sequence yields a $1$-ultraexact embedding at  $\lambda$ in a rank-initial segment of the generic extension.

\begin{lemma}[${\rm NBG}-{\rm AC}$]\label{lemma:WoddinsForcingUltraexact}
  Suppose $\map{j}{V}{V}$ is a non-trivial  elementary embedding,   $\lambda \in C^{(3)}$ is the supremum of its critical sequence, and  $\delta>\lambda$ is  a supercompact cardinal with $j(\delta)=\delta$. If $\QQQ$ is the partial order  given by Theorem \ref{theorem:WoodinForceZFC} and $G$ is $\QQQ$-generic over $V$, then there is a $1$-ultraexact embedding at  $\lambda$ in $V[G]_\delta$ that extends $j\restriction V_\lambda$.  
\end{lemma}

\begin{proof}
 Since the partial order $\QQQ$ is  homogeneous, it will be  sufficient to show that some condition in $\QQQ$ forces the existence, in $V[G]_\delta$, of a $1$-ultraexact embedding at $\lambda$ that extends the map $j\restriction V_\lambda$.  

 Let $p$ be the condition in $\PPP$ given by Clause \ref{item:Woodin4-7} in the theorem above. 
 Suppose $G$ is a $\QQQ$-generic filter over $V$ with $p\in G$ and  
 %We will show that there is a $2$-ultraexact embedding at $\lambda$ in $V[G]_\delta$ by appealing to the characterization given by Lemma \ref{lemma:EquivalentUltra}.
 let $G_0*G_1$ denote the filter on $\PPP*\dot{\RRR}$ induced by the dense embedding  given by Clause \ref{item:Woodin4-3} in Theorem \ref{theorem:WoodinForceZFC}. Thus, we have $V[G_0 \ast G_1]=V[G]$. Since $p \in G$,  we know that $p \in G_0$, and so, by Clause \ref{item:Woodin4-7},  the map $j\restriction V_{\lambda+1}$ lifts to an elementary embedding $$\map{j_*}{V[G_0]_{\lambda+1}}{V[G_0]_{\lambda+1}}$$ in $V[G_0]$. In particular, by Clause \ref{item:Woodin4-6} in Theorem \ref{theorem:WoodinForceZFC}, this implies  that $\lambda$ is a limit of strongly inaccessible cardinals in $V[G_0]$, and moreover, by Clause \ref{item:Woodin4-4}, it follows that $V[G]_\lambda$ has cardinality $\lambda$ in $V[G]_\delta$. 
 Thus, in $V[G]_\delta$, we can find $\lambda<\eta\in C^{(2)}$ and an elementary submodel $X$ of $V[G]_\eta$ of cardinality $\lambda$ with $V[G]_\lambda\cup\{j_*\restriction V[G]_\lambda\}\subseteq X$. Let $\map{\pi}{X}{M}$ denote the corresponding transitive collapse. 
 
 Clearly, we have $M\in H(\lambda^+)^{V[G]_\delta}$, so we can again use  Clauses \ref{item:Woodin4-4} and \ref{item:Woodin4-6} in Theorem \ref{theorem:WoodinForceZFC} to conclude that $M$ is an element of $V[G_0]$.  The homogeneity of $\dot{\RRR}^{G_0}$ in $V[G_0]$ then implies that whenever $F$ is $\dot{\RRR}^{G_0}$-generic over $V[G_0]$, then, in $V[G_0,F]_\delta$, we can find a cardinal $\lambda<\zeta\in C^{(2)}$ and an elementary submodel $Y$ of $V[G_0,F]_\zeta$ such that $V[G_0,F]_\lambda\cup\{\lambda\}\subseteq Y$ and the transitive collapse of $Y$ is equal to $M$. 
  Pick a $\PPP$-name $\dot{M}$ in $V$ with $\dot{M}^{G_0}=M$. Then, there is a condition $p_0$ in $G_0$ with the property that whenever $H_0*H_1$ is $(\PPP*\dot{\RRR})$-generic over $V$ with $p_0\in H_0$, then, in $V[H_0,H_1]_\delta$, we can find $\lambda<\zeta\in C^{(2)}$ and an elementary submodel $Y$ of $V[H_0,H_1]_\zeta$ such that $V[H_0,H_1]_\lambda\cup\{\lambda\}\subseteq Y$ and the transitive collapse of $Y$ is equal to $\dot{M}^{H_0}$.

  Note that by Clause \ref{item:Woodin4-2} in Theorem \ref{theorem:WoodinForceZFC}, we have $j(\PPP)=\PPP$. Thus, by the elementarity of $j$, we then have that whenever $H_0*H_1$ is $(\PPP*j(\dot{\RRR}))$-generic over $V$ with $j(p_0)\in H_0$, then, in $V[H_0,H_1]_\delta$, we can find $\lambda<\zeta\in C^{(2)}$ and an elementary submodel $Y$ of $V[H_0,H_1]_\zeta$ such that $V[H_0,H_1]_\lambda\cup\{\lambda\}\subseteq Y$ and the transitive collapse of $Y$ is equal to $j(\dot{M})^{H_0}$. Note also that since $\QQQ$ is definable in $V_\delta$ by a formula without parameters, and since $j(\delta)=\delta$, we know that $j(\QQQ)=\QQQ$. By elementarity, and Clause \ref{item:Woodin4-3}, this implies that there is a dense embedding of $\QQQ$ into $\PPP*j(\dot{\RRR})$ in $V$ that sends every condition $q$ in $\PPP$ to $(q,\mathbbm{1}_{j(\dot{\RRR})})$. Hence, there is $F\in V[G]$ that is $j(\dot{\RRR})^{G_0}$-generic over $V[G_0]$ with $V[G]=V[G_0,F]$. 

  Since $p_0\in G_0$ and $j[G_0]\subseteq G_0$, we may now conclude that, in $V[G]_\delta$, there exist a cardinal $\lambda<\zeta\in C^{(2)}$ and an elementary submodel $Y$ of $V[G]_\zeta$ such that $V[G]_\lambda\cup\{\lambda\}\subseteq Y$ and the transitive collapse $\tau$ of $Y$ is an isomorphism onto $j(\dot{M})^{G_0}$. 
  The elementary embedding $j_\ast$, being a lifting of $j\restriction V_{\lambda +1}$ to $V[G_0]_{\lambda+1}$,  now yields that $j_*(M)=j(\dot{M})^{G_0}$, and hence  $$\map{j_*\restriction M}{M}{j(\dot{M})^{G_0}}$$ is an elementary embedding in $V[G]_\delta$. This shows that the composition $$\map{\tau^{{-}1}\circ(j_*\restriction M)\circ\pi}{X}{Y}$$ given by
  \[
\begin{tikzcd}
{V[G]_\eta} & X \arrow[l, "\succ", phantom] \arrow[d, "\pi"] & Y \arrow[r, "\prec", phantom]       & {V[G]_\zeta} \\
            & M \arrow[r, "j_*\upharpoonright M"]                & {j(\dot M)^{G_0}} \arrow[u, "\tau^{-1}"] &             
\end{tikzcd}
\]
is an elementary elementary embedding from $X$ to $V[G]_\zeta$ in $V[G]_\delta$ and, since we know that $\pi\restriction V[G]_\lambda=\id_{V_\lambda}$ and $\tau^{{-}1}\restriction V[G]_\lambda=\id_{V[G]_\lambda}$, we can conclude  that $$(\tau^{{-}1}\circ(j_*\restriction M)\circ\pi)\restriction V[G]_\lambda ~ = ~ j_*\restriction V[G]_\lambda\in X.$$ Since $j_*\restriction V[G]_\lambda$ extends $j\restriction V_\lambda$, this equality  also shows that the constructed elementary embedding extends $j\restriction V_\lambda$.  
%Appealing to Lemma \ref{lemma:EquivalentUltra}, this shows that there is a $2$-ultraexact embedding at $\lambda$ in $V[{G}]_\delta$. 
\end{proof}

 Woodin's theorem and the lemma above  will now yield a proof of the main result of this section:

\begin{proof}[Proof of Theorem \ref{theorem:ConsUltraexactAboveSupercompact}]
 Work in $\mathrm{NBG}-\mathrm{AC}$ and assume that there is a $C^{(3)}$-Reinhardt cardinal $\kappa$, witnessed by $\map{j}{V}{V}$, and there is a supercompact cardinal greater than the supremum, $\lambda$, of the critical sequence, $\seq{\lambda_m} {m<\omega}$, of $j$. Let $\delta$ denote the least supercompact above $\lambda$. Since $\delta$ is definable using the parameter $\lambda$, it follows that $j(\delta)=\delta$.

 Let $\QQQ$ be the partial order given by Theorem \ref{theorem:WoodinForceZFC} and suppose $G$ is $\QQQ$-generic over $V$. Then Clause \ref{item:Woodin3} of Theorem \ref{theorem:WoodinForceZFC} yields that $V[G]_\delta$ is a model of ZFC. Moreover, since $V_\lambda\prec_{\Sigma_3} V$,  the cardinals $\lambda_m$ of the critical sequence of $j$ are extendible cardinals in $V$. Hence, since  Clause \ref{item:Woodin3} of Theorem \ref{theorem:WoodinForceZFC} ensures that extendible cardinals below $\delta$ in $V$ are extendible in $V[G]_\delta$, we have that,  in $V[G]_\delta$, the $\lambda_m$ are extendible cardinals, for all $m<\omega$.  
 %
 % Moreover,  arguing as in the proof of Clause \ref{item:Woodin3} of Theorem \ref{theorem:WoodinForceZFC}, every  restriction embedding $\map{j\restriction V_{\lambda_{m+1}}}{V_{\lambda_{m+1}}}{V_{\lambda_{m+2}}}$ lifts to an elementary embedding $$\map{j_m}{V[G]_{\lambda_{m+1}}}{V[G]_{\lambda_{m+2}}}.$$ Now, in $V[G]_\delta$, for every $m<\omega$, we can define $$\mathcal{U}_m ~ = ~ \Set{ X\subseteq \lambda_m}{j_m[\lambda_m]\in j_m(X)}$$ and   easily verify that $\mathcal{U}_m$ is an ultrafilter in $V[G]_\delta$. So, since $V[G]_\delta$ is a model of ZFC,  we can apply the ultrafilter characterization of $m$-hugeness, as in {\cite[Theorem 24.8]{Kan:THI}}, to conclude that $\lambda_0$ is $m$-huge for all $0<m<\omega$. Similarly, one can show that $\lambda_k$ is $m$-huge, for all $0<k,m <\omega$. 
 In addition, Lemma \ref{lemma:WoddinsForcingUltraexact} shows that, in $V[G]_\delta$, there is a $1$-ultraexact embedding at $\lambda$ and hence Lemma \ref{lemma:EquivalentUltra} shows that $\lambda$ is ultraexacting in $V[G]_\delta$. Therefore, the set $V[G]_\delta$ is a model of the theory in the statement of Theorem \ref{theorem:ConsUltraexactAboveSupercompact}. By absoluteness, this theory is also consistent in the ground model $V$, yielding the statement of the theorem.  
\end{proof}

%%%%%%

\subsection{A failure of the  HOD Conjecture}\label{sect6}

We can draw the following conclusions concerning the interaction between exacting cardinals and extendible cardinals. We refer the reader to  \cite[Definition 7.13]{MR4022642} for more on the Ultimate-L Conjecture.

\begin{conclusion}\label{coroWHODConj}
  Let $T$ denote the theory consisting  of the axioms of ZFC together with the statement that there is an exacting cardinal above an extendible cardinal. 
  %$\mathrm{ZFC}$ + there is  $2$-ultraexact embedding at some $\lambda$ and an extendible cardinal $\delta<\lambda$. Then,
 \begin{enumerate}
  \item\label{coroWHODConj1} The theory $T$ proves that there is a huge cardinal above an extendible cardinal and  all sufficiently large regular cardinals are $\omega$-strongly measurable in $\HOD$. 
  %then the interval $(\delta,\lambda)$ contains a huge cardinal and every regular cardinal greater than or equal to $\delta$ is $\omega$-strongly measurable in $\HOD$.%, {i.e.,} the $\HOD$ Hypothesis fails. 

  \item\label{coroWHODConj2}  $\mathrm{PA} + \mathrm{con}(T)$ disproves the Weak $\HOD$ Conjecture and the Ultimate-$L$ Conjecture.

 \item\label{coroWHODConj3} The theory $\mathrm{NBG}-\mathrm{AC}$ + ``there is a $C^{(3)}$-Reinhardt cardinal and a supercompact cardinal greater than the supremum of the critical sequence'' proves $\mathrm{con}(T)$. 
 %\jpa{I wonder if this is redundant since the way it is written now it is just repeating Theorem \ref{theorem:ConsUltraexactAboveSupercompact}.}
\end{enumerate}
\end{conclusion}

\begin{proof}
    By Theorem \ref{VnotHOD}, the theory $T$ proves that some singular cardinal that is greater than an extendible cardinal  is regular in $\HOD$.  
    According to Woodin's $\HOD$-dichotomy theorem (see {\cite[Theorems 3.34 \& 3.39]{MR3632568}}), this implies that every regular cardinal greater or equal to the given extendible cardinal is $\omega$-strongly measurable in $\HOD$\footnote{Note that, by a result of Goldberg {\cite[Theorem 2.10]{MR4693981}}, this conclusion only requires a strongly compact cardinal}. Since we also know that there is an exacting cardinal $\lambda$ above the given extendible cardinal, it follows that there exists an I3-embedding $\map{j}{V_\lambda}{V_\lambda}$ and therefore $\lambda$ is a limit of huge cardinals. This proves \ref{coroWHODConj1} and \ref{coroWHODConj2} (recall that the Weak Ultimate-L Conjecture implies the Weak HOD Conjecture by \cite{MR4022642}). Finally,  \ref{coroWHODConj3} follows directly from Theorem \ref{theorem:ConsUltraexactAboveSupercompact}.
    %
    %Now, $V_\lambda$ satisfies ZFC together with ``there exists an extendible cardinal below a huge cardinal.'' 
    %The extendibility of $\delta$ then ensures that there is an inaccessible cardinal $\kappa>\lambda$ and a supercompact cardinal $\gamma<\delta$. Then $\gamma$ is supercompact in $V_\kappa$. 
\end{proof}

\subsection{Concluding remarks}
We would like to end the current discussion with some remarks.
\subsubsection{}\label{remark1}
In \cite{MR4022642}, it is shown that the Weak $\HOD$ Conjecture fails under the assumption of the consistency of $\mathrm{NBG}-\mathrm{AC}$ with the existence of a Reinhardt cardinal  and an $\mathrm{I3}$-embedding above the supremum, $\lambda$, of the critical sequence of $j$, {i.e.,}  a non-trivial elementary embedding $\map{i}{V_\mu}{V_\mu}$, for some limit $\mu$ (see \cite[Theorem  8.4]{MR4022642}) having critical point greater than $\lambda$.  
While the theory employed in Theorem \ref{theorem:ConsUltraexactAboveSupercompact} is stronger than this, 
the novelty of Conclusion \ref{coroWHODConj} is that the Weak HOD Conjecture can be refuted from a combination of natural large cardinal principles compatible with the Axiom of Choice, moreover the individual consistency strength of each of these principles is weaker than the existence of an I0-embedding (by Theorem \ref{theorem:ExactFromI0}). It seems unlikely that the hypothesis of Theorem \ref{theorem:ConsUltraexactAboveSupercompact} can be weakened substantially, thus we conjecture:

\begin{conjecture}
Suppose that $\mathsf{ZFC}$ holds, $\lambda$ is ultraexacting and $\kappa<\lambda$ is extendible. Then, there is a set model of $\mathsf{ZF}$ with a rank-Berkeley cardinal.
\end{conjecture}

\subsubsection{}
After an earlier draft of this article had been circulated and following discussions with Woodin, the authors have learned that the hypothesis of \cite[Theorem 200]{WOEM1} possesses several of the key features of ultraexacting cardinals: its consistency with ZFC follows from the consistency of ZFC with I0, it implies that $V$ is not equal to HOD, and the existence of such a cardinal above an extendible cardinal implies a failure of the HOD Hypothesis. This hypothesis is a variation of \textit{Laver's axiom} (see \cite[Definition 1]{MR2914848}) asserting the existence of an elementary embedding 
\[\map{j}{(V_{\lambda+1},T)}{ (V_{\lambda+1},T)},\]
where $T$ is the $\Sigma_2$-theory of $V$ with parameters in $V_{\lambda+1}$.
A follow-up article joint with Goldberg \cite{ABGL25}  contains  an argument showing that such an embedding exists if and only if $\lambda$ is ultraexacting, yielding yet another characterization of this notion. 

In contrast, the notion of an exacting cardinal appears not to be equivalent to any principle studied previously, despite its natural definition, and it also suffices to establish Conclusion \ref{coroWHODConj}. It is also strictly weaker in terms of consistency strength. Indeed, by Theorem \ref{TheoremIntroConUE}, they are \textit{strictly} weaker than I0. Refinements of the arguments in \S\ref{SectConsistencyExact} can be used to weaken this hypothesis further, below I2 and within the iterability hierarchy of \cite{AnDi19} (see \cite{ABGL25}).

\subsubsection{} The authors do not believe that the work presented here bears any negative prognosis whatsoever on the Inner Model Program or even on the Ultimate-L \textit{Program}, despite the connection between exacting and ultraexacting cardinals and the Ultimate-L \textit{Conjecture}. Woodin has suggested a \textit{Revised HOD Conjecture}: Suppose that there is an extendible cardinal. Then either the HOD Hypothesis holds or there is no supercompact cardinal in HOD.

\subsubsection{} We observe that consistent large cardinals compatible with the Axiom of Choice can imply that $V$ is not equal to HOD. Moreover, one of two things must happen: if exacting cardinals are consistent together with extendible cardinals below them, then the HOD Conjecture must fail, and thus the set-theoretic universe V is dramatically different from the sub-universe HOD of definable sets; if not, one must accept the existence of consistent large cardinal notions which are mutually incompatible, and thus the divergence of the large-cardinal hierarchy. In either case, the authors believe that some of the commonly held beliefs concerning strong axioms of infinity have to be revised.

%%%%%%%%%%%%%%%%%%%
%%%%%%%%%%%%%%%%%%%

\bibliographystyle{amsplain} 
\bibliography{masterbiblio}

\end{document}